\documentclass[english,linesnumbered,ruled,vlined]{article}
\usepackage[a4paper]{geometry}
\usepackage{babel}
\usepackage{verbatim}
\usepackage{mathtools}
\usepackage{booktabs}
\usepackage{enumitem}
\usepackage{bm}
\usepackage{algorithm2e}
\usepackage{amsmath}
\usepackage{amsthm}
\usepackage{amssymb}
\usepackage{minitoc}
\usepackage{color,xcolor}
\usepackage[numbers]{natbib}
\usepackage[pdfusetitle,
 bookmarks=true,bookmarksnumbered=false,bookmarksopen=false,
 breaklinks=false,pdfborder={0 0 1},backref=false,colorlinks=false]{hyperref}

\usepackage{cleveref}

\makeatletter
\theoremstyle{definition}
\newtheorem{defn}{\protect\definitionname}[section]
\theoremstyle{plain}
\newtheorem{thm}{\protect\theoremname}[section]
\theoremstyle{plain}
\newtheorem{prop}{\protect\propositionname}[section]
\theoremstyle{plain}
\newtheorem{assumption}{\protect\assumptionname}
\theoremstyle{plain}
\newtheorem{lem}{\protect\lemmaname}[section]
\theoremstyle{remark}
\newtheorem{rem}{\protect\remarkname}[section]
\theoremstyle{definition}
\newtheorem{example}{\protect\examplename}[section]
\providecommand{\examplename}{Example}
\theoremstyle{corollary}
\newtheorem{coro}{\protect\corollaryname}[section]

\newcommand{\assign}{\coloneqq}

\newcommand{\tmop}[1]{\ensuremath{\operatorname{#1}}}
\newcommand{\tmtextbf}[1]{\text{{\bfseries{#1}}}}

\newcommand{\phietaenvk}{{\hat{\phi}}_{k}}
\newcommand{\hatx}{\hat{\xbf}}
\newcommand{\x}{\mathbf{x}}
\newcommand{\tma}{\mathbf{a}}
\newcommand{\y}{\mathbf{y}}
\newcommand{\qbf}{\mathbf{q}}

\newcommand{\agls}{\texttt{AGLS}}

\newcommand{\sspg}{\texttt{SSPG}}
\newcommand{\gm}{\texttt{GM}}

\newcommand{\proxacc}{\texttt{AGD-SIPP}}
\newcommand{\sgm}{\texttt{SGM}}

\newcommand{\proxaccvr}{\texttt{ASGD-SIPP}}
\newcommand{\asgdsipp}{\texttt{ASGD-SIPP}}
\newcommand{\aglssipp}{\texttt{AGLS-SIPP}}

\providecommand{\assumptionname}{Assumption}
\providecommand{\definitionname}{Definition}
\providecommand{\lemmaname}{Lemma}
\providecommand{\propositionname}{Proposition}
\providecommand{\remarkname}{Remark}
\providecommand{\theoremname}{Theorem}
\providecommand{\corollaryname}{Corollary}

\crefdefaultlabelformat{#2\textbf{#1}#3} %
\crefname{section}{\textbf{section}}{\textbf{sections}}
\Crefname{section}{\textbf{Section}}{\textbf{Sections}}
\crefname{thm}{\textbf{Theorem}}{\textbf{theorems}}
\Crefname{thm}{\textbf{Theorem}}{\textbf{Theorems}}
\crefname{lem}{\textbf{Lemma}}{\textbf{lemmas}}
\Crefname{lem}{\textbf{Lemma}}{\textbf{Lemmas}}
\crefname{prop}{\textbf{proposition}}{\textbf{propositions}}
\Crefname{prop}{\textbf{Proposition}}{\textbf{Propositions}}
\crefname{algorithm}{\textbf{algorithm}}{\textbf{algorithms}}
\Crefname{algorithm}{\textbf{Algorithm}}{\textbf{Algorithms}}
\crefname{coro}{\textbf{Corollary}}{\textbf{corollaries}}
\Crefname{coro}{\textbf{Corollary}}{\textbf{corollaries}}
\crefname{defn}{\textbf{Definition}}{\textbf{definitions}}
\Crefname{defn}{\textbf{Definition}}{\textbf{definitions}}
\crefname{table}{\textbf{Table}}{\textbf{tables}}
\Crefname{table}{\textbf{Table}}{\textbf{tables}}
\crefname{figure}{\textbf{Figure}}{\textbf{figures}}
\Crefname{figure}{\textbf{Figure}}{\textbf{figures}}
\crefname{exple}{\textbf{Example}}{\textbf{examples}}
\Crefname{exple}{\textbf{Example}}{\textbf{examples}}
\Crefname{assumption}{\textbf{Assumption}}{\textbf{Assumptions}}
\crefname{assumption}{\textbf{Assumption}}{\textbf{Assumptions}}
\Crefname{rem}{\textbf{Remark}}{\textbf{Remarks}}
\crefname{rem}{\textbf{Remark}}{\textbf{Remarks}}

\newif\ifcomments
\commentstrue 
\ifcomments
  \newcommand{\comm}[2][]{\textcolor{red}{\textbf{[#1:} #2\textbf{]}}}
  \newcommand{\qd}[1]{\comm[QD]{#1}}
  \newcommand{\zl}[1]{\comm[ZL]{#1}}
\else
  \newcommand{\comm}[2][]{}
  \newcommand{\qd}[1]{}
  \newcommand{\zl}[1]{}
  \newcommand{\xx}[1]{}
\fi

\providecommand{\corollaryname}{Corollary}

\makeatother

\begin{document}
\global\long\def\inprod#1#2{\left\langle #1,#2\right\rangle }%
\global\long\def\inner#1#2{\langle#1,#2\rangle}%
\global\long\def\binner#1#2{\big\langle#1,#2\big\rangle}%
\global\long\def\Binner#1#2{\Big\langle#1,#2\Big\rangle}%

\global\long\def\abs#1{|#1|}%

\global\long\def\norm#1{\lVert#1\rVert}%
\global\long\def\bnorm#1{\big\Vert#1\big\Vert}%
\global\long\def\Bnorm#1{\Big\Vert#1\Big\Vert}%

\global\long\def\setnorm#1{\Vert#1\Vert_{-}}%
\global\long\def\bsetnorm#1{\big\Vert#1\big\Vert_{-}}%
\global\long\def\Bsetnorm#1{\Big\Vert#1\Big\Vert_{-}}%

\global\long\def\brbra#1{\big(#1\big)}%
\global\long\def\Brbra#1{\Big(#1\Big)}%
\global\long\def\rbra#1{(#1)}%
\global\long\def\sbra#1{[#1]}%
\global\long\def\bsbra#1{\big[#1\big]}%
\global\long\def\Bsbra#1{\Big[#1\Big]}%
\global\long\def\cbra#1{\{#1\}}%
\global\long\def\bcbra#1{\big\{#1\big\}}%
\global\long\def\Bcbra#1{\Big\{#1\Big\}}%

\global\long\def\vertiii#1{\left\vert \kern-0.25ex  \left\vert \kern-0.25ex  \left\vert #1\right\vert \kern-0.25ex  \right\vert \kern-0.25ex  \right\vert }%

\global\long\def\matr#1{\bm{#1}}%

\global\long\def\til#1{\tilde{#1}}%
\global\long\def\wtil#1{\widetilde{#1}}%

\global\long\def\wh#1{\widehat{#1}}%

\global\long\def\mcal#1{\mathcal{#1}}%

\global\long\def\mbb#1{\mathbb{#1}}%

\global\long\def\mtt#1{\mathtt{#1}}%

\global\long\def\ttt#1{\texttt{#1}}%

\global\long\def\dtxt{\textrm{d}}%

\global\long\def\bignorm#1{\bigl\Vert#1\bigr\Vert}%
\global\long\def\Bignorm#1{\Bigl\Vert#1\Bigr\Vert}%

\global\long\def\rmn#1#2{\mathbb{R}^{#1\times#2}}%

\global\long\def\deri#1#2{\frac{d#1}{d#2}}%
\global\long\def\pderi#1#2{\frac{\partial#1}{\partial#2}}%

\global\long\def\limk{\lim_{k\rightarrow\infty}}%

\global\long\def\smid{\mskip1mu\mid\mskip1mu}%

\global\long\def\trans{\textrm{T}}%

\global\long\def\onebf{\mathbf{1}}%
\global\long\def\zerobf{\mathbf{0}}%
\global\long\def\zero{\mathbf{0}}%

\global\long\def\Euc{\mathrm{E}}%
\global\long\def\Expe{\mathbb{E}}%

\global\long\def\rank{\mathrm{rank}}%
\global\long\def\range{\mathrm{range}}%
\global\long\def\diam{\mathrm{diam}}%
\global\long\def\epi{\mathrm{epi} }%
\global\long\def\relint{\mathrm{relint} }%
\global\long\def\dom{\mathrm{dom}}%
\global\long\def\prox{\mathrm{prox}}%
\global\long\def\proj{\mathrm{Proj}}%
\global\long\def\for{\mathrm{for}}%
\global\long\def\diag{\mathrm{diag}}%
\global\long\def\and{\mathrm{and}}%
\global\long\def\var{\textrm{VaR}}%
\global\long\def\where{\mathrm{where}}%
\global\long\def\dist{\mathrm{dist}}%
\global\long\def\textop{\mathrm{op}}%
\global\long\def\level{\mathrm{lev}}%

\global\long\def\st{\mathrm{s.t.}}%

\global\long\def\inte{\operatornamewithlimits{int}}%
\global\long\def\cov{\mathrm{Cov}}%

\global\long\def\argmin{\operatornamewithlimits{arg\,min}}%
\global\long\def\argmax{\operatornamewithlimits{arg\,max}}%
\global\long\def\maximize{\operatornamewithlimits{maximize}}%
\global\long\def\minimize{\operatornamewithlimits{minimize}}%

\global\long\def\tr{\operatornamewithlimits{tr}}%

\global\long\def\dis{\operatornamewithlimits{dist}}%

\global\long\def\prob{{\rm Pr}}%

\global\long\def\spans{\textrm{span}}%
\global\long\def\st{\operatornamewithlimits{s.t.}}%
\global\long\def\subjectto{\textrm{subject to}}%

\global\long\def\Var{\operatornamewithlimits{Var}}%

\global\long\def\raw{\rightarrow}%
\global\long\def\law{\leftarrow}%
\global\long\def\Raw{\Rightarrow}%
\global\long\def\Law{\Leftarrow}%

\global\long\def\vep{\varepsilon}%

\global\long\def\dom{\operatornamewithlimits{dom}}%

\begingroup
\expandafter\ifx\csname oldsum\endcsname\relax
  \let\oldsum\sum
\fi
\endgroup
\global\long\def\tsum{{\textstyle {\oldsum}}}%

\global\long\def\Cbb{\mathbb{C}}%
\global\long\def\Ebb{\mathbb{E}}%
\global\long\def\Fbb{\mathbb{F}}%
\global\long\def\Nbb{\mathbb{N}}%
\global\long\def\Rbb{\mathbb{R}}%
\global\long\def\Vbb{\mathbb{V}}%
\global\long\def\reals{\mathbb{R}}%

\global\long\def\extR{\widebar{\mathbb{R}}}%
\global\long\def\Pbb{\mathbb{P}}%

\global\long\def\Mrm{\mathrm{M}}%
\global\long\def\Acal{\mathcal{A}}%
\global\long\def\Bcal{\mathcal{B}}%
\global\long\def\Ccal{\mathcal{C}}%
\global\long\def\Dcal{\mathcal{D}}%
\global\long\def\Ecal{\mathcal{E}}%
\global\long\def\Fcal{\mathcal{F}}%
\global\long\def\Gcal{\mathcal{G}}%
\global\long\def\Hcal{\mathcal{H}}%
\global\long\def\Ical{\mathcal{I}}%
\global\long\def\Kcal{\mathcal{K}}%
\global\long\def\Lcal{\mathcal{L}}%
\global\long\def\Mcal{\mathcal{M}}%
\global\long\def\Ncal{\mathcal{N}}%
\global\long\def\Ocal{\mathcal{O}}%
\global\long\def\Pcal{\mathcal{P}}%
\global\long\def\Scal{\mathcal{S}}%
\global\long\def\Rcal{\mathcal{R}}%
\global\long\def\Tcal{\mathcal{T}}%
\global\long\def\Wcal{\mathcal{W}}%
\global\long\def\Xcal{\mathcal{X}}%
\global\long\def\Ycal{\mathcal{Y}}%
\global\long\def\Zcal{\mathcal{Z}}%

\global\long\def\abf{\mathbf{a}}%
\global\long\def\bbf{\mathbf{b}}%
\global\long\def\cbf{\mathbf{c}}%
\global\long\def\dbf{\mathbf{d}}%
\global\long\def\ebf{\mathbf{e}}%
\global\long\def\fbf{\mathbf{f}}%
\global\long\def\gbf{\mathbf{g}}%
\global\long\def\pbf{\mathbf{p}}%
\global\long\def\sbf{\mathbf{s}}%
\global\long\def\lbf{\mathbf{l}}%
\global\long\def\ubf{\mathbf{u}}%
\global\long\def\vbf{\mathbf{v}}%
\global\long\def\wbf{\mathbf{w}}%
\global\long\def\xbf{\mathbf{x}}%
\global\long\def\ybf{\mathbf{y}}%
\global\long\def\zbf{\mathbf{z}}%

\global\long\def\Abf{\mathbf{A}}%
\global\long\def\Ubf{\mathbf{U}}%
\global\long\def\Pbf{\mathbf{P}}%
\global\long\def\Ibf{\mathbf{I}}%
\global\long\def\Ebf{\mathbf{E}}%
\global\long\def\Mbf{\mathbf{M}}%
\global\long\def\Qbf{\mathbf{Q}}%
\global\long\def\Lbf{\mathbf{L}}%
\global\long\def\Pbf{\mathbf{P}}%
\global\long\def\Wbf{\mathbf{W}}%

\global\long\def\lambf{\bm{\lambda}}%
\global\long\def\mubf{\bm{\mu}}%
\global\long\def\alphabf{\bm{\alpha}}%
\global\long\def\sigmabf{\bm{\sigma}}%
\global\long\def\thetabf{\bm{\theta}}%
\global\long\def\deltabf{\bm{\delta}}%
\global\long\def\vepbf{\bm{\vep}}%
\global\long\def\pibf{\bm{\pi}}%
\global\long\def\phibf{\bm{\phi}}%

\global\long\def\abm{\bm{a}}%
\global\long\def\bbm{\bm{b}}%
\global\long\def\cbm{\bm{c}}%
\global\long\def\dbm{\bm{d}}%
\global\long\def\ebm{\bm{e}}%
\global\long\def\fbm{\bm{f}}%
\global\long\def\gbm{\bm{g}}%
\global\long\def\hbm{\bm{h}}%
\global\long\def\pbm{\bm{p}}%
\global\long\def\qbm{\bm{q}}%
\global\long\def\rbm{\bm{r}}%
\global\long\def\sbm{\bm{s}}%
\global\long\def\tbm{\bm{t}}%
\global\long\def\ubm{\bm{u}}%
\global\long\def\vbm{\bm{v}}%
\global\long\def\wbm{\bm{w}}%
\global\long\def\xbm{\bm{x}}%
\global\long\def\ybm{\bm{y}}%
\global\long\def\zbm{\bm{z}}%

\global\long\def\Abm{\bm{A}}%
\global\long\def\Bbm{\bm{B}}%
\global\long\def\Cbm{\bm{C}}%
\global\long\def\Dbm{\bm{D}}%
\global\long\def\Ebm{\bm{E}}%
\global\long\def\Fbm{\bm{F}}%
\global\long\def\Gbm{\bm{G}}%
\global\long\def\Hbm{\bm{H}}%
\global\long\def\Ibm{\bm{I}}%
\global\long\def\Jbm{\bm{J}}%
\global\long\def\Lbm{\bm{L}}%
\global\long\def\Obm{\bm{O}}%
\global\long\def\Pbm{\bm{P}}%
\global\long\def\Qbm{\bm{Q}}%
\global\long\def\Rbm{\bm{R}}%
\global\long\def\Sbm{\bm{S}}%
\global\long\def\Ubm{\bm{U}}%
\global\long\def\Vbm{\bm{V}}%
\global\long\def\Wbm{\bm{W}}%
\global\long\def\Xbm{\bm{X}}%
\global\long\def\Ybm{\bm{Y}}%
\global\long\def\Zbm{\bm{Z}}%
\global\long\def\lambm{\bm{\lambda}}%

\global\long\def\alphabm{\bm{\alpha}}%
\global\long\def\albm{\bm{\alpha}}%
\global\long\def\pibm{\bm{\pi}}%
\global\long\def\taubm{\bm{\tau}}%
\global\long\def\mubm{\bm{\mu}}%
\global\long\def\yrm{\mathrm{y}}%

\global\long\def\ssgd{\texttt{SSGD}}%
\global\long\def\sipp{\texttt{SIPP}}%

\title{\textbf{A Smooth Approximation Framework for Weakly Convex Optimization}}

\author{Qi Deng\thanks{qdeng24@sjtu.edu.cn, Shanghai Jiao Tong University} \qquad\qquad\quad Wenzhi Gao\thanks{gwz@stanford.edu, Stanford University}  }

\maketitle

\begin{abstract}
Standard complexity analyses for weakly convex optimization rely on the Moreau envelope technique proposed by Davis and Drusvyatskiy (2019). The main insight is that nonsmooth algorithms, such as proximal subgradient, proximal point, and their stochastic variants, implicitly minimize a smooth surrogate function induced by the Moreau envelope. Meanwhile, explicit smoothing, which directly minimizes a smooth approximation of the objective, has long been recognized as an efficient strategy for nonsmooth optimization. 
In this paper, we generalize the notion of smoothable functions, 
which was proposed by Beck and Teboulle (2012) for nonsmooth convex
optimization. This generalization provides a unified viewpoint on several important
smoothing techniques for weakly convex optimization, including Nesterov-type smoothing
and Moreau envelope smoothing. Our theory yields a framework for designing
smooth approximation algorithms for both deterministic and stochastic weakly convex
problems with provable complexity guarantees. 
Furthermore, our theory extends to the smooth approximation of non-Lipschitz functions, allowing for complexity analysis even when global Lipschitz
continuity does not hold.

\end{abstract}

\section{Introduction}

Our primary problem of interest is represented as follows: 
\begin{equation}
\min_{\xbf\in\Rbb^{d}}\ \phi(\xbf) = f(\xbf) + r(\xbf). \label{pb:main}
\end{equation}
Here, $f$ is a nonsmooth nonconvex function, and $r$ is a convex, lower-bounded, and lower semi-continuous function. We assume that $f$ is $\rho$-weakly convex for some $\rho > 0$; that is, the function $f(\xbf) + \frac{\rho}{2}\|\xbf\|^2$ is convex. We also assume that $r$ is prox-friendly, meaning its associated proximal operator can be computed efficiently. Weakly convex optimization problems frequently arise in various data-driven applications, including signal processing~\citep{duchi2019solving}, machine learning~\citep{charisopoulos2021low}, and reinforcement learning~\citep{wang2023policy}.

The nonsmoothness of $f$ prevents the direct application of classical complexity results for smooth nonconvex optimization~\citep{ghadimi2016mini, saeed-lan-nonconvex-2013}.  Early progress in weakly convex optimization was achieved
using double-loop proximal algorithms that require solving auxiliary subproblems at
each iteration~\citep{damek19Proximally, drusvyatskiy2019efficiency}. A major
breakthrough came from \citet{davis2019stochastic}, who proved that single-loop
methods, such as the (stochastic) subgradient method, achieve an
$\Ocal(1/\varepsilon^4)$ complexity bound for finding an approximate stationary point.
Their work reveals an implicit smoothing effect in nonsmooth optimization: algorithms like the proximal subgradient method can be interpreted as optimizing the Moreau envelope $\phi^{\hat\rho}(\xbf)$ ($\hat\rho>\rho$), a smooth approximation of $\phi$. 
Notably, $\phi$ and its Moreau envelope $\phi^{\hat\rho}$
share the same stationary points, and the gradient of the envelope provides a natural
measure of approximate stationarity for~\eqref{pb:main}. This insight has established
the Moreau envelope as a central tool in the algorithm analysis for weakly convex problems, influencing a wide range of subsequent work
\citep{mai2020convergence, deng2021minibatch, gao24-stochastic, chen2021distributed,
li2024revisiting, zhu2023unified}.

An alternative to implicit smoothing is \emph{explicit smoothing}, which replaces
\eqref{pb:main} with a smooth optimization problem:
\begin{equation}\label{pb:comp-smooth}
    \min_{\xbf\in\Rbb^{d}}\, \phi_{\eta}(\xbf) = f_{\eta}(\xbf) + r(\xbf).
\end{equation}
Here, $f_\eta:\Rbb^d \to \Rbb$ is a smooth surrogate of $f$, parameterized by a
smoothing parameter $\eta > 0$ that controls the approximation accuracy.

Explicit smoothing has a long history~\citep{chen2012smoothing}. In particular,
Nesterov's pioneering work~\citep{nesterov2005smooth} showed that applying accelerated
methods to a suitably smoothed approximation of a structured convex problem improves
the complexity of finding an $\varepsilon$-optimal solution from $\Ocal(1/\varepsilon^2)$
to $\Ocal(1/\varepsilon)$. Building on this idea,
\citet{Beck2012smoothing} developed a unified smoothing framework that integrates
multiple smoothing schemes with accelerated gradient methods. However, both \cite{nesterov2005smooth} and \cite{Beck2012smoothing} restrict their
analyses to convex problems. While earlier works have examined the asymptotic
behavior of smoothing in nonconvex settings~\citep{chen2012smoothing, xu2009smooth}, a rigorous complexity theory for explicit smoothing in the weakly convex regime has remained underdeveloped.

\subsection{Contributions}
To address this gap, we extend the smooth approximation framework of
\citet{Beck2012smoothing} from the convex setting to the broader class of weakly
convex problems. In particular, our framework seamlessly incorporates two prominent
smoothing techniques as special cases:  1) Moreau envelope smoothing, and 2) Nesterov-type smoothing for compositions of convex functions with smooth mappings. A key advantage of this generalization arises in stochastic optimization, where one can
smooth the function associated with each stochastic sample individually and then
perform minibatching. This strategy can provide computational benefits over
model-based minibatching algorithms~\citep{davis2019stochastic, deng2021minibatch},
which require solving a potentially expensive proximal subproblem at every iteration. 
Another important consequence is that Moreau envelope smoothing becomes substantially
more practical: it can now be used not only for convergence
analysis~\citep{davis2019stochastic}, but also as an effective and scalable algorithmic tool.

It is important to emphasize that the smooth approximation we propose, akin to that of \citet{Beck2012smoothing}, describes a broad class of surrogate functions
and is not tied to any particular smoothing technique. 
We establish quantitative
relationships connecting the stationarity measures of the smoothed problem
\eqref{pb:comp-smooth} to those of the original nonsmooth problem
\eqref{pb:main}. A key implication is that the design of the smooth surrogate can be
decoupled from the design of the optimization algorithm. This stands in contrast to
many prior smoothing strategies for weakly convex optimization, which often rely on
exploiting structural properties of the underlying problem
\citep{bohm2021variable,drusvyatskiy2019efficiency}.

The effectiveness of our approach is primarily governed by the choice of the smoothing
parameter $\eta$, which introduces a fundamental trade-off: a smaller $\eta$ enhances
the approximation fidelity to the original objective but worsens the conditioning
(i.e., increases the gradient Lipschitz constant $L_\eta$) of $f_\eta$; conversely, a
larger $\eta$ improves conditioning at the cost of reduced approximation accuracy.
Consequently, simply applying standard gradient-based algorithms to the smoothed problem does
not automatically yield better complexity guarantees.
To improve the convergence result, we exploit the asymmetry in the curvature bounds of $f_\eta$: while the upper curvature $L_\eta$ may become arbitrarily large as $\eta \rightarrow 0^+$, the lower curvature, determined by the weak convexity modulus of $f_\eta$, remains bounded below. To take advantage of this structure, we propose solving the smoothed problem using an inexact proximal point (IPP) framework. By applying accelerated (and stochastic) algorithms to the subproblems of the proximal point method, we can mitigate the adverse effects of the ill conditioning induced by small $\eta$.
This approach yields an overall complexity of $\Ocal(1/\vep^3)$ in the deterministic setting. For general stochastic optimization, we obtain an $\Ocal(\max \{1/\varepsilon^3, 1/(m\varepsilon^4)\})$ complexity when using a minibatch size of $m$.
We summarize our complexity results and compare them with those of the standard proximal subgradient method in \Cref{tab:complexity}.

\begin{table}[h]
\centering
\caption{Complexity guarantees for achieving $\varepsilon$-approximate stationarity; $m$: size of minibatch.}
\label{tab:complexity}
\begin{tabular}{ccc}
\toprule
\textsf{Methods} & \textsf{Deterministic} & \textsf{Stochastic} \\
\midrule
\textsf{Subgradient-based method} & $\Ocal(\frac{1}{\vep^4}) $ & $\Ocal(\frac{1}{\vep^4}) $ \\
\textsf{IPP-based smoothing approach} & $\Ocal(\frac{1}{\vep^3}) $ & $\Ocal\left(\max\{\frac{1}{\vep^3}, \frac{1}{m\vep^4}\}\right)$ \\
\bottomrule
\end{tabular}
\end{table}

A further contribution of our work is the relaxation of the global Lipschitz
smoothness assumption~\citep{Beck2012smoothing} in the design of smooth approximation
functions. As a result, we can handle objectives that are not globally Lipschitz,
thereby substantially broadening the class of problems to which our framework applies.
For example, in real phase retrieval, the loss
$\ell(\xbf) = |\langle \abf, \xbf \rangle ^{2}- b|$ has an unbounded subgradient in $\Rbb^d$, and
standard Nesterov-type smoothing produces approximations that fail to satisfy global Lipschitz smoothness.
To overcome this limitation, we employ a line search-based accelerated gradient
method for convex optimization problems without global Lipschitz smoothness. The use
of line search ensures that accelerated convergence rates for convex problems remain
attainable even when global smoothness is absent. Under mild assumptions,
we show that the inexact proximal point method, equipped with this accelerated solver
for its subproblems, still achieves an $\Ocal(1/\varepsilon^{3})$ complexity
guarantee. This result markedly improves upon the previous $\Ocal(1/\vep^4)$ complexity achieved by non-Lipschitz subgradient methods~\citep{grimmer2019convergence, zhu2023unified}. Finally, our theoretical findings are supported by numerical experiments. Our results
demonstrate that smoothing leads to smoother convergence behavior and outperforms
standard subgradient methods on robust nonlinear regression problems.

\subsection{Related works}

The $\Ocal(1/\vep^4)$ complexity is well established as the standard result for general weakly convex optimization~\citep{davis2019stochastic}. The central analytical tool in such results is typically the Moreau envelope, which serves as a potential function in the convergence analysis~\citep{davis2019stochastic}. To obtain better complexity guarantees, previous works have leveraged additional structural properties of the objective through smoothing techniques. To the best of our knowledge, all existing smoothing-based approaches for weakly convex functions impose explicit structural assumptions on the objective. For instance, \citet{drusvyatskiy2019efficiency} studied composite functions of the form $f(\xbf) = h(F(\xbf)) + r(\xbf)$, where $h$ is convex and $F$ is a smooth map. They applied the prox-linear method to this composite setup, solving each subproblem inexactly via accelerated gradient methods and achieving an improved complexity of $\Ocal(1/\vep^3)$. 
Similarly, \citet{bohm2021variable,peng2023riemannian} developed a variable smoothing technique for problems of the form $f(\xbf) = g(A\xbf)$, where $g$ is weakly convex and $A$ is a linear operator. 
However, their analysis relies on the surjectivity of $A$, an assumption that our framework does not require. Notably, these previous smoothing techniques arise as special instances within our broader framework. Our proposed approach is more general, accommodating a wider class of weakly convex objectives without imposing restrictive structural assumptions on either the objective function or the smoothing procedure.

The relaxation of global Lipschitz continuity assumptions in weakly convex optimization has received significant attention in recent years. \citet{mai2021stability} pioneered this direction by analyzing the stability and convergence of stochastic gradient clipping methods beyond Lipschitz continuity and smoothness assumptions. \citet{li2024revisiting} extended the classical subgradient method to handle non-Lipschitz convex and weakly convex functions, establishing a complexity of $\Ocal(1/\vep^2)$ for convex objectives and $\Ocal(1/\vep^4)$ for weakly convex objectives, without requiring modifications to the algorithm or imposing additional assumptions on the subgradients. In the stochastic setting, \citet{gao2024stochastic} developed  adaptive regularization strategies that maintain the $\Ocal(1/\vep^4)$ complexity for stochastic weakly convex optimization, allowing the Lipschitz parameter to be either a general function of the iterate norm or estimated locally through random samples. \citet{zhu2023unified} provided a unified analysis framework of subgradient methods for minimizing composite nonconvex, nonsmooth, and non-Lipschitz functions, establishing convergence guarantees under diminishing step sizes. 

The use of smoothing techniques extends far beyond weakly convex functions and has a rich and influential history (see, e.g.,~\citep{zang1980smoothing,xu2009smooth,chen2004smoothing,chen2012smoothing,lin2022gradient,kume2024variable}). Notably, \citet{xu2009smooth} introduced a smooth sample average approximation (SAA) framework for nonsmooth stochastic optimization, demonstrating that stationary points of the smoothed SAA converge, in a suitable sense, to stationary points of the original nonsmooth problem, and illustrating their results through several concrete applications. \citet{chen2012smoothing} established a general smoothing theory based on gradient consistency, which unifies many classical smoothing methods and ensures the asymptotic convergence of gradient-type algorithms to Clarke stationary points under mild assumptions. In contrast to these works, the present paper addresses weakly convex (and potentially non-Lipschitz) composite optimization problems and seeks to develop an explicit complexity theory. More recently, \citet{lin2022gradient} proposed a gradient-free approach to nonsmooth, nonconvex optimization using zeroth-order randomized smoothing. It is important to note that their analysis pertains to more general Lipschitz continuous objectives and establishes convergence only to $\delta$-Goldstein stationary points, a notion weaker than the stationarity considered in this work.

\section{Preliminaries}\label{sec:preliminary}

\paragraph{Notations} We use boldface letters (e.g. $\xbf$ and $\ybf$) to represent vectors. Let $\Rbb^d$ be a $d$-dimensional Euclidean space with Euclidean inner product $\inner{\cdot}{\cdot}$; we use $\norm{\xbf}$ to express the induced norm $\norm{\cdot}=\sqrt{\inner{\cdot}{\cdot}}$. We use $\norm{\cdot}_p$ to denote the vector $L^p$-norm. Hence $\norm{\cdot}_2$ is $\norm{\cdot}$ by default.  For a set $\Scal\subset\Rbb^d$, we define $\dist(\xbf, \Scal)\assign\inf \{\norm{\ybf-\xbf}: \ybf\in \Scal\}$ and use
$\norm{\Scal}\assign \dist(\zero, \Scal)$ to denote its distance to the origin.  
For any convex function $f:\Rbb^d \rightarrow (-\infty,+\infty]$, its Fenchel conjugate is defined by $f^*(\ybf)\coloneqq \sup_{\xbf} \{\inner{\ybf}{\xbf} - f(\xbf)\}$.  When $f$ is closed and convex, we have biconjugacy $f=f^{**}$ \citep[4.2]{borwein2006convex}.
Denote the proximal operator of $f$ by $\prox_{f/\beta}(\xbf)\coloneqq\argmin_{\ybf\in\Rbb^d}\{f(\ybf)+\frac{\beta}{2}\norm{\ybf-\xbf}^{2}\}$. We use \(\level(f,v) \coloneqq \{\mathbf{x} \in \mathbb{R}^d : f(\mathbf{x}) \le v\}\) to denote the $v$-sublevel set of a function $f$. 

\paragraph{Subdifferential and approximate stationarity}
Let $f:\Rbb^d\raw(-\infty,+\infty]$ be a proper lower semi-continuous function. 
The Fr\'{e}chet subdifferential at $\xbf$ is given by $\hat{\partial}f(\xbf)=\{\vbf: {f(\ybf)\ge f(\xbf)+\inner{\vbf}{\ybf-\xbf}}+o(\norm{\ybf-\xbf}), \ \text{as } \ybf\raw\xbf \}$. 
The limiting subdifferential is defined by $\partial f(\xbf)=\{\vbf: \exists \xbf^k\raw\xbf, f(\xbf^k)\raw f(\xbf), \vbf^k\in\hat\partial f(\xbf^k), \vbf^k\raw \vbf\}$. 
$f$ is said to be a \emph{$\rho$-weakly convex function} ($\rho\ge 0$) if $f(\xbf)+\frac{\rho}{2}\norm{\xbf}^{2}$ is convex. For (weakly) convex functions, these two subdifferentials coincide. 
For a nonsmooth problem, $\norm{\partial f(\xbf)}$ may not be a good convergence criterion. For example, consider $f(x)=|x|$: whenever $x\neq 0$, we have $\partial f(x)={\operatorname{sign}(x)}$ and $\|\partial f(x)\|=1$, irrespective of how far $x$ is from the minimizer $x^{\star}=0$. Therefore, we adopt the following notion of stationarity.
\begin{defn}[Approximate stationarity]\label{defi:apprx-stationary}
We say that $\xbf$ is a ($\delta,\vep$)-(approximate) stationary point of
problem~\eqref{pb:main} if there exists a point $\hatx\in\dom\phi$
such that 
$\norm{\hatx-\xbf}\le\delta$ and $\norm{\partial \phi(\hatx)}\le\vep$. For convenience, we informally call $\xbf$ an $\varepsilon$-approximate stationary point when $\norm{\hatx-\xbf}\le \Ocal(\vep)$ and $\norm{\partial \phi(\hatx)}\le\Ocal(\vep)$.
\end{defn}

\paragraph{Moreau Envelope}
Let $f$ be a $\rho$-weakly convex function. The Moreau envelope~\citep{moreau1965proximite} of $f$ is defined as the smooth function
\begin{equation}\label{eq:moreau-envelope}
f^{\beta}(\xbf)\coloneqq\,\min_{\ybf} \Big\{ f(\ybf)+\frac{\beta}{2}\norm{\ybf-\xbf}^{2} \Big\},
\end{equation}
where $\beta\in(\rho, \infty)$ is the parameter of the proximal term. Moreau envelope plays a central role in our analysis. It will be used as both an approximate stationary criterion and an explicit smoothing tool. 
\begin{lem}\label{lem:moreau}
The Moreau envelope $f^{\beta}$ defined in \eqref{eq:moreau-envelope} is differentiable with gradient
\begin{equation}
    \nabla f^{\beta}(\xbf)=\beta\brbra{\xbf-\prox_{f/\beta}(\xbf)} \in \partial f(\prox_{f/\beta}(\xbf)). \label{eqn:moreau-gradient}
\end{equation}
\end{lem}

\paragraph{Smooth and generalized smooth functions}
Analyzing smooth approximations of non-globally Lipschitz functions requires a generalization of global smoothness, which we detail in \Cref{defi:gen-smooth}.
\begin{defn}[Generalized smoothness]\label{defi:gen-smooth}
    Let $f:\Rbb^{d}\to\Rbb$ be differentiable on a closed, convex set $\Xcal\subseteq\Rbb^{d}$. We say that $f$ is \emph{$\Lcal$-generalized  smooth} on $\Xcal$ if
\begin{equation}\label{eq:gen-smooth}
\norm{\nabla f(\xbf)-\nabla f(\ybf)} \le \Lcal(\xbf,\ybf) \norm{\xbf-\ybf}\quad \text{for all } \xbf,\ybf \in \Xcal,
\end{equation}
where $\Lcal:\Xcal\times \Xcal \raw [0,\infty)$ is non-negative and symmetric in its two arguments on its domain. 
\end{defn}
\begin{rem}
Condition~\eqref{eq:gen-smooth} strictly generalizes the classical global $L$-Lipschitz smoothness assumption: if $\Lcal(\xbf,\ybf)\equiv L$ for some constant $L>0$, $f$ is said to be an (globally) $L$-Lipschitz smooth function on~$\Xcal$. The requirement that $\mathcal{L}(\xbf, \ybf)$ be symmetric is made without loss of generality; if an initial $\Lcal$ is not symmetric, it can be replaced by $\min\{\mathcal{L}(\mathbf{x},\mathbf{y}),\mathcal{L}(\mathbf{y},\mathbf{x})\}$.
\end{rem}

The following lemma establishes useful curvature bounds for weakly convex and smooth functions.
\begin{lem} \label{lem:smoothness}
    Let $g$ be a $\rho$-weakly convex and differentiable function on the domain $\Xcal$. 
\begin{enumerate}[label=\arabic*), leftmargin=15pt]
    \item If $g$ is $L$-Lipschitz smooth, then
$g(\xbf)-g(\ybf) \le \inner{\nabla g(\ybf)}{\xbf-\ybf} + \frac{L}{2} \norm{\ybf-\xbf}^2,$ $\text{for all } \xbf,\ybf\in\Xcal$.
\item If $g$ is $\Lcal$-generalized smooth, then
$g(\xbf)-g(\ybf) \le \inner{\nabla g(\ybf)}{\xbf-\ybf} + \bsbra{\frac{\rho}{2} + \Lcal(\ybf, \xbf)}\norm{\xbf-\ybf}^2,$ $\text{for all } \xbf,\ybf\in\Xcal$.
\end{enumerate}
\end{lem}

The above results show that generalized smoothness leads to a looser quadratic upper bound (with a factor of $2$ and additional $\rho$) compared to the standard smooth case. \Cref{prop:tightness} demonstrates that this bound is nearly tight.
\begin{prop} \label{prop:tightness}Let $\rho\ge 0$. For any $\vep \in (0,1)$, there exists a $\rho$-weakly convex and $\Lcal_g$-generalized smooth function $g$, together with points $\xbf,\ybf\in\dom g$, such that
    \[g(\xbf)\ge g(\ybf) +  \inner{\nabla g(\ybf)}{\xbf-\ybf} + (1-\vep)\frac{\rho+2\Lcal_g(\ybf, \xbf)}{2} \norm{\ybf-\xbf}^2.\]
\end{prop}

\section{Smoothing theory}

This section develops the theory of smooth approximation for weakly convex optimization. We begin by formalizing the notion of smooth approximation of a function and
establishing its connection to the stationarity measures of the original function. As an important application of our framework, we discuss partial smoothing for composite optimization problems.

\subsection{Smoothable function} 
To quantify the quality of a smooth approximation, we start by formalizing the definition of a (generalized) smooth approximation of a function $f$.
\begin{defn}
\label{df:smooth-apprx} Let $f:\Rbb^d\raw\Rbb$ be weakly
convex on a closed convex set $\Xcal\subseteq\Rbb^d$. We say that $f$ admits a generalized smooth approximation (or generalized smoothable) on $\Xcal$ if  there exists  a family of continuously differentiable functions $\{f_\eta\}_{\eta>0}$, such that for any $\eta>0$, i) $f_\eta$ is $\bar\rho$-weakly convex for a constant $\bar\rho>0$, ii) there exists a nonnegative function $\Rcal_{\eta}:\Xcal\rightarrow \Rbb_+$, and a symmetric function $\Lcal_{\eta}:\Xcal\times \Xcal \rightarrow \Rbb_+$,
such that the following conditions hold:
\begin{enumerate}[label=\textbf{S\arabic*:},ref={\textbf{S\arabic*}}]
\item $f_{\eta}(\xbf)\le f(\xbf)\le f_{\eta}(\xbf)+\Rcal_\eta(\xbf)$, for any $\xbf\in\Xcal$.
\label{SA-func} 
\item \label{SA-smooth} $f_\eta(\xbf)$ is $\Lcal_\eta$-generalized smooth on $\Xcal$:
$\norm{\nabla f_\eta(\xbf)-\nabla f_\eta(\ybf)}\le \Lcal_\eta(\xbf, \ybf)\,\norm{\xbf-\ybf}.$
\end{enumerate}
When these conditions are met, $f_{\eta}$ is called a $(\bar\rho, \Rcal_\eta,\Lcal_\eta)$-smooth approximation (SA) of $f$. For brevity, we may refer to it as a $(\bar\rho,\eta)$-SA, or, when the context is clear, simply an $\eta$-SA of $f$.
\end{defn}

Conditions~\ref{SA-func}--\ref{SA-smooth} differ from the standard notion of smooth approximation~\citep{Beck2012smoothing}, in which $\Rcal_\eta$ and $\Lcal_\eta$ do not depend on $\xbf$ or $\ybf$.
This added flexibility is crucial for handling nonsmooth functions that are not globally Lipschitz continuous, thereby broadening the class of problems that can be addressed; see \Cref{sec:nes-smoothing}.
Moreover, our framework extends smooth approximation theory beyond the convex setting to encompass nonconvex functions.
In the convex setting, convergence rates are typically measured in terms of the function value gap, and Condition~\ref{SA-func} directly relates the optimality gap of the nonsmooth problem to that of its smooth approximation.
This connection is less informative in nonconvex optimization, where achieving global optimality is generally intractable.
Accordingly, we instead establish a relationship between the first-order properties of the approximation $f_{\eta}$ and those of the original function $f$.
We start by showing that the gradient of the smoothed function can be interpreted as a specific approximate subgradient of $f$, which is defined below. 
\begin{defn}[Approximate subgradient]
Let $f$ be a proper, lower semi-continuous, and weakly convex function.
A vector $\ubf\in\Rbb^{d}$ is called an $(\bar\rho, \varepsilon)$-subgradient
of $f$ at $\xbf\in\dom f$ if, for some $\bar{\rho}\ge 0$,
\[
f(\ybf)\ge f(\xbf)+\inner{\ubf}{\ybf-\xbf}-\frac{\bar\rho}{2}\norm{\ybf-\xbf}^{2}- \varepsilon
\]
holds for any $\ybf\in\Rbb^{d}$. The set of all such vectors is called the $(\bar\rho, \varepsilon)$-subdifferential.
\end{defn}
A smooth approximation provides a natural source of approximate subgradients.
\begin{prop}\label{prop:rel-sa-epsi-subgrad}
    Let $f_\eta$ be a $(\bar\rho,\eta)$-smooth approximation of $f$, then $\nabla f_\eta(\xbf)$ is a $(\bar\rho,\Rcal_\eta(\xbf))$-subgradient of $f$ at $\xbf$.
\end{prop}
\begin{proof}
    From the definition of a smooth approximation (\ref{SA-func}) and weak convexity of  $f_\eta$, we have
    \begin{equation*}
        \begin{aligned}
            f(\ybf) -f(\xbf) & \ge f_\eta(\ybf) -f_\eta(\xbf) + f_\eta(\xbf) - f(\xbf) \\ 
            & \ge \inner{\nabla f_\eta(\xbf)}{\ybf-\xbf} - \frac{\bar\rho}{2}\norm{\xbf-\ybf}^2 - \Rcal_\eta(\xbf),
        \end{aligned}
    \end{equation*}
    which completes our proof.
\end{proof}
\begin{rem}
If $f$ is convex, then a $(0,\vep)$-subgradient reduces to the $\vep$-subgradient in convex optimization~\citep{brondsted1965subdifferentiability}.
To our knowledge, its extension to weakly convex functions was previously explored by \citet{ruszczynski1987linearization}, and its use in the complexity analysis for nonconvex optimization has been more recently leveraged in \citet{boob2024level} and \citet{van2024weak}.
Our definition differs slightly by allowing the weak convexity parameter of the approximation $\bar\rho$ to be different from $\rho$. This flexibility allows us to handle smoothing operations that might increase the weak convexity modulus (i.e. $\bar{\rho}>\rho$), a scenario we will examine further in the next section.
\end{rem}

The following theorem connects the norm of the $\vep$-subgradient and approximate stationarity of the problem $\min_\xbf f(\xbf)$. 
\begin{thm}
\label{thm:e-subgrad} 
Let $f$ be proper, lower semi-continuous, and $\rho$-weakly convex $\rho>0$, and let $\xbf\in\dom f$.
\begin{enumerate}[label=\arabic*), leftmargin=15pt]
\item If $\vbf$ is a $(\bar\rho,\varepsilon)$-subgradient of $f$ at $\xbf$, then $\xbf$
is a $\Brbra{\sqrt{\frac{2\varepsilon}{\hat{\rho}-\rho}},\norm{\vbf}+\hat{\rho}\sqrt{\frac{2\varepsilon}{\hat{\rho}-\rho}}}$-stationary
point for any $\hat{\rho}>\max\{\bar\rho, \rho\}$. If $\varepsilon=0$ and $\bar\rho=\rho$, then $\vbf\in\partial f(\xbf)$.
\item Conversely, if $\xbf$ is an {$({\eta}/{\rho},\eta)$}-stationary point, then there exists a $(\rho,\varepsilon)$-subgradient $\vbf$ of $f$ at $\xbf$ such that $\norm{\vbf}\le 2\eta$, where $\varepsilon\coloneqq{2\eta^2}/{\rho}+\eta\norm{\partial f(\xbf)}/\rho$. %
\end{enumerate}
\end{thm}
\begin{proof} %

\textsf{Part 1).} Let $\vbf$ be an $\vep$-subgradient of $f$ at $\xbf$. Then, for any $\ybf$ and any $\hat{\rho}>\max\{\rho,\bar\rho\}$, we have 
\begin{equation}
f(\ybf)\ge f(\xbf)+\inner{\vbf}{\ybf-\xbf}-\frac{\hat{\rho}}{2}\norm{\ybf-\xbf}^{2}-\varepsilon.
\end{equation}
Rearranging the terms, we have
\begin{equation}
f(\xbf)\le f(\ybf)+\frac{\hat{\rho}}{2}\norm{\ybf-\xbf}^{2}+\inner{\vbf}{\xbf-\ybf}+\varepsilon.\label{eq:mid-01}
\end{equation}
Let us define $\Psi(\ybf)=f(\ybf)+\frac{\hat{\rho}}{2}\norm{\ybf-\xbf}^{2}+\inner{\vbf}{\xbf-\ybf}$.
It is clear that $\Psi(\xbf)=f(\xbf)$. Then \eqref{eq:mid-01} implies
\begin{equation}
\Psi(\xbf)\le\Psi(\ybf)+\varepsilon,\quad\ \text{for all }\ybf\in\Rbb^{d}.\label{eq:mid-03}
\end{equation}
Due to the weak convexity of $f$ and $\hat{\rho}>\rho$, $\Psi(\ybf)$
is $(\hat{\rho}-\rho)$-strongly convex, let $\ybf_{\xbf}$ denote
the global minimizer~$\argmin_{\ybf}\Psi(\ybf)$. It then follows
from strong convexity that 
\begin{equation}
\Psi(\xbf)-\Psi(\ybf_{\xbf})\ge\frac{\hat{\rho}-\rho}{2}\norm{\ybf_{\xbf}-\xbf}^{2}.\label{eq:mid-02}
\end{equation}
Combining \eqref{eq:mid-03} and \eqref{eq:mid-02}, we have 
\begin{equation}\label{eq:mid-08}
\norm{\ybf_{\xbf}-\xbf}\le\sqrt{\frac{2\varepsilon}{\hat{\rho}-\rho}}.
\end{equation}
Using the global optimality of $\ybf_{\xbf}$, we have 
\begin{equation}\label{eq:mid-10}
0\in\partial f(\ybf_{\xbf})+\hat{\rho}(\ybf_{\xbf}-\xbf)-\vbf.
\end{equation}
which implies $\norm{\partial f(\ybf_{\xbf})}\le\norm{\vbf}+\hat{\rho}\norm{\ybf_{\xbf}-\xbf}\le\norm{\vbf}+\hat{\rho}\sqrt{\frac{2\varepsilon}{\hat{\rho}-\rho}}$.

Suppose $\varepsilon=0$ and $\bar\rho=\rho$, then combining \eqref{eq:mid-08} and \eqref{eq:mid-10}, it is easy to see $\vbf\in\partial f(\xbf)$.
Hence, we complete the proof of \textsf{Part 1)}.

\textsf{Part 2).} In view of the $(\eta/\rho,\eta)$-stationary
condition, there exists an $\hatx$ and a subgradient $\hat{\ubf}\in\partial f(\hatx)$
such that $\norm{\hatx-\xbf}\le\eta/\rho$ and $\norm{\hat{\ubf}}\le\eta$. Let $\ubf\in\partial f(\xbf)$ be a minimum-norm subgradient. 
Applying the definition of weak convexity,  we have 
\begin{equation}\label{eq:mid-04}
\begin{aligned}
f(\ybf) & \ge f(\hatx)+\inner{\hat{\ubf}}{\ybf-\hatx}-\frac{\rho}{2}\norm{\ybf-\hatx}^{2} \nonumber\\
     & \ge f(\xbf) + \inner{\ubf}{\hat\xbf-\xbf}+\inner{\hat{\ubf}}{\ybf-\hatx} -\frac{\rho}{2}\norm{\ybf-\hatx}^{2} -\frac{\rho}{2}\norm{\xbf-\hatx}^{2}\nonumber \\
 & = f(\xbf)+\inner{\hat{\ubf}}{\ybf-\xbf} +\inner{\hat{\ubf}-\ubf}{\xbf-\hat\xbf} -\frac{\rho}{2}\norm{\ybf-\hat\xbf}^{2} -\frac{\rho}{2}\norm{\xbf-\hatx}^{2}.
\end{aligned}
\end{equation}
It is elementary to check 
\[
\frac{\rho}{2}\norm{\ybf-\hatx}^{2}=\frac{\rho}{2}\norm{\ybf-\xbf}^{2}+\frac{\rho}{2}\norm{\xbf-\hatx}^{2}+\rho\inner{\ybf-\xbf}{\xbf-\hatx}.
\]
Placing this result in \eqref{eq:mid-04}, we have 
\[\begin{aligned}
f(\ybf) & \ge f(\xbf)+\inner{\hat{\ubf}-\rho(\xbf-\hat\xbf)}{\ybf-\xbf}-\frac{\rho}{2}\norm{\ybf-\xbf}^{2}-{\rho}\norm{\xbf-\hatx}^{2}-\norm{\hat{\ubf}-\ubf}\norm{\hatx-\xbf}\\
 & \ge f(\xbf)+\inner{\hat{\ubf}-\rho(\xbf-\hat\xbf)}{\ybf-\xbf}-\frac{\rho}{2}\norm{\ybf-\xbf}^{2}-\frac{\eta^2}{\rho}-(\eta+\norm{\ubf})\frac{\eta}{\rho}.
\end{aligned}
\]
Then it is clear that $\vbf=\hat{\ubf}-\rho(\xbf-\hat\xbf)$ is an $\varepsilon$-subgradient of size $\norm{\vbf}\le \norm{\hat\ubf}+\rho\norm{\xbf-\hat\xbf}\le 2\eta$, where $\varepsilon\coloneqq{2\eta^2}/{\rho}+\eta\norm{\partial f(\xbf)}/\rho$.
\end{proof}

\paragraph{Intuition for algorithm design}
We illustrate the motivation of smoothing when $f$ is real-valued. It is natural to consider the smooth approximation problem $\min_{\xbf} f_\eta(\xbf)$ instead.
In view of \Cref{prop:rel-sa-epsi-subgrad} and \Cref{thm:e-subgrad}, achieving a small gradient norm $\|\nabla f_\eta(\xbf)\|$ directly corresponds to attaining an approximate stationary point of \eqref{pb:main}. 
Specifically, finding a solution $\xbf$ such that $\norm{\nabla f_\eta(\xbf)}\le \vep$ yields an $(\Ocal(\sqrt{\Rcal_\eta(\xbf)}), \Ocal(\vep+\sqrt{\Rcal_\eta(\xbf)}))$-stationary point to the original nonsmooth problem.
Consequently, our goal reduces to applying smooth algorithms to minimize the gradient norm $\norm{\nabla f_\eta(\xbf)}$. 
Furthermore, while our illustration is based on an unconstrained problem, the same idea extends to optimizing a composite objective \(\phi(\mathbf{x})=f(\mathbf{x})+r(\mathbf{x})\), where $f(\xbf)$ can be more complicated, such as the expectation function $f(\xbf)=\Expe_\xi [f(\xbf,\xi)]$. We will develop termination criteria other than the gradient norm for this composite setting.

\subsection{Partial smooth minimization of composite problems}\label{sec:smoothing-composite}

Consider the composite problem \eqref{pb:main}, which seeks to minimize the sum of two functions, $f$ and $r$, both of which can be nonsmooth.
In many applications, the function $r$ serves to enforce some desirable structure of the solution, such as sparsity, or represents the indicator function of a constraint set. Functions of this nature are typically preserved in their exact, nonsmooth form. Therefore, we focus on the \emph{partially smoothed} problem
\begin{equation} \label{pb:comp-smooth-2}
	\min_{\xbf\in\Rbb^{d}}\,\phi_{\eta}(\xbf)=f_{\eta}(\xbf)+r(\xbf)
\end{equation}
 where only $ f $ is replaced by its smooth approximation $f_{\eta}$.
We assume that the smooth approximation \( f_{\eta} \) is computationally tractable, namely, its (approximate) first-order information is available via certain oracle calls. 
To analyze algorithms for solving \eqref{pb:comp-smooth-2}, we consider two commonly used stationarity measures from the literature.
The first is the norm of the \emph{generalized gradient}, which is closely associated with the proximal gradient method \citep{Saeed2016accelerated, nesterov2013gradient}. Given a stepsize $\gamma > 0$, the generalized gradient is defined by:
\begin{equation}\label{eq:generalized-grad}
\Gcal_\gamma ({\xbf}) \coloneqq \frac{1}{\gamma}(\xbf-\hat{\xbf}),\ \ \text{where } \, 
\hat{\xbf}\coloneqq
\prox_{\gamma r}(\xbf-\gamma \nabla f_\eta(\xbf)).
\end{equation}
The optimality condition for the proximal operator implies that
\[
\zerobf \in\nabla f_{\eta}(\xbf)+\frac{1}{\gamma}(\hat{\xbf}-\xbf)+\partial r(\hat{\xbf}).
\]
When $\Gcal_{\gamma}\rbra{\xbf}=\zerobf$, $\xbf$ satisfies the first-order stationarity condition for \eqref{pb:comp-smooth}. Alternatively, \citet{davis2019stochastic} propose the \emph{gradient norm of the Moreau envelope}:  $\norm{\nabla\phi_{\eta}^{\beta}(\xbf)}$ as a stationarity measure. The quantitative relationship between these two measures has been explored by \citet{drusvyatskiy2019efficiency}.
The following theorem connects these criteria for the smoothed problem to the approximate stationarity of the original nonsmooth problem. 
\begin{thm}
\label{thm:smooth-criteria-convert} Suppose that $f$ is $\rho$-weakly convex and $f_\eta$ is a $(\bar{\rho},\eta)$-smooth approximation of $f$.
Let $\xbf\in\dom\phi$.  
\begin{enumerate}[label=\arabic*), leftmargin=15pt]
\item If $\xbf$ satisfies $\norm{\Gcal_{\gamma}(\xbf)}\le\vep$
for some $\vep>0$, then $\xbf$ is a $\Brbra{\sqrt{\frac{2\Rcal_\eta(\hat{\xbf})}{\hat{\rho}-\rho}}+\gamma\vep,\left(1+\gamma \Lcal_{\eta}({\xbf, \hat{\xbf}})\right)\vep+\hat{\rho}\sqrt{\frac{2\Rcal_\eta(\hat{\xbf})}{\hat{\rho}-\rho}}}$-stationary
point for any $\hat{\rho}>\max\{\rho,\bar\rho\}$, where $\hat{\xbf}$ is defined in \eqref{eq:generalized-grad}. 
\item Assume $\beta > \max\{\rho,\bar\rho\}$. If $\norm{\nabla\phi_{\eta}^{\beta}(\xbf)}\le\vep$, then $\xbf$
is a $\Brbra{\sqrt{\frac{2\Rcal_\eta(\ybf_\xbf)}{\beta-\rho}}+\beta^{-1}\vep,\beta\sqrt{\frac{2\Rcal_\eta(\ybf_\xbf)}{\beta-\rho}}+\vep}$-stationary point of problem~\eqref{pb:main}, where $\ybf_{\xbf}=\prox_{\phi_{\eta}/\beta}(\xbf)$.
\end{enumerate}
\end{thm}
\begin{proof}
\textsf{Part 1).} By definition, $\norm{\Gcal_\gamma(\xbf)} = \gamma^{-1} \norm{\xbf-\hat{\xbf}} \leq \varepsilon$. Using the optimality condition of the proximal operator, there exists $\vbf \in \partial r(\hat{\xbf})$ such that 
$\xbf - \hat{\xbf} = \gamma(\nabla f_\eta(\xbf) + \vbf)$.
We now bound the subgradient norm of the smoothed objective at $\hat{\xbf}$:
\begin{equation}
\begin{aligned}
\norm{\partial\phi_{\eta}(\hat{\xbf})} & \le \norm{\nabla f_\eta(\hat{\xbf}) + \vbf} \\
&\leq \norm{\nabla f_{\eta}(\xbf) - \nabla f_{\eta}(\hat{\xbf})} + \norm{\nabla f_{\eta}(\xbf) + \vbf}\\
&= \norm{\nabla f_{\eta}(\xbf) - \nabla f_{\eta}(\hat{\xbf})} + \gamma^{-1}\norm{\xbf - \hat{\xbf}}\\
&\leq \Lcal_{\eta}(\xbf, \hat{\xbf})\norm{\xbf-\hat{\xbf}} + \gamma^{-1}\norm{\xbf-\hat{\xbf}}\\
&= \left(1 + \gamma \Lcal_{\eta}(\xbf, \hat{\xbf})\right)\varepsilon.
\end{aligned}
\end{equation}
Since $\phi_\eta(\hat{\xbf}) \leq \phi(\hat{\xbf}) \leq \phi_\eta(\hat{\xbf}) + \Rcal_\eta(\hat{\xbf})$, by an argument similar to that in the proof of \Cref{prop:rel-sa-epsi-subgrad}, 
we see that the subgradient $\nabla f_\eta(\hat{\xbf}) + \vbf$ is a $(\bar{\rho}, \Rcal_\eta(\hat{\xbf}))$-subgradient of $\phi(\hat{\xbf})$. 
Applying \Cref{thm:e-subgrad}, it follows that the point $\hat{\xbf}$ is a $\Big(\sqrt{\frac{2\Rcal_\eta(\hat{\xbf})}{\hat{\rho}-\rho}}, \big(1+\gamma \Lcal_{\eta}(\xbf,\hat{\xbf})\big)\varepsilon + \hat{\rho}\sqrt{\frac{2\Rcal_\eta(\hat{\xbf})}{\hat{\rho}-\rho}}\Big)$-stationary point. 
This implies the existence of a point $\tilde{\xbf}$ satisfying:
\[\norm{\tilde{\xbf}-\hat{\xbf}} \leq \sqrt{\frac{2\Rcal_\eta(\hat{\xbf})}{\hat{\rho}-\rho}}, \ \  \text{and} \ \ 
\norm{\partial\phi(\tilde{\xbf})} \leq \left(1+\gamma \Lcal_\eta(\xbf,\hat{\xbf})\right)\varepsilon + \hat{\rho}\sqrt{\frac{2\Rcal_\eta(\hat{\xbf})}{\hat{\rho}-\rho}}.
\]
Finally, applying the triangle inequality yields:
\begin{equation}
\norm{\xbf-\tilde{\xbf}} \leq \norm{\tilde{\xbf}-\hat{\xbf}} + \norm{\hat{\xbf}-\xbf} \leq \sqrt{\frac{2\Rcal_\eta(\hat{\xbf})}{\hat{\rho}-\rho}} + \gamma\varepsilon.
\end{equation}
This completes the proof of \textsf{Part 1)}. 

\textsf{Part 2)}. Using \Cref{lem:moreau}, we have 
\begin{equation}
\norm{\ybf_{\xbf}-\xbf}=\beta^{-1}\norm{\nabla\phi_{\eta}^{\beta}(\xbf)}\le\beta^{-1}\vep.\label{eq:opt-10}
\end{equation}
Moreover, the stationary condition shows
\[
\zerobf\in\partial\phi_{\eta}(\ybf_{\xbf})+ {\beta}(\ybf_{\xbf}-\xbf) = \nabla f_\eta (\ybf_\xbf) + \partial r(\ybf_\xbf) + {\beta}(\ybf_{\xbf}-\xbf),
\]
which implies $\norm{\nabla f_{\eta}(\ybf_{\xbf})+\vbf}\le\vep$
for some $\vbf\in\partial r(\ybf_{\xbf})$. Since $\nabla f_\eta(\ybf_\xbf)$ is a $(\bar\rho,R_\eta(\ybf_\xbf))$ subgradient of $f(\ybf_\xbf)$ and $\vbf$ is a subgradient of $r(\ybf_\xbf)$, 
it is easy to see that $\nabla f_{\eta}(\ybf_{\xbf})+\vbf$ is a $(\bar\rho,R_\eta(\ybf_\xbf))$-subgradient of $\phi(\ybf_{\xbf})$.  
Consequently, \Cref{thm:e-subgrad} implies $\ybf_{\xbf}$ is a $\Brbra{\sqrt{\frac{2\Rcal_\eta(\ybf_\xbf)}{\beta-\rho}},\beta\sqrt{\frac{2\Rcal_\eta(\ybf_\xbf)}{\beta-\rho}}+\vep}$-stationary
point of $\phi$. Namely, there exists some $\hat{\ybf}$ such that 
\[
\norm{\hat{\ybf}-\ybf_{\xbf}}\le\sqrt{\frac{2\Rcal_\eta(\ybf_\xbf)}{\beta-\rho}},\quad\norm{\partial\phi(\hat{\ybf})}\le\beta\sqrt{\frac{2\Rcal_\eta(\ybf_\xbf)}{\beta-\rho}}+\vep.
\]
Using triangle inequality and \eqref{eq:opt-10}, we get 
\[
\norm{\hat{\ybf}-\xbf}\le\norm{\hat{\ybf}-\ybf_{\xbf}}+\norm{\ybf_{\xbf}-\xbf}\le\sqrt{\frac{2\Rcal_\eta(\ybf_\xbf)}{\beta-\rho}}+\beta^{-1}\vep.
\]
This result and the bound on $\norm{\partial\phi(\hat{\ybf})}$ lead
to the approximate stationarity of $\xbf$. 
\end{proof}

\begin{rem}
\Cref{thm:smooth-criteria-convert} provides a guidance on selecting the smoothing parameter $\eta$. For many standard smoothing techniques, the approximation error $\Rcal_\eta$ is of order $\Ocal(\eta)$. To obtain a target $(\vep,\vep)$-stationary point for the original problem, the dominating term $\sqrt{2\Rcal_\eta(\ybf_\xbf)/(\beta-\rho)}$ must be of order $\Ocal(\vep)$, which means we need to enforce $\Rcal_\eta = \Ocal(\vep^2)$. For $\Rcal_\eta=\Ocal(\eta)$ the smoothing parameter should be chosen to satisfy $\eta = \Ocal(\vep^2)$.
\end{rem}

\section{Smoothing operations}\label{sec:smoothing-operations}

In this section, we discuss several smoothing techniques and their applications to weakly convex functions. 

\subsection{Generalized Nesterov's smooth approximation}\label{sec:nes-smoothing}

Building on the seminal work of \citet{nesterov2005smooth}, we study nonsmooth functions that can be expressed as the composition of a convex function with a \emph{nonlinear} map:
\begin{equation}
f(\xbf)=h(A(\xbf)).\label{eq:composite-form}
\end{equation}
Here, $h:\Vbb\raw\reals$ is a convex continuous function, $A:\Ebb\rightarrow \Vbb$ is a smooth mapping, and $\Ebb$ and $\Vbb$ are two finite-dimensional vector spaces. 
In Nesterov's setting, adding a strongly convex proximal term in the dual space yields a Lipschitz smooth approximation function. Our extension generalizes this idea by accommodating the additional curvature induced by the nonlinear map.

\paragraph{Problem setup}  We equip $\Ebb$ with the standard Euclidean norm $\norm{\cdot}$ and $\Vbb$ with a general norm $\norm{\cdot}_\Vbb$. Let $\Vbb^*$ be the dual space with $\norm{\ybf}_{\Vbb^*} = 
 \sup \{\langle \xbf, \ybf \rangle: \norm{\xbf}_\Vbb\leq1\}$. Since $\Ebb$ is a Euclidean space, it is self-dual, and we denote its norm and the corresponding dual norm by $\norm{\cdot}_{\Ebb}=\norm{\cdot}_{\Ebb^*}=\norm{\cdot}$.
 Let $T:\Ebb\!\to\!\Vbb$ be a linear map between $\Ebb$ and $\Vbb$. Its operator norm is defined by $\norm{T}_\mathrm{op} \coloneqq \sup_{\|\xbf\|_{\Ebb}= 1} \norm{T\xbf}_\Vbb$. For the conjugate operator $T^*:\Vbb^*\!\to\!\Ebb^*$, we have $\norm{T^*}_{\mathrm{op}}=\norm{T}_{\mathrm{op}}$. A useful fact is 
\begin{equation}\label{eq:op-norm-bound}
\norm{T\xbf}_{\Vbb}\le \norm{T}_\textop \cdot \norm{\xbf}_{\Ebb}, \text{ and } \norm{T^*\ybf} \le \norm{T^*}_\textop \cdot \norm{\ybf}_{\Vbb^*}.
\end{equation}
The nonlinear map  $A$ is assumed to be continuously differentiable on a closed convex set $\Xcal\subseteq\Ebb$.   
We denote the Jacobian of $A$ at $\xbf$ by $\nabla A(\xbf)\;=\;[\nabla A_{1}(\xbf),\dots,\nabla A_{m}(\xbf)]^\top$ and assume that
\[
\norm{\nabla A(\xbf)-\nabla A(\hat{\xbf})}_\mathrm{op} \le L_A \norm{\xbf-\hat\xbf}, \quad \text{for all } \xbf,\hat\xbf\in\Xcal.
\]

Let $h^*(\ybf)=\sup_{\zbf\in\Vbb}\{\inner{\zbf}{\ybf}-h(\zbf)\}$ be the convex conjugate of $h$. Due to bi-conjugacy (e.g.~\citep[Theorem 4.2.1]{borwein2006convex}), we can express $f$ as 
\[
f(\xbf)=\max_{\ybf\in\Vbb^*}\bcbra{\inner{\ybf}{A(\xbf)}-h^{*}(\ybf)}.
\]
Let $\omega:\Vbb^*\rightarrow\Rbb\cup \{+\infty\}$ be a prox-function, namely, $\omega$ is lower semi-continuous, differentiable and $\sigma$-strongly convex ($\sigma>0$) on $\dom h^*$ with respect to $\norm{\cdot}_{\Vbb^*}$. 
Without loss of generality, we assume $\min_{\ybf\in\Vbb^*}\omega(\ybf)=0$; otherwise
we can simply let $\ybf^{\star}=\argmin_{\ybf\in\dom h^*}\omega(\ybf)$ and
then replace $\omega$ by $\tilde{\omega}(\ybf)=\omega(\ybf)-\omega(\ybf^{\star})$.  Furthermore, we assume $\dom h^*$ to be a bounded set and define $B\coloneqq \sup\{\norm{\ybf}_{\Vbb^*}: \ybf\in\dom h^* \}$ and 
$ D\coloneqq\sup \{\omega(\ybf): \ybf\in\dom h^*\}$. Now, consider the smooth function 
\begin{equation}\label{eq:nesterov-smoothing}
f_{\eta}(\xbf)\coloneqq h_{\eta}(A(\xbf)),\quad \text{where }\ h_{\eta}(\zbf)\coloneqq\sup_{\ybf\in\Vbb^*}\ \{\inner{\ybf}{\zbf}-h^{*}(\ybf)-\eta\omega(\ybf)\}.
\end{equation}
The key properties of $h_\eta$ are summarized below.
\begin{prop}\label{prop:hmu}
  Let $\eta>0$, then 
\begin{enumerate}[label=\arabic*), leftmargin=15pt]
    \item $h_{\eta}(\zbf)\le h(\zbf)\le h_{\eta}(\zbf)+\eta D$ for all $\zbf \in \Vbb$.
    \item $h_{\eta}$ is continuously differentiable with gradient $\nabla h_\eta (\zbf)=\ybf_{\zbf}\coloneqq\argmax_{\ybf\in\Vbb^*}\bcbra{\langle \ybf, \zbf \rangle -h^{*}(\ybf)-\eta\,\omega(\ybf)}.$ 
    \item $\nabla h_\eta$ is $L_{h_{\eta}}$-Lipschitz continuous with $L_{h_{\eta}}=\frac{1}{\sigma\eta}$. That is, for any $\zbf_1, \zbf_2\in\Vbb$, we have $\norm{\nabla h_\eta(\zbf_1)-\nabla h_\eta(\zbf_2)}_{\Vbb^*} \le \frac{1}{\sigma\eta} \norm{\zbf_1-\zbf_2}_\Vbb$.
\end{enumerate}
\end{prop}
\begin{proof}
\textsf{Part 1).} The inequality $h_{\eta}(\zbf)\le h(\zbf)$ follows from the non-negativity of $\omega$. For the second inequality,
note that 
\[
h_{\eta}(\zbf) =\max_{\ybf\in{\Vbb^*}}\ \{\inner{\ybf}{\zbf}-h^{*}(\ybf)-\eta\omega(\ybf)\}\ge\max_{\ybf\in \Vbb^*}\ \{\inner{\ybf}{\zbf}-h^{*}(\ybf)-\eta D\} = h(\zbf)-\eta D.
\]
\textsf{Part 2).}  The expression of the gradient $\nabla h_\eta(\zbf)$ directly follows from Danskin's theorem.\\
\textsf{Part 3).} Lipschitz smoothness of $h_\eta$ follows from the Baillon-Haddad theorem \citep[Lemma 4.1]{Beck2012smoothing}. 
\end{proof}

We now use \Cref{prop:hmu} to show that $f_{\eta}$ is a well-behaved smooth approximation of $f$.
\begin{thm}
\label{thm:nes-smoothing} We have the following properties of $f_{\eta}$:  
\begin{enumerate}[label=\arabic*), leftmargin=15pt]
\item $f_{\eta}(\xbf)\le f(\xbf)\le f_{\eta}(\xbf)+\Rcal_\eta(\xbf)$ with $\Rcal_\eta(\xbf)\equiv\eta D$. 
\item $f_{\eta}$ is continuously differentiable on $\Xcal$ with gradient $\nabla f_{\eta}(\xbf)=\nabla A(\xbf)^\top \ybf_{A(\xbf)}$, where $\ybf_{(\cdot)}$ is defined in \textsf{Part 2)} of \Cref{prop:hmu}.
\item $f_{\eta}$ is $B L_A$-weakly convex and generalized Lipschitz smooth with parameter 
\[
\Lcal_\eta(\xbf_1,\xbf_2)= \frac{1}{\sigma \eta} \sup_{0\le\theta\le 1}\norm{\nabla A(\theta \xbf_1+(1-\theta)\xbf_2)}_\mathrm{op}^2+BL_A.
\]
\end{enumerate}
Consequently, $f_{\eta}$ is a $(BL_A,\Rcal_\eta, \Lcal_\eta)$-smooth approximation of $f$ on $\Xcal$.
\end{thm}
\begin{proof} 

\textsf{Part 1)} immediately follows from \textsf{Part 1)} of \Cref{prop:hmu}.  

\textsf{Part 2).} The continuous differentiability and the gradient expression follow from  \Cref{prop:hmu} and the chain rule. 

\textsf{Part 3).} We first show the weak convexity of $f_\eta$. Using convexity of $h_{\eta}$, we have 
\begin{equation}\label{eq:mid-05}
\begin{aligned}
f_{\eta}(\xbf_{1}) & =h_{\eta}(A(\xbf_{1})) \\
 & \ge h_{\eta}(A(\xbf_{2}))+\nabla h_{\eta}(A(\xbf_{2}))^\top\bsbra{A(\xbf_{1})-A(\xbf_{2})} \\
 & =f_{\eta}(\xbf_{2})+\nabla h_{\eta}(A(\xbf_{2}))^\top\bsbra{\nabla A(\xbf_{2})(\xbf_{1}-\xbf_{2})} \\
 & \quad+\nabla h_{\eta}(A(\xbf_{2}))^\top\bsbra{A(\xbf_{1})-A(\xbf_{2})-\nabla A(\xbf_{2})(\xbf_{1}-\xbf_{2})}  \\
 & =f_{\eta}(\xbf_{2})+\nabla f_{\eta}(\xbf_{2})^\top(\xbf_{1}-\xbf_{2})+\nabla h_{\eta}(A(\xbf_{2}))^\top\bsbra{A(\xbf_{1})-A(\xbf_{2})-\nabla A(\xbf_{2})(\xbf_{1}-\xbf_{2})},
\end{aligned}
\end{equation}
where the last equality uses the chain rule $\nabla f_{\eta}(\xbf_{2})=\nabla A(\xbf_{2})^\top \nabla h_{\eta}(A(\xbf_{2}))$.
Moreover, the Lipschitz smoothness of $A$ implies
\begin{equation}
\norm{A(\xbf_{1})-A(\xbf_{2})-\nabla A(\xbf_{2})(\xbf_{1}-\xbf_{2})}_\Vbb\le\frac{L_{A}}{2}\norm{\xbf_{1}-\xbf_{2}}^{2}.\label{eq:mid-07}
\end{equation}
It then follows from \eqref{eq:mid-07} and Cauchy-Schwartz inequality
that 
\begin{equation}\label{eq:mid-06}
\begin{aligned}
 & \nabla h_{\eta}(A(\xbf_{2}))^\top\bsbra{A(\xbf_{1})-A(\xbf_{2})-\nabla A(\xbf_{2})(\xbf_{1}-\xbf_{2})} \\
 & \ge-\norm{\nabla h_{\eta}(A(\xbf_{2}))}_{\Vbb^*}\cdot\norm{A(\xbf_{1})-A(\xbf_{2})-\nabla A(\xbf_{2})(\xbf_{1}-\xbf_{2})}_{\Vbb} \\
 & \ge-\frac{L_{A}}{2}\norm{\nabla h_{\eta}(A(\xbf_{2}))}_{\Vbb^*} \cdot \norm{\xbf_{1}-\xbf_{2}}^{2} \\
 & \ge-\frac{BL_{A}}{2}\norm{\xbf_{1}-\xbf_{2}}^{2}.
\end{aligned}
\end{equation}
Combining \eqref{eq:mid-05} and \eqref{eq:mid-06} gives the desired weak convexity.

To establish the generalized smoothness property, we have
\begin{equation}
    \begin{aligned}
       &  \norm{\nabla f_{\eta} (\xbf_1)-\nabla f_{\eta} (\xbf_2)} \\
       & = \norm{\nabla A(\xbf_1)^\top \nabla h_\eta(A(\xbf_1))-\nabla A(\xbf_2)^\top\nabla h_\eta(A(\xbf_2))}  \\
        & = \norm{\nabla A(\xbf_1)^\top [\nabla h_\eta(A(\xbf_1))-\nabla h_\eta(A(\xbf_2))] + [\nabla A(\xbf_1)-\nabla A(\xbf_2)]^\top \nabla h_\eta(A(\xbf_2))  } \\
        & \le \norm{\nabla A(\xbf_1)}_\mathrm{op}\norm{\nabla h_\eta(A(\xbf_1))-\nabla h_\eta(A(\xbf_2))}_{\Vbb^*} + \norm{\nabla A(\xbf_1)-\nabla A(\xbf_2)}_\mathrm{op} \norm{\nabla h_\eta(A(\xbf_2))}_{\Vbb^*} \\
        & \le L_{h_\eta} \norm{\nabla A(\xbf_1)}_\mathrm{op} \norm{A(\xbf_1)-A(\xbf_2)}_\Vbb + B L_A \norm{\xbf_1-\xbf_2} 
    \end{aligned}
\end{equation}
where the first inequality follows from~\eqref{eq:op-norm-bound} and the second inequality applies the Lipschitz continuity of $\nabla h_\eta$ and $\nabla A$. By the mean value theorem,
   \begin{equation*}
   \begin{aligned}
   \norm{A(\xbf_1)-A(\xbf_2)}_\Vbb &\le \sup_{0\le \theta\le 1} \norm{\nabla A(\theta \xbf_1+(1-\theta)\xbf_2)}_\text{op}\norm{(\xbf_1-\xbf_2)}.
   \end{aligned}
   \end{equation*}
   Combining the above two relations, we have
   \[\norm{\nabla f_{\eta} (\xbf_1)-\nabla f_{\eta} (\xbf_2)} \le \Bcbra{\sup_{0\le\theta\le 1}\norm{\nabla A(\theta \xbf_1+(1-\theta)\xbf_2)}_\text{op}^2L_{h_\eta}+BL_A}  \norm{\xbf_1-\xbf_2}.\]
\end{proof}

\begin{rem}\label{remark:nesterov-smoothing-bound}
\Cref{thm:nes-smoothing} shows that the generalized smoothness of $f_{\eta}$ depends on the local properties of $\nabla A$. If this local parameter can be bounded uniformly, $f_{\eta}$ is then globally Lipschitz smooth. Two notable cases include:
\begin{enumerate}[label=\arabic*), leftmargin=15pt]
\item \textit{Linear Mapping:} If $A$ is an affine map, i.e., $A(\xbf)=Q\xbf+\bbf$ for some fixed $Q, \bbf$, then its Jacobian $\nabla A(\xbf) = Q$ is constant and $L_A=0$. The function $f_\eta(\xbf)$ is convex, and the smoothness parameter from \Cref{thm:nes-smoothing} becomes a global constant:
 $\Lcal_{\eta}(\xbf_1,\xbf_2) = \frac{\|Q\|_\mathrm{op}^2}{\sigma\eta}$.
\item \textit{Bounded Domain:} If the domain $\Xcal$ is compact, then the continuous map $\nabla A$ is bounded, i.e., $M\coloneqq\sup_{\xbf \in \Xcal} \|\nabla A(\xbf)\|_\mathrm{op} < \infty$. Hence, $\Lcal_\eta$ can be uniformly bounded as: $\Lcal_{\eta}(\xbf_1,\xbf_2)\le \frac{M^2}{\sigma\eta}+BL_A$.
\end{enumerate}
\end{rem}

\begin{example}[Piecewise smooth function]  Consider the max-type function 
\[f(\xbf)=
   \max_{1 \leq j \leq m}  f_j(\xbf) ,\] 
   where each $f_j(\xbf)$ is smooth. 
   This is a special case of \eqref{eq:composite-form} where 
   $h(\zbf)=\max_{1\le j\le m} z_j=\max_{\ybf\in\Scal^m} \langle \ybf, \zbf \rangle$ and $\Scal^m=\{\ybf\in\Rbb^m: \langle \onebf_m, \ybf \rangle  =1, \ybf\ge \zerobf\}$ is the probability simplex. Define the prox-function $\omega(\ybf)=\sum_{i=1}^m y_i \log y_i + \log m$ when $\ybf\in\Scal^m$ and $\omega(\ybf)=+\infty$ otherwise. Then the smoothed approximation is given by
\begin{equation} \label{eqn:softmax}
	   f_\eta(\xbf)= \max_{\ybf\in\Scal^m} \cbra{\langle \ybf,  F(\xbf) \rangle -\eta\omega(\ybf)} = \eta \log \Brbra{\sum_{i=1}^m \exp \Brbra{\frac{f_i(\xbf)}{\eta}}} - \eta \log m,
\end{equation}
   where $F(\xbf)=[f_1(\xbf),\ldots, f_m(\xbf)]^\top$. This is precisely the softmax approximation. 
\end{example}
We summarize the properties of the softmax operator below. 
\begin{coro} Suppose  $f_j(\xbf)$ is $L_j$-smooth for $j \in [m]$.
  Then  $f_\eta(\xbf)$ in \eqref{eqn:softmax} is a $(\bar\rho, \Rcal_\eta, \Lcal_\eta)$-SA of $f$ with $\bar\rho=\max_{j \in [m]} L_j$, $\Rcal_\eta(\xbf)=\eta \log m$, and $\Lcal_{\eta}(\xbf_1,\xbf_2)=\bar\rho+\eta^{-1} \sup_{\theta \in [0, 1]}\max_{1\le j\le m}\norm{\nabla f_j(\theta\xbf_1+(1-\theta)\hat\xbf_2)}^2 $.
\end{coro}
\begin{proof}
     Let $\norm{\cdot}_\Vbb=\norm{\cdot}_\infty$ and its dual norm $\norm{\cdot}_{\Vbb^*}=\norm{\cdot}_1$.
     It is known that $\omega$ is 1-strongly convex with respect to $\norm{\cdot}_1$ (i.e., $\sigma=1$).
     Then the smoothing function $h_\eta(\zbf)=\max_{\ybf\in\Scal^m} \{\langle \ybf, \zbf \rangle -\eta\omega(\ybf)\}$ is $\eta^{-1}$-Lipschitz smooth with respect to $\norm{\cdot}_\infty$.
   The operator norm $\norm{\cdot}_\text{op}$ for any mapping $A\in \Rbb^{m\times n}$ is given by $\norm{A}_\text{op}=\max_{j \in [m]} \norm{A_{j,:}}$, where $A_{j,:}$ denotes the $j$-th row. By the smoothness of $f_j$, 
   \[
   \norm{\nabla F(\xbf)-\nabla F(\hat\xbf)}_\text{op} = \max_{1\le j\le m} \norm{\nabla f_j(\xbf)-\nabla f_j(\hat\xbf)} \le \max_{1\le j\le m} L_j \norm{\xbf-\hat\xbf}.
   \]
   Hence, $F$ is Lipschitz smooth with $L=\max_{1\le j\le m} L_j$.
   It is easy to check that $D=\sup_{\zbf\in\Scal^m} \omega(\zbf)=\log m$ and $B=\sup_ {\ybf\in\Scal^m} \norm{\ybf}_1 = 1$. This completes the proof in view of the definition from \Cref{thm:nes-smoothing}.
\end{proof}

\subsection{Moreau envelope smoothing}
We present several properties regarding the weak convexity and Lipschitz smoothness of the Moreau envelope. While some of these results appear in prior works~\citep{davis2019stochastic,hoheisel2010proximal}, we include them here for completeness.
\begin{prop}
\label{prop:moreau-smoothing} Let $f$ be $\rho$-weakly convex and set $\beta\in(\rho,\infty)$. Then for any $\xbf\in\dom f$, we have 
 \begin{enumerate}[label=\arabic*), leftmargin=15pt]
 \item For any $\xbf, \ybf \in \dom f$, we have the following quadratic approximation bounds: 
 \begin{equation}
 \begin{aligned}
 f^{\beta}(\xbf) & \ge f^{\beta}(\ybf)+\inner{\nabla f^{\beta}(\ybf)}{\xbf-\ybf}-\frac{\rho}{2(1-\rho/\beta)}\norm{\ybf-\xbf}^{2},\\
 f^{\beta}(\xbf) & \le f^{\beta}(\ybf)+\inner{\nabla f^{\beta}(\ybf)}{\xbf-\ybf}+\frac{\beta}{2}\norm{\ybf-\xbf}^{2}.
 \end{aligned}
 \label{eq:quad-bound-2}
 \end{equation}
 Moreover, $\nabla f^{\beta}$ is $\max\{\beta,\frac{\rho}{1-\rho/\beta}\}$-Lipschitz continuous.
 \item Let $\xbf\in\dom f$ be subdifferentiable, we have 
 \[
 f^{\beta}(\xbf)+\frac{(1-\rho/\beta)}{2\beta}\norm{\nabla f^{\beta}(\xbf)}^{2}\le f(\xbf)\le f^{\beta}(\xbf)+\frac{\norm{\partial f(\xbf)}^{2}}{2(\beta-\rho)}.
 \]
 \end{enumerate}
\end{prop}
With all these setups, we are ready to establish the smoothing properties of the Moreau envelope. 
\begin{thm}
\label{thm:moreau-2} Suppose $f$ is $\rho$-weakly convex. Let $\eta>0$ and define $f_\eta(\xbf)=f^\beta(\xbf)$, where $\beta=\rho+\max\{\eta^{-1},\rho\}$.
Then $f_\eta$ is an $(2\rho,\eta)$-smooth approximation of $f$ with 
\begin{equation}\label{eq:opt-02}
\Rcal_\eta(\xbf) =\frac{\eta}{2\max\{1, \rho\eta\}} \norm{\partial f(\xbf)}^2, \quad \Lcal_\eta(\xbf,\ybf)=2\rho+\eta^{-1}.
\end{equation}
\end{thm}
\begin{proof}
In view of \eqref{eq:quad-bound-2}, we have
\[
\begin{aligned}
f_\eta(\xbf) - f_\eta(\ybf)+\inner{\nabla f_\eta(\ybf)}{\xbf-\ybf} 
& \ge -\frac{\rho}{2(1-\rho/\beta)}\norm{\ybf-\xbf}^{2} \\
& = -\frac{\rho}{2}\left(1+\frac{\rho}{\max\{\eta^{-1},\rho\}}\right)\norm{\ybf-\xbf}^{2} \\
& \ge -\rho \norm{\ybf-\xbf}^{2}.
\end{aligned}
\]
Hence $f_\eta$ is $2\rho$-weakly convex. The choice of $\beta$ ensures $\beta \ge \rho/(1-\rho/\beta)$. Together with the Lipschitz continuity of $\nabla f^{\beta}$, we have
\[
\norm{\nabla f^{\beta}({\xbf})-\nabla f^{\beta}(\ybf)} \le (\rho+\max\{\eta^{-1},\rho\})\norm{\xbf-\ybf} \le (2\rho+\eta^{-1}) \norm{\xbf-\ybf}.
\]
\textsf{Part 2)} of \Cref{prop:moreau-smoothing} implies $f_\eta(\xbf)\le f(\xbf)\le {f_\eta(\xbf) + \norm{\partial f(\xbf)}^2}/{2(\max\{\eta^{-1},\rho\})}$.
\end{proof}

\begin{rem}
   \Cref{thm:moreau-2} implies that for a fixed $\xbf$, the error term $\Rcal_\eta(\xbf)$ is $\Ocal(\min\{\eta, \rho^{-1}\})$, while the smoothness constant of the approximation is $\Ocal(\rho + \eta^{-1})$. The error term differs slightly from that of convex smoothing, where $\Rcal_\eta(\xbf)=R\eta$ for some $R>0$ \cite{beck2017first}. The difference arises as applying the Moreau envelope to a nonconvex function can increase its negative curvature (i.e., yield a larger weak convexity constant), as shown in \eqref{eq:quad-bound-2} of \Cref{prop:moreau-smoothing}.
\end{rem}

\begin{rem}\label{remark:moreau-smoothing-bound}
The error term $\Rcal_\eta(\xbf)$ depends on the least-norm subgradient at $\xbf$. In practice, it is sometimes possible to establish a uniform bound on $\Rcal_\eta(\xbf)$ independent of $\xbf$. Suppose that $\xbf$ is in a compact set $\Xcal \subseteq \mathrm{int}(\dom f)$. Since a weakly convex function is locally Lipschitz continuous in the interior of its domain, as stated in \Cref{lem:lip-bdd-subgrad}, local Lipschitz continuity implies that the subdifferential $\partial f(\xbf)$ is bounded. The compactness of $\Xcal$ then ensures that the subgradient norm is uniformly bounded over $\Xcal$: $\sup\{\|\partial f(\xbf)\|: \xbf\in\Xcal\}<\infty$. 
\end{rem}

\paragraph{Smoothing $h(F(\xbf))$ via the Moreau envelope}
Consider the composite nonsmooth function in~\eqref{eq:composite-form}, where $h : \mathbb{R}^m \rightarrow \mathbb{R}$ is a convex, piecewise linear function, and $F : \mathbb{R}^d \rightarrow \mathbb{R}^m$ is a smooth nonlinear mapping. Computing the Moreau envelope of $f$ in this setting requires solving the proximal problem
\[
    \min_{\xbf \in \mathbb{R}^d} \big\{ h(F(\xbf)) + \frac{\gamma}{2} \|\xbf - \hatx\|^2 \big\} 
\]
for a given reference point $\hatx \in \mathbb{R}^d$. This problem can sometimes be solved in closed form~\cite[Section 5]{davis2019stochastic}, or more generally via the prox-linear algorithm
\[
    \xbf^{k+1} = \argmin_{\xbf} \big\{ h\big(F(\xbf^k) + \nabla F(\xbf^k)(\xbf - \xbf^k)\big) + \frac{\gamma}{2} \|\xbf - \xbf^k\|^2 \big\},
\]
which admits linear convergence guarantees~\cite[Section D.3]{deng2021minibatch}. When $m$ is moderate and $h$ has a simple structure (e.g., $h(\cdot) = |\cdot|$), it is advantageous to reformulate the problem as a min-max saddle-point problem:
\[
    \min_{\xbf} \max_{\ybf} \left\{ \langle\ybf, F(\xbf)\rangle - h^{\ast}(\ybf) \right\} + \frac{\gamma}{2} \|\xbf - \hatx\|^2 = \max_{\ybf} \Big\{ - h^{\ast}(\ybf) + \min_{\xbf} \big[ \langle\ybf, F(\xbf)\rangle + \frac{\gamma}{2} \|\xbf - \hatx\|^2 \big] \Big\}.
\]
This dual reformulation enables one to solve a lower-dimensional maximization problem in $\ybf \in \mathbb{R}^m$, thus reducing computational complexity. 
In many machine learning applications, $h$ is separable, and $F$ is elementwise quadratic, so it suffices to consider the univariate case with $m = 1$ and $F(\xbf) = \tfrac{1}{2} \langle \xbf, Q\xbf \rangle + \langle \qbf, \xbf \rangle$, where $Q \in \mathbb{R}^{n \times n}$ is symmetric.  
As a concrete example, for $h(z) = |z|$,  $h^{\ast}(y) = \delta_{[a, b]}(y)$ is the indicator of interval $[a,b ]$ for some $a, b \in \mathbb{R}$. Choosing $\gamma > b\lambda_{\max}(Q)$, we have
\begin{equation}
\begin{aligned}
  & \max_{y \in [a, b]} \min_{\xbf} \left\{ y F(\xbf) + \frac{\gamma}{2}\|\xbf - \bar{\xbf}\|^2 \right\} \nonumber\\
  &= \max_{y \in [a, b]} \min_{\xbf} \left\{ \tfrac{1}{2} y \langle \xbf, Q \xbf \rangle + y \langle \qbf, \xbf \rangle + \frac{\gamma}{2}\|\xbf - \bar{\xbf}\|^2 \right\} \nonumber\\
  &= \max_{y \in [a, b]} \left\{ -\tfrac{1}{2} (\gamma \bar{\xbf} - y \qbf)^{\top} (Q y + \gamma I)^{-1} (\gamma \bar{\xbf} - y \qbf) + \frac{\gamma}{2}\|\bar{\xbf}\|^2 \right\} \nonumber\\
  & =: \max_{y \in [a, b]} \tau(y). \nonumber
\end{aligned}
\end{equation}
The solution is obtained by maximizing $\tau(y)$ over $y \in [a, b]$, which can be efficiently accomplished by finding roots of $\tau'(y) = 0$ and checking the endpoints $a$ and $b$.

\paragraph{Comparison and discussion}
We now compare the generalized Nesterov smoothing with the classical Moreau envelope smoothing for the composite function $h(A(\xbf))$. The two techniques are closely related and are, in certain cases, identical. Specifically, suppose $A$ is the identity map $A(\xbf)=\xbf$, then we have $h^*(\ybf)=f^*(\ybf)$. 
In this case, Nesterov's smoothing reduces to the infimal convolution smooth approximation. If we choose the prox-function to be the squared norm, $\omega(\ybf) = \frac{1}{2}\|\ybf\|^2$, then it can be further interpreted as the dual formulation of Moreau envelope smoothing. See \citep[Sec. 4.4]{beck2017first}.
For a general nonlinear map, the two smoothing approaches are different, revealing a fundamental trade-off: 
\begin{enumerate}[label=\arabic*), leftmargin=15pt]
\item Nesterov's smoothing is often more computationally tractable since the subproblem in \eqref{eq:nesterov-smoothing} is solved over the dual space $\dom h^*$. It can be more efficient if the structure of $h^*$ is simple. 
In contrast, computing the Moreau envelope of $f$ requires evaluating a nontrivial proximal operator, which often requires subroutines.
An exception is in stochastic optimization, where the proximal subproblem involving only a single sample sometimes admits a closed-form solution~\citep[Sec. 5]{davis2019stochastic}.
\item The advantage of the Moreau envelope smoothing is that it always yields a globally Lipschitz smooth function. In contrast, as established in \Cref{thm:nes-smoothing}, the Nesterov-smoothed function $f_\eta$ can have non-Lipschitz gradient that depends on the evaluation point. 
\end{enumerate}

\subsection{Smoothing by parts}
Beyond smoothing a single objective, the principle of smoothing can be applied in a component-wise manner to functions with more complex structures, such as sums, expectations, or nested compositions. This allows us to leverage the underlying structure and then construct valid smooth approximations for a broader class of problems.

\subsubsection{Summation function} 
Consider a function formed by the weighted sum of two nonsmooth, weakly convex functions. We can smooth each component independently; the sum of these approximations then constitutes a valid smooth approximation of the original function. 
\begin{prop} \label{prop:smooth-sum}
Let $\alpha_1, \alpha_2 \geq 0$.
Let  $f(\xbf)=\alpha_1 f_1(\xbf)+\alpha_2 f_2(\xbf)$, where $f_{1}$ and $f_{2}$ are two weakly convex functions.
Suppose $f_{i,\eta}$ is a $(\bar{\rho}_i,\Rcal_{i,\eta}, \Lcal_{i,\eta})$-SA of $f_{i}$ where $i\in\{1,2\}$. Then
$f_\eta(\xbf)=\alpha_1 f_{1,\eta}(\xbf)+\alpha_2 f_{2,\eta}(\xbf)$ is an $(\alpha_1{\bar\rho}_1+\alpha_2{\bar\rho}_2,\Rcal_\eta,\Lcal_\eta)$-SA of $f$ with $\Rcal_\eta(\xbf)\coloneqq\alpha_1\Rcal_{1,\eta}(\xbf)+\alpha_2\Rcal_{2,\eta}(\xbf)$ and $\Lcal_\eta(\xbf, \ybf)\coloneqq \alpha_1\Lcal_{1,\eta}(\xbf,\ybf)+\alpha_2\Lcal_{2,\eta}(\xbf,\ybf)$.
\end{prop}
\begin{proof}
    The proof follows from the definition.
\end{proof}
As a corollary of \Cref{prop:smooth-sum}, we can construct a smooth approximation of a finite-sum objective by smoothing each component. 

\begin{coro}\label{coro:smooth-finite-sum}
    Let $f(\xbf) = \frac{1}{m}\sum_{i=1}^m f_i(\xbf)$, where each $f_i$ is weakly convex.
    Suppose $f_{i, \eta}$ is a $(\bar{\rho}_i,\Rcal_{i,\eta}, \Lcal_{i,\eta})$-SA of $f_{i}$. Then $f_\eta(\xbf) = \frac{1}{m} \sum_{i=1}^m f_{i, \eta}(\xbf)$ is a $(\bar\rho,\Rcal_\eta,\Lcal_\eta)$-SA of $f$ with $\bar\rho\coloneqq\frac{1}{m}\sum_{i=1}^m \bar{\rho}_i $,  $\Rcal_\eta(\xbf)\coloneqq\frac{1}{m} \sum_{i=1}^m \Rcal_{i,\eta}(\xbf)$ and $\Lcal_\eta(\xbf, \ybf)\coloneqq \frac{1}{m} \sum_{i=1}^m \Lcal_{i,\eta}(\xbf,\ybf)$.
\end{coro}
\begin{proof}
The proof immediately follows from \Cref{prop:smooth-sum} and induction.
\end{proof}

\subsubsection{Expectation function}\label{subsec:expectation-func}
Consider stochastic optimization where the objective takes the following expectation form
\begin{equation}\label{eq:expectation-form}
f(\xbf)\coloneqq\Expe_{\xi}\bsbra{f(\xbf,\xi)},
\end{equation}
where $f(\cdot,\xi)$ is $\rho$-weakly convex ($\rho\ge 0$) in $\xbf$ for every outcome of the random variable $\xi$. Directly applying smoothing techniques, such as the Moreau envelope smoothing, to $f$ is often computationally challenging, as it requires solving an optimization problem that itself involves an expectation.
A more practical approach is to smooth the integrand $f(\cdot,\xi)$ and then compute the expectation.
Let $f_{\eta}(\cdot,\xi)$ be a smooth approximation of $f(\cdot,\xi)$. We then consider the smoothed objective:
\begin{equation}
{f}_{\eta}(\xbf)\coloneqq\Expe_{\xi}\bsbra{f_{\eta}(\xbf,\xi)},\label{eq:smooth-stochastic-func}
\end{equation}
assuming the integral $f_{\eta}$ is well-defined.
The following theorem shows that this procedure yields a valid smooth approximation of $f$ under standard regularity conditions.
\begin{thm}
 \label{thm:smooth-stochastic}
 Suppose for almost every $\xi$, \(f_{\eta}(\cdot,\xi)\) is a $(\bar{\rho},\Rcal_{\xi, \eta},\Lcal_{\xi,\eta})$-SA of $f(\cdot,\xi)$, where  \(\Rcal_{\xi,\eta}: \Rbb^d \to (0,\infty)\) and \(\Lcal_{\xi,\eta}: \Rbb^d  \to (0,\infty)\) are integrable stochastic functions.
Further assume that for any fixed $\xbf\in\Rbb^d$ that
\begin{enumerate}[leftmargin=15pt, label=\arabic*)]
	\item \(\Expe_\xi [|f_\eta(\xbf,\xi)|]< + \infty\),
	\item there exists a nonnegative random variable $C(\xi)$ with $\Expe_\xi[C(\xi)]<+\infty$ such that $f_\eta(\cdot,\xi)$ is Lipschitz continuous in the neighborhood of $\xbf$ with module $C(\xi)$, almost surely.
\end{enumerate}
Then, ${f}_\eta$ is differentiable with its gradient given by $\nabla {f}_\eta(\xbf)=\Expe_\xi\bsbra{\nabla f_\eta(\xbf,\xi)}$.
Let $\Rcal_\eta(\xbf) \coloneqq \Expe_{\xi} \bsbra{\Rcal_{\xi, \eta}(\xbf)}$, $\Lcal_\eta(\xbf, \ybf) \coloneqq \Expe_{\xi} \bsbra{\Lcal_{\xi,\eta}(\xbf,\ybf)}$.
Then
${f}_{\eta}$ is a $(\bar{\rho}, \Rcal_\eta, \Lcal_\eta)$-SA of $f$. 
\end{thm}
\begin{proof}
Applying \citep[Proposition 2]{ruszczynski2003optimality}, we can show \(f_\eta(\cdot,\xi)\) is locally Lipschitz and differentiable: 
\begin{equation}\label{eq:mid-11}
\nabla {f}_\eta(\xbf)=\nabla\,\Expe_\xi[f_\eta(\xbf,\xi)]=\Expe_\xi\bsbra{\nabla f_\eta(\xbf,\xi)}
.
\end{equation}
By the smooth approximation property \ref{SA-func} of $f_\eta(\cdot,\xi)$, we have  
$f_{\eta}(\xbf,\xi)\le f(\xbf,\xi)\le f_{\eta}(\xbf,\xi)+\Rcal_{\xi, \eta}(\xbf)$.
Taking expectations over $\xi$ on both sides yields
${f}_{\eta}(\xbf) \le f(\xbf) \le {f}_{\eta}(\xbf) + \Rcal_\eta(\xbf),$
which verifies the approximation bound \ref{SA-func} of ${f}_\eta$. 
Using a similar argument and the interchangeability result of expectation and differentiation~\eqref{eq:mid-11}, it is easy to show ${f}_\eta$ is $\bar\rho$-weakly convex.

In view of the smoothness condition \ref{SA-smooth} of $f_\eta(\cdot,\xi)$, we have
\begin{equation}
    \begin{aligned}
        \bnorm{\nabla {f}_\eta(\xbf)-\nabla {f}_\eta(\ybf)} & = \norm{\Expe[\nabla {f}_\eta(\xbf,\xi) -\nabla {f}_\eta(\ybf,\xi)]} \\
        & \le \Expe [\norm{\nabla {f}_\eta(\xbf,\xi) -\nabla {f}_\eta(\ybf,\xi)}] \\
        & \le \Expe [\Lcal_{\xi, \eta}(\xbf,\ybf)] \norm{\xbf-\ybf}  = \Lcal_\eta(\xbf,\ybf)\norm{\xbf-\ybf},
    \end{aligned}
\end{equation}
where the first inequality is by Jensen's inequality. This completes the proof.
\end{proof}

\subsection{More composite functions}
\paragraph{Composition of convex and weakly convex functions}
Consider a composition function of the form $f(\xbf)=h(F(\xbf))$, where the inner function is also nonsmooth. Let $h:\Rbb^m \to \Rbb_+$ be a convex function and $F:\Rbb^d \to \Rbb^m$ be a map where each component $F_i:\Rbb^d \to \Rbb$ is $\rho$-weakly convex.
We assume that $h$ is coordinate-wise non-decreasing, which means the subgradient $h'(\xbf)$ has nonnegative components. Such functions arise in penalty methods. For example, consider the penalty function $h(\zbf) = \|[\zbf]_+\|_1$.
\begin{thm}\label{thm:composite-smoothing}
Let ${h}_\eta$ be a $(0, \Rcal_{1,\eta}, \Lcal_{1,\eta})$-SA of $h$. For the map $F$, let  $F_\eta$ be a component-wise SA, such that each component $F_{i,\eta}(\xbf)$ is a $(\bar{\rho}_F, \Rcal_{2, \eta}, \Lcal_{2, \eta})$-SA of $F_i(\xbf)$.
Assume that both $h$ and its approximation $h_\eta$ are $M$-Lipschitz continuous and have non-negative (sub)gradients. Then, the composite function $f_\eta(\xbf)={h}_\eta(F_\eta(\xbf))$ is a $(\sqrt{m}\bar{\rho}_F M,\Rcal_\eta,\Lcal_\eta)$-SA of $f$ with 
\begin{equation*}
\begin{aligned}
\Rcal_\eta(\xbf) & = \Rcal_{1, \eta}(F(\xbf)) + \sqrt{m} M \Rcal_{2, \eta}(\xbf),\\
\Lcal_\eta(\xbf) & = \Lcal_{1, \eta} (F_\eta(\xbf),F_\eta(\ybf)) \min\{\norm{\nabla  F_\eta(\xbf)},\norm{\nabla F_\eta(\ybf)}\} + M \Lcal_{2, \eta}(\xbf, \ybf).
\end{aligned}
\end{equation*}
\end{thm}
\begin{proof} The proof is similar to that of \Cref{thm:nes-smoothing}.
We establish the weak convexity of the composite approximation $f_\eta(\xbf) = h_\eta(F_\eta(\xbf))$. 
Let $\mathbf{1}_m\in\Rbb^m$ be an all-one vector.
 Due to the convexity of $h_\eta$,  weak convexity of $F_\eta(\xbf)$, and the non-negativity of $\nabla h_\eta(F_\eta(\ybf))$, we have:
\begin{equation}
\begin{aligned}
f_\eta(\xbf) - f_\eta(\ybf) &\ge \langle \nabla h_\eta(F_\eta(\ybf)), F_\eta(\xbf)-F_\eta(\ybf) \rangle \\
&\ge \langle \nabla h_\eta(F_\eta(\ybf)), \nabla F_\eta(\ybf)(\xbf-\ybf) - \frac{\bar{\rho}_F}{2}\|\xbf-\ybf\|^2 \mathbf{1}_m \rangle \\
&= \langle \nabla f_\eta(\ybf), \xbf-\ybf \rangle - \frac{\bar{\rho}_F}{2}\|\xbf-\ybf\|^2 \langle \nabla h_\eta(F_\eta(\ybf)), \mathbf{1}_m \rangle \\
&\ge \langle \nabla f_\eta(\ybf), \xbf-\ybf \rangle - \frac{\sqrt{m}\bar{\rho}_F M}{2}\|\xbf-\ybf\|^2,
\end{aligned}
\end{equation}
where the last inequality follows from Cauchy's inequality: $\norm{\nabla h({F}(\ybf))^\top \onebf_m}\le \sqrt{m} M$.
Next, we show that \( f(\xbf) \ge {f}_\eta(\xbf) \). Since $h_\eta$ is element-wise increasing, i.e., \( \zbf_1 \ge \zbf_2 \) (element-wise) implies \( h_\eta(\zbf_1) \ge h_\eta(\zbf_2) \), we obtain
$h(F(\xbf)) \ge h(F_\eta(\xbf)) \ge {h}_\eta(F_\eta(\xbf)).$ To bound the approximation error, we consider
\begin{equation*}
\begin{aligned}
 f(\xbf) - f_\eta(\xbf) & = h(F(\xbf)) - {h}_\eta(F_\eta(\xbf)) \\
& = \left[ h(F(\xbf)) - {h}_\eta(F(\xbf)) \right] + \left[ {h}_\eta(F(\xbf)) - {h}_\eta(F_\eta(\xbf)) \right] \\
& \le \Rcal_{1, \eta}(F(\xbf)) + M \| F(\xbf) - F_\eta(\xbf) \| \\
& \le \Rcal_{1, \eta}(F(\xbf)) + \sqrt{m} M\Rcal_{2, \eta}(\xbf).
\end{aligned}
\end{equation*}
Finally, we establish the smoothness of \({f}_\eta(\xbf)\) using the same technique as in \Cref{thm:nes-smoothing}:
\begin{equation*}
    \begin{aligned}
        &\norm{\nabla f_\eta(\xbf)-\nabla f_\eta(\ybf)}  \\
        & =\norm{\nabla F_\eta(\xbf)^\top \nabla h_\eta(F_\eta(\xbf)) - \nabla F_\eta(\ybf)^\top \nabla h_\eta(F_\eta(\ybf))} \\
        & = \norm{\nabla F_\eta(\xbf)^\top (\nabla h_\eta(F_\eta(\xbf)) -\nabla h_\eta(F_\eta(\ybf)))  + (\nabla F_\eta(\xbf)-\nabla F_\eta(\ybf))^\top \nabla h_\eta(F_\eta(\ybf))} \\
        & \le \cbra{\Lcal_{1, \eta} (F_\eta(\xbf), F_\eta(\ybf)) \norm{\nabla F_\eta(\xbf)} + M \Lcal_{2, \eta}(\xbf, \ybf)} \norm{\xbf-\ybf}.
    \end{aligned}
\end{equation*}
This completes the proof.
\end{proof}

\paragraph{Composition of a weakly convex function and a linear map}
We now examine objective functions of the form $f(\xbf) = g(A\xbf)$, where $g$ is a $\rho$-weakly convex function and $A$ is a linear operator. This class of problems is studied in \citet{bohm2021variable}. 

\begin{thm}
Let $g_\eta$ be a $(\bar{\rho}, \Rcal_\eta, \Lcal_\eta)$-smooth approximation (SA) of $g$. Then $f_\eta$ is a $(\tilde{\rho}, \tilde{R}_\eta, \tilde{\Lcal}_\eta)$-SA of $f$ with parameters given by
\[
\tilde{\rho} = \bar{\rho}\norm{A}_\mathrm{op}^2, \quad \tilde{R}_\eta(\xbf) = \Rcal_\eta(A\xbf), \quad \tilde{\Lcal}_\eta(\xbf, \ybf) = \norm{A}_\mathrm{op} \Lcal_\eta(A\xbf, A\ybf).
\]
\end{thm}

\begin{proof}
By the weak convexity of $g_\eta$, it holds that
\begin{equation*}
    \begin{aligned}
        f_\eta(\xbf) - f_\eta(\ybf)
        & = g_\eta(A\xbf) - g_\eta(A\ybf) \\
        & \ge \nabla g_\eta(A\ybf)^\top (A\xbf - A\ybf) - \frac{\bar{\rho}}{2} \|A\xbf - A\ybf\|^2 \\
        & \geq \nabla f_\eta(\ybf)^\top (\xbf - \ybf) - \frac{\bar{\rho}\norm{A}^2_\mathrm{op}}{2}\|\xbf - \ybf\|^2,
    \end{aligned}
\end{equation*}
where we used the chain rule $\nabla f_\eta(\ybf) = A^\top \nabla g_\eta(A\ybf)$ and the fact that $\|A\xbf - A\ybf\| \le \|A\|_\mathrm{op}\|\xbf - \ybf\|$. 

Moreover, by the SA property of $g_\eta$, we obtain
\[
f_\eta(\xbf) = g_\eta(A\xbf) \le g(A\xbf) \le g_\eta(A\xbf) + \Rcal_\eta(A\xbf)=f_\eta(\xbf)+\Rcal_\eta(\xbf),
\]
and the Lipschitz continuity of the gradient satisfies
\[
\|\nabla f_\eta(\xbf) - \nabla f_\eta(\ybf)\| = \|A^\top (\nabla g_\eta(A\xbf) - \nabla g_\eta(A\ybf))\| \le \|A\|_\mathrm{op} \Lcal_\eta(A\xbf, A\ybf)\|\xbf - \ybf\|.
\]
This completes the proof.
\end{proof}
\section{Algorithms for smooth approximation}
\label{sec:algorithms-for-smooth-approximation}

This section discusses how to design algorithms based on the smooth approximation theory developed so far. We consider the setting where $f$ is given by the expectation of a stochastic function, namely, $f(\xbf)=\Expe_{\xi}\bsbra{f(\xbf,\xi)}$, with $f$ and $f(\cdot,\xi)$ defined as in \Cref{subsec:expectation-func}. To address such problems, it is natural to employ optimization algorithms for smooth functions to solve an appropriately constructed surrogate problem~\eqref{pb:comp-smooth}. 
 Suppose $f_\eta$ is a smooth approximation of $f$ satisfying the assumptions in \Cref{thm:smooth-stochastic}. The resulting smoothed optimization problem takes the form
\begin{equation}
\min_{\xbf\in\Rbb^{d}}\ \phi_\eta(\xbf)\coloneqq f_\eta(\xbf)+r(\xbf),\quad\text{where }f_\eta(\xbf)=\Expe_{\xi}\bsbra{f_\eta(\xbf,\xi)}.\label{pb:sa-smoothing}
\end{equation}
To elucidate the intuition behind our algorithm design, we introduce, for this section only, the following simplifying assumption on $\Rcal_\eta$ and $\Lcal_\eta$.

\begin{assumption}\label{assum:standard-SA}
Suppose \Cref{df:smooth-apprx} holds and that there exists $\theta > 0$ such that for all $\eta\in(0,\theta]$,
    \begin{equation}
    \Rcal_\eta (\xbf) \equiv R_\eta\coloneqq R\cdot \eta, \quad \Lcal_\eta(\xbf,\ybf)\equiv L_\eta\coloneqq B+\frac{L}{\eta}, \label{eq:standard-smoothing}
    \end{equation}
    for some constants $R, B, L \ge 0$. Furthermore, we assume that the target accuracy $\vep \ll\theta$.
\end{assumption}
Condition~\eqref{eq:standard-smoothing} is satisfied, for example, when $\dom r$ is a compact subset in the interior of $\Xcal$; see \Cref{remark:moreau-smoothing-bound} and \Cref{remark:nesterov-smoothing-bound} for further discussions. We restrict $\eta$ to a bounded interval since, in general, $\Rcal_\eta$ may not be linear with respect to $\eta$ over $\mathbb{R}$ (as occurs, for instance, with Moreau envelope smoothing). This restriction to $\eta \in (0, \theta]$ is mild, since our analysis ultimately adopts a choice of $\eta = \Ocal(\vep^{2})$. Note that while we assume $\Rcal_\eta$ and $\Lcal_\eta$ are independent of the variables $\xbf$ and $\ybf$, this assumption will be relaxed to more general settings in which these quantities depend on the variables, as discussed in \Cref{sec:generalized-smooth}. The following assumption regarding the bounded variance of the gradient is standard in the literature \cite{lan2020first}.
\begin{assumption} \label{asp:noise}
$\Expe_\xi[\norm{\nabla f_\eta(\xbf,\xi)-\nabla f_\eta(\xbf)}^2]\le \sigma^2$, for any $\xbf\in \dom r$.
\end{assumption}
For the rest of this section, we assume that both \Cref{assum:standard-SA,asp:noise}, and the assumptions in \Cref{thm:smooth-stochastic} are satisfied. 

\subsection{Smooth minimization with gradient-based methods} \label{sec-gradient-based-smoothing}

Given the smooth nonconvex problem \eqref{pb:sa-smoothing}, one can directly apply proximal gradient-type methods. \Cref{alg:Smoothed-spg} presents the stochastic
proximal gradient method for solving \eqref{pb:sa-smoothing}. We adopt a minibatch
strategy, drawing a batch of $m$ i.i.d.\ samples at each iteration to form the
gradient estimator $\gbf^{k-1}$. When $\sigma = 0$, the problem reduces to the
deterministic setting, in which case \Cref{alg:Smoothed-spg} recovers the standard
proximal gradient method.
\begin{algorithm}[h]
\caption{The smoothing stochastic proximal gradient method (\sspg)\protect}
\label{alg:Smoothed-spg}
\KwIn{Initial point $\xbf^{0}$, stepsize $\gamma > 0$;}
\For{$k=1,2,3,\ldots$}{
Sample minibatch $\{\xi^{k-1}_1,\ldots,\xi^{k-1}_m\}$ and compute $\gbf^{k-1}=\frac{1}{m}\sum_{i=1}^m \nabla f_\eta(\xbf^{k-1},\xi^{k-1}_i) $\;
  $\xbf^k= \prox_{\gamma r}(\xbf^{k-1}-\gamma \gbf^{k-1})$\;
} 
\end{algorithm}

There are two ways to analyze \Cref{alg:Smoothed-spg}. One follows the analysis in \citet{ghadimi2016mini}, which can be interpreted as a perturbed proximal gradient method. The batch size $m$ needs to be chosen sufficiently large to control the stochastic approximation error.
\begin{thm}\label{thm:rate-smooth-sgd}
Suppose we set $\gamma=1/(2L_{\eta})$ in \Cref{alg:Smoothed-spg}. Then, the iterates satisfy
\[
\min_{0\leq k \leq K-1} \Ebb[\|\Gcal_{\gamma}(\xbf^{k})\|^{2}] \leq \frac{8L_{\eta}\bigl(\Delta + R\eta\bigr)}{K} + \frac{6\sigma^{2}}{m},
\]
where $\Delta \geq \phi(\xbf^{0}) - \min_{\xbf} \phi(\xbf)$ and the generalized gradient $\mathcal{G}_\gamma$ is defined in \eqref{eq:generalized-grad}.
\end{thm}

\paragraph{Deterministic optimization ($\sigma=0$).}
When $\sigma=0$, \Cref{alg:Smoothed-spg} reduces to the proximal gradient method. 
In order to obtain an $(\vep,\vep)$-stationary point of \eqref{pb:main}, according to \Cref{thm:smooth-criteria-convert}, we would need $\eta=\Ocal(\vep^2)$.
Under this setting, \Cref{thm:rate-smooth-sgd} establishes an $\Ocal(1/\vep^4)$ complexity bound. This complexity matches, but does not improve upon, the worst-case complexity of the proximal subgradient method~\citep{davis2019stochastic}. The lack of improvement is attributed to the poor conditioning (large $L_\eta$) introduced by smoothing.

\paragraph{Stochastic optimization ($\sigma>0$).} The deterministic analysis suggests that in the stochastic setting, the total iteration number will be at least $\Ocal(1/\vep^{4})$. 
Even worse, a direct application of \Cref{thm:rate-smooth-sgd} requires a large minibatch size of $m = \Ocal({\sigma^2}/\vep^2)$ to control the gradient variance. Combined with the iteration requirement $K=\Ocal (L_\eta/\vep^2)$, the sample complexity becomes
\[
mK =  \Ocal\Big(\frac{\sigma^2}{\vep^2}\Big) \cdot \Ocal \Big(\frac{B+L/\vep^2}{\vep^2}\Big) = \Ocal\Big(\frac{1}{\vep^6}\Big).
\]
This complexity is even inferior to the $\Ocal(1/\vep^4)$ result of the standard stochastic subgradient method. Fortunately, this suboptimal bound is an artifact of the analysis. A sharper complexity can be derived by analyzing the algorithm through the lens of the Moreau envelope~\citep{davis2019stochastic}, which allows an arbitrary batch size. Following the notation from the previous sections, we denote the Moreau envelope of $\phi_\eta$ by \[
\phi_\eta^{\hat\rho}(\xbf)  \coloneqq \min_\zbf \Big\{\phi_\eta(\zbf)+\frac{\hat\rho}{2}\norm{\zbf-\xbf}^2 \Big\}.
\]
By leveraging the Moreau envelope as a potential function and following the analysis of \citet{davis2019stochastic}, we establish the following convergence guarantee.
\begin{thm}\label{thm:moreau-analysis}
Let $\hat{\rho}>\bar{\rho}$ and  $\gamma<(\hat{\rho}+L_{\eta})^{-1}$ in \Cref{alg:Smoothed-spg}. Then we have
\[
\frac{1}{K}\sum^{K-1}_{k=0}\Ebb[\norm{\nabla\phi_\eta^{\hat\rho}(\xbf^{k})}^{2}]\le\frac{\hat{\rho}(\hat{\rho}-\bar{\rho}+\gamma^{-1})}{\hat{\rho}-\bar{\rho}}\frac{\Ebb[\phi_\eta^{\hat\rho}(\xbf^{0})-\phi_\eta^{\hat\rho}(\xbf^{K})]}{K}+\frac{\hat{\rho}^{2}}{2(\hat{\rho}-\bar{\rho})}\frac{\sigma^{2}}{(\gamma^{-1}-\hat{\rho}-L_{\eta})m}.
\]
Suppose $\Delta\ge\phi(\xbf^{0})-\min_{\xbf}\phi(\xbf)$. Let $m=1$,
$\gamma=(c\sqrt{K}+\hat{\rho}+L_{\eta})^{-1}$ where $c=\sqrt{\frac{\hat{\rho}}{2(\Delta+R\eta)}}\sigma$, then 
\begin{equation}\label{eq:rate-moreau}
\min_{0\le k \le K-1} \Ebb[\norm{\nabla\phi_\eta^{\hat\rho}(\xbf^k)}^{2}]\le \frac{\hat{\rho}}{\hat{\rho}-\bar{\rho}}\Big\{ \frac{(2\hat{\rho}-\bar{\rho}+L_{\eta})(\Delta+R\eta)}{K}+\sqrt{\frac{2\hat{\rho}(\Delta+R\eta)}{K}}\sigma\Big\} .
\end{equation}
\end{thm}
\begin{rem}
 To guarantee that the expected squared gradient norm in \eqref{eq:rate-moreau} is bounded by $\Ocal(\vep^2)$, we need to select the number of iterations $K$ sufficiently large so that both terms on the right-hand side of \eqref{eq:rate-moreau} are controlled. Choosing $\eta = \Theta(\vep^2)$ and $m=1$, the total sample complexity becomes
\[
mK = 1 \cdot \Ocal\bigg(\max\Big\{\frac{2\hat{\rho}-\bar{\rho}+L_\eta}{\vep^2}, \frac{\sigma^2}{\vep^4}\Big\}\bigg) = \Ocal\Big(\frac{1}{\vep^4}\Big),
\]
which matches the sample complexity bound achieved by the stochastic subgradient method~\citep{davis2019stochastic}.
\end{rem}

The preceding analyses indicate that straightforward gradient-based smoothing methods
can match, but cannot improve upon, existing complexity bounds. To overcome the
limitations caused by poor problem conditioning, we resort to Nesterov's acceleration.
Specifically, under a proximal point scheme, the original problem is reduced to a
sequence of convex subproblems whose condition numbers are on the order of
$\Ocal(1/\varepsilon^{2})$. These subproblems can then be solved using accelerated
gradient descent, which requires only $\Ocal(1/\sqrt{\varepsilon^{2}})=\Ocal(1/\varepsilon)$
iterations. Combined with the $\Ocal(1/\varepsilon^{2})$ complexity of the outer
proximal-point updates, it yields an overall complexity of $\Ocal(1/\varepsilon^{3})$. The same proximal-point strategy extends naturally to the stochastic setting. Although
the sample complexity is not improved in the worst case, our sharper analysis
demonstrates that smoothing provides clear advantages in minibatching regimes.

\subsection{Smooth minimization with inexact proximal point} \label{sec:smoothing-proxpt}
We see from \Cref{sec-gradient-based-smoothing} that directly applying smoothing
 does not surpass the complexity of the subgradient method \citep{davis2019stochastic}.
The main issue is that the simple gradient descent algorithm is ineffective in dealing with the ill-conditioned smoothed problem arising from the asymmetry between the lower curvature $\bar{\rho}$ and the upper curvature $L_\eta$.  
To further improve the convergence rate, we propose solving the smooth problem using the proximal point method~(\Cref{alg:Smoothed-ipp}), which turns the original problem into a sequence of strongly convex subproblems.

\begin{algorithm}[h]
\caption{The smoothed inexact proximal point method (\sipp)\protect}
\label{alg:Smoothed-ipp} \KwIn{Initial point $\xbf^{0}$, parameter $\hat{\rho} > \bar{\rho}$;}
\For{$k=1,2,3,\ldots$}{
Compute an approximate solution $\xbf^k$ to the following problem:
\begin{equation}\label{eq:pp-subprob}
\xbf^k \approx \argmin_\xbf\, \Big\{ \phietaenvk(\xbf) \assign \phi_{\eta}(\xbf)+\frac{\hat{\rho}}{2}\norm{\xbf-\xbf^{k-1}}^{2} \Big\},
\end{equation}
such that condition~\eqref{eq:sub-criteria-1} is satisfied.
} 
\end{algorithm}

The subproblems \eqref{eq:pp-subprob} can be solved efficiently with acceleration. For convenience, we denote 
\[\phietaenvk(\xbf)\coloneqq{\phi_\eta}(\xbf)+\frac{\hat{\rho}}{2}\norm{\xbf-\xbf^{k-1}}^{2} \quad \text{and} \quad \hatx^k\coloneqq\prox_{{\phi_\eta/\hat{\rho}}}(\xbf^{k-1}) = \argmin_{\xbf} \phietaenvk(\xbf).\]  
Let ${\xbf}^{\star}_{\eta}$ be an optimal solution to the smoothed problem.
We assume that the proximal subproblem \eqref{eq:pp-subprob} is inexactly solved: the suboptimality of $\xbf^{k}$ consists of a relative error to the initial suboptimality plus an absolute error term:
\begin{equation}
\Expe_{k}\bsbra{\phietaenvk(\xbf^{k})-\phietaenvk(\hatx^{k})}\le\lambda\bsbra{\phietaenvk(\xbf^{k-1})-\phietaenvk(\hatx^{k})}+\zeta_{k},\quad\text{where }\lambda\in[0,1),\ \zeta_{k}\in[0,\infty).\label{eq:sub-criteria-1}
\end{equation}
Here, $\Expe_{k}[\cdot]$ is short for the conditional expectation $\Expe[\cdot|\Fcal_k]$, where $\Fcal_k$ denotes the $\sigma$-algebra generated by $\xbf^{1},\xbf^{2},\ldots,\xbf^{k-1}$. Let 
$\zeta_{k}$ be the stochastic error determined by $\Fcal_k$.
The bound~\eqref{eq:sub-criteria-1}  characterizes the convergence rate of many first-order methods. For instance, in smooth convex optimization, many gradient-based methods guarantee
the shrinkage of the optimality gap relative to the initial one. In stochastic optimization, the factor $\zeta_{k}$ often accounts for the error due to stochastic noise. The following theorem develops the main convergence property of \Cref{alg:Smoothed-ipp}. 
\begin{thm}
\label{thm:sipp-rate}
Let $\hat{\rho} > \bar{\rho}$ in \Cref{alg:Smoothed-ipp}. Then we have
\[
\min_{1\le k\le K}\Expe \bsbra{\norm{\nabla \phi_\eta^{\hat\rho} (\xbf^{k})}^{2}} \le \frac{2\hat\rho^2}{\hat\rho-\bar\rho} \frac{(1+\lambda)\sbra{\phi(\xbf^{0})-\phi(\xbf^{\star})+R\eta}+\sum_{k=1}^{K}\Expe[\zeta_{k}]}{(1-\lambda)K}.
\]
\end{thm}
\begin{proof}
Using the definition of $\phietaenvk$ and $\hatx^{k}$,
we have the lower-bound 
\vspace{-0.3\baselineskip}
\begin{align}
\phietaenvk(\hatx^{k}) & \le\phi_{\eta}(\xbf^{k-1})+\frac{\hat{\rho}}{2}\norm{\xbf^{k-1}-\xbf^{k-1}}^{2}=\phietaenvk(\xbf^{k-1}),\label{eq:phi-tilde-1}\\
\phietaenvk(\hatx^{k}) & \le{\phi_\eta}({\xbf}^{\star}_{\eta})+\frac{\hat{\rho}}{2}\norm{{\xbf}^{\star}_{\eta}-\xbf^{k-1}}^{2},\label{eq:phi-tilde-2}
\end{align}
\vspace{-40pt}

\noindent and the upper bound 
\begin{equation}
\phietaenvk(\hatx^{k})\ge\min_{\xbf}\,\phi_\eta(\xbf)={\phi_\eta}({\xbf}^{\star}_{\eta}).\label{eq:phi-tilde-3}
\end{equation}
It follows that
\[
\begin{aligned} 
\phietaenvk(\xbf^{k-1})-\phietaenvk(\hatx^{k})
 & ={\phi_\eta}(\xbf^{k-1})-\phietaenvk(\hatx^{k})\\
 & \le\hat{\phi}_{k-1}(\xbf^{k-1})-\phietaenvk(\hatx^{k})\\
 & =\hat{\phi}_{k-1}(\hatx^{k-1})-\phietaenvk(\hatx^{k})+\hat{\phi}_{k-1}(\xbf^{k-1})-\hat{\phi}_{k-1}(\hatx^{k-1})
\end{aligned}
\]
for $k\ge 2$. 
Taking conditional expectation and applying \eqref{eq:sub-criteria-1}, we arrive at 
\[
\Expe_{k-1}\bsbra{\phietaenvk(\xbf^{k-1})-\phietaenvk(\hatx^{k})}\le\hat{\phi}_{k-1}(\hatx^{k-1})-\Ebb_{k-1}[\phietaenvk(\hatx^{k})]+\lambda\bsbra{\hat{\phi}_{k-1}(\xbf^{k-2})-\hat{\phi}_{k-1}(\hatx^{k-1})}+\zeta_{k-1}.
\]
Summing up the above bound for $k=2,3,\ldots, K+1$, and taking the expectation
over all the randomness,
\[
\sum_{k=2}^{K+1}\Expe\bsbra{\phietaenvk(\xbf^{k-1})-\phietaenvk(\hatx^{k})}\le\hat{\phi}_{1}(\hatx^{1})-\Expe[\hat{\phi}_{K+1}(\hatx^{K+1})]+\lambda\sum_{k=1}^{K}\Expe\bsbra{\phietaenvk(\xbf^{k-1})-\phietaenvk(\hatx^{k})}+\sum_{k=1}^{K}\Expe[\zeta_{k}].
\]
Subtracting $\lambda\sum_{k=2}^{K}\Expe\bsbra{\phietaenvk(\xbf^{k-1})-\phietaenvk(\hatx^{k})}$
on both sides, we have 
\begin{equation}\label{eq:opt-13}
\begin{aligned}
 & (1-\lambda)\sum_{k=2}^{K+1}\Expe\bsbra{\phietaenvk(\xbf^{k-1})-\phietaenvk(\hatx^{k})}  \\
 & \le\hat{\phi}_{1}(\hatx^{1})-\Expe\bsbra{\hat{\phi}_{K+1}(\hatx^{K+1})}+\lambda\Expe\bsbra{\hat{\phi}_{1}(\xbf^{0})-\hat{\phi}_{1}(\hatx^{1})}+\sum_{k=1}^{K}\Expe[\zeta_{k}] \\
 & \le\hat{\phi}_{1}(\hatx^{1})-{\phi}_\eta({\xbf}^{\star}_{\eta})+\lambda\bsbra{\hat{\phi}_{1}(\xbf^{0})-{\phi_\eta}({\xbf}^{\star}_{\eta})}+\sum_{k=1}^{K}\Expe[\zeta_{k}] \\
 & \le(1+\lambda)\bsbra{{\phi_\eta}(\xbf^{0})-{\phi_\eta}({\xbf}^{\star}_{\eta})}+\sum_{k=1}^{K}\Expe[\zeta_{k}]
\end{aligned}
\end{equation}
where the second inequality uses \eqref{eq:phi-tilde-3} and the last one uses the fact that $\hat{\phi}_{1}(\hatx^{1})\le\hat{\phi}_{1}(\xbf^{0})={\phi_\eta}(\xbf^{0})$.
Moreover, since $\phietaenvk$ is $(\hat{\rho}-\bar\rho)$-strongly convex, we have
\begin{equation}
\norm{\hatx^{k}-\xbf^{k-1}}^{2}\le\frac{2}{\hat{\rho}-\bar\rho}\bsbra{\phietaenvk(\xbf^{k-1})-\phietaenvk(\hatx^{k})}.\label{eq:opt-14}
\end{equation}
Putting \eqref{eq:opt-13} and \eqref{eq:opt-14} together and noticing $\norm{\nabla \phi_\eta^{\hat\rho}(\xbf^{k-1})}=\hat\rho\norm{\hatx^{k}-\xbf^{k-1}}$, we have
\[
\sum_{k=2}^{K+1}\Expe[\norm{\nabla \phi_\eta^{\hat\rho}(\xbf^{k-1})}^{2}] 
=\hat\rho^2\sum_{k=2}^{K+1}\Ebb[\norm{\hatx^{k}-\xbf^{k-1}}^{2}]
\le \frac{2\hat\rho^2}{\hat\rho-\bar\rho}\frac{(1+\lambda)\sbra{{\phi}_\eta(\xbf^{0})- {\phi_\eta}({\xbf}^{\star}_{\eta})}+\sum_{k=1}^{K}\Expe[\zeta_{k}]}{(1-\lambda)}.
\]
In view of the definition of smooth approximation, we have 
\[
{\phi_\eta}(\xbf^{0})-{\phi_\eta}({\xbf}^{\star}_{\eta})\le\phi(\xbf^{0})-{\phi_\eta}({\xbf}^{\star}_{\eta})\le\phi(\xbf^{0})-\phi(\xbf^\star)+R\eta.
\]
Combining the above two inequalities gives the desired result.
\end{proof}

Next, we consider solving the proximal subproblems in \eqref{pb:sa-smoothing}. We employ accelerated stochastic gradient descent~\citep[(4.2.5)-(4.2.7)]{lan2020first}. 
It is a stochastic variant of the accelerated gradient descent that obtains the optimal $\Ocal(L/K^2 + \sigma^2/\sqrt{K})$ convergence rate for stochastic and $L$-smooth convex optimization problems. When $\sigma=0$, the algorithm reduces to the optimal gradient method with Nesterov's acceleration.

We refer to \Cref{alg:Smoothed-ipp} with the accelerated stochastic gradient method as a subroutine for solving the proximal subproblems as the {\asgdsipp} method. Specifically, at each iteration $k$ of {\sipp}, the accelerated method is initialized at $\xbf^{k-1}$ and executed for $T_k$ iterations to obtain an approximate solution to the subproblem~\eqref{eq:pp-subprob}. 
Observe that in \eqref{eq:pp-subprob}, the objective function $f_\eta(\xbf) + \frac{\hat{\rho}}{2}\|\xbf-\xbf^{k-1}\|^2$ is $(L_\eta+\hat\rho)$-smooth and $(\hat\rho-\bar\rho)$-strongly convex.
According to~\citep[Prop. 4.6]{lan2020first}, the output $\xbf^k$ produced by \texttt{ASGD} satisfies the following inequality:
\begin{equation}
\begin{aligned}
\Ebb_k\bsbra{\phietaenvk(\xbf^{k})-\phietaenvk(\hatx^{k})} 
    & \le \frac{2(L_\eta+\hat\rho)}{T_k(T_k+1)} \norm{\xbf^{k-1}-\hatx^k}^2
+\frac{4 \sigma^2}{(\hat\rho-\bar\rho)(T_k+1)}\\
 & \le \frac{4(L_\eta+\hat\rho)}{(\hat\rho-\bar\rho)T_k^2} \bsbra{\phietaenvk(\xbf^{k-1})-\phietaenvk(\hatx^{k})} +\frac{4\sigma^2}{(\hat\rho-\bar\rho)T_k},
\end{aligned}
\label{eq:sub-criteria-acc-stoc}
\end{equation}
for solving the proximal point subproblem. In view of \eqref{eq:sub-criteria-acc-stoc} and \Cref{thm:sipp-rate}, we have the sample complexity bound for {\asgdsipp} as follows:
\begin{thm}[Complexity of {\asgdsipp}]
    Let $\hat\rho=2\bar\rho$, $\eta=\vep^{2}$, and define
    \[
    T_k=\max\Big\{2\sqrt{\frac{2(B+L/\eta+2\bar\rho)}{\bar\rho}},\ \frac{8\sigma^2}{\vep^2}\Big\}.
    \]
    Then, the total number of stochastic gradient computations required to obtain an $\vep$-approximate stationary point of problem~\eqref{pb:main} is bounded by
    \[
    \mathcal{O}\Big(
        \frac{\, \bar\rho\left[\phi(\xbf^0) - \phi(\xbf^\star) + R\eta\right]}{\vep^2}
        \cdot
        \max\Big\{
            \sqrt{\frac{B + L/\vep^2 + 2\bar\rho}{\bar\rho}},\ \frac{\sigma^2}{\vep^2}
        \Big\}
    \Big).
    \]
\end{thm}
\begin{proof}
The construction of $T_k$ ensures that $\frac{4(L_\eta+\hat\rho)}{\bar\rho T_k^2}\le \frac{1}{2}=\lambda$ and  $\zeta_k=\frac{4\sigma^2}{\bar\rho T_k} \le \frac{\vep^2}{2\bar\rho}$. 
    Applying \Cref{thm:sipp-rate}, we have
\begin{equation}
\begin{aligned}
\min_{1\le k\le K}\Expe[\norm{\nabla{\phi}_\eta^{\hat\rho}(\xbf^{k})}^{2}] 
& \le\frac{24\bar\rho\sbra{\phi(\xbf^{0})-\phi({\xbf}^{\star})+R\eta}}{K} + 8 \vep^2. \\ 
\end{aligned}
\end{equation}
Therefore, it requires at most $K=\frac{3\bar\rho\sbra{\phi(\xbf^{0})-\phi({\xbf}^{\star})+R\eta}}{\vep^{2}}$ iterations of \sipp{} to have the error bound $\min_{1\le k\le K}\Expe [\norm{\nabla{\phi}_\eta^{\hat\rho}(\xbf^{k})}]\le 4\vep$.
The total number of stochastic gradient computations is bounded by:
\[
N=\sum_{k=1}^{K}T_{k}=\mathcal{O}\left(
\frac{\bar\rho\bsbra{\phi(\xbf^{0})-\phi({\xbf}^{\star})+R\eta}}{\vep^{2}} \cdot
\max\left\{\sqrt{\frac{B+L/\vep^2+2\bar\rho}{\bar\rho}}, \frac{\sigma^2}{\vep^2}\right\}\right) = \Ocal \Big(\max \Big\{\frac{1}{\varepsilon^3}, \frac{\sigma^2}{\varepsilon^4}\Big\}\Big).
\]
\end{proof}
\begin{rem}
    As an alternative to the smoothing-based approach, one may employ the proximal subgradient method~\citep{davis2019stochastic}, which achieves a sample complexity of 
    $   
    \mathcal{O}( {M^2}/{\vep^4} )$,
    where $M$ denotes the Lipschitz constant of $f$. This complexity bound is independent of the batch size. 
    Note that in the minibatch setting (with batch size $m$), the smoothing-based method achieves a sample complexity of
    $\mathcal{O}(\max\{ {1}/{\vep^3}, {\sigma^2}/(m\vep^4)\} )$.
    It replaces the dependence on $M^2$ with $\sigma^2$ and allows minibatching to achieve variance reduction. 
A prior work by \citet{deng2021minibatch} also considered minibatching for
weakly convex optimization; however, their minibatch prox-linear and proximal point
algorithms require solving increasingly complex proximal subproblems as batch size $m$ increases. In contrast, our
smoothing-based approach yields simpler proximal subproblems that can be applied to
each sample independently. \Cref{tab:complexity} summarizes the resulting sample complexity guarantees.

\begin{table}[h!]
\caption{Sample complexities of minibatch algorithms ($m$: batch size)}
\centering
\begin{tabular}{@{}lccc@{}}
\toprule
\textsf{Algorithm} & \textsf{Subgradient~\citep{davis2019stochastic}} & {\asgdsipp} & \textsf{Model-based~\citep{deng2021minibatch}} \\
\midrule
\textsf{Complexity} & 
$\mathcal{O}( \frac{M^2}{\varepsilon^4})$ &
$\mathcal{O}\big(\max\big\{\frac{1}{\varepsilon^3}, \frac{\sigma^2}{m\varepsilon^4}\big\}\big)$ &
$\mathcal{O}\big(\max\big\{\frac{1}{\varepsilon^2}, \frac{M^2}{m\varepsilon^4}\big\}\big)$ \\
\bottomrule
\end{tabular}
\end{table}
\end{rem}

\begin{rem}
Instead of using inexact proximal point, it is also feasible to adopt other algorithms that can exploit the asymmetry of the function's upper/lower curvature. For example, adopting the parameter-free algorithm from \cite{lan2023optimal} yields an
$\Ocal\big(\frac{\sqrt{L_\eta \bar{\rho} }}{\varepsilon^2}\big) = \Ocal(\frac{1}{\varepsilon^3})$ complexity without knowing $\bar{\rho}$.
\end{rem}

\section{Smoothing algorithms for generalized smooth problems} \label{sec:generalized-smooth}
Building on the intuitions from \Cref{sec:algorithms-for-smooth-approximation}, this section relaxes \Cref{assum:standard-SA} and considers the setting of generalized smooth approximations where $\Rcal$ and $\Lcal_\eta$ are not constants. 
We continue to utilize the inexact proximal point method framework, which iteratively solves a sequence of generalized smooth subproblems. Notably, while the subproblems remain convex, they no longer exhibit global Lipschitz smoothness. 
Therefore, a key step is to develop an accelerated algorithm tailored to these generalized smooth subproblems. For simplicity, we focus on deterministic optimization and leave the stochastic setting to the future work. 
In the remainder of this section, we introduce an accelerated gradient method with line search for convex generalized smooth optimization problems. This algorithm will be used as a subroutine within the inexact proximal point framework for minimizing generalized smooth approximation of non-Lipschitz weakly convex functions.

\subsection{An accelerated method for generalized smooth convex optimization}
We begin by presenting a variant of accelerated gradient descent tailored for convex generalized smooth optimization problems.
Consider the following convex program:
\[
\min_{\xbf}\, \psi(\xbf) \coloneqq g(\xbf) + \pi(\xbf),
\]
where $\pi$ is $\mu$-strongly convex (with $\mu \geq 0$), lower semi-continuous, and $g$ is convex and generalized smooth. Specifically, $g$ satisfies the condition:
\begin{equation}\label{eq:g(x)-gensmooth}
\norm{\nabla g(\ybf) - \nabla g(\xbf)} \leq \Lcal(\xbf, \ybf) \norm{\ybf - \xbf},
\end{equation}
where $\Lcal: \Rbb^d \times \Rbb^d \rightarrow (0, \infty)$ is a symmetric non-negative continuous map. 

In the classical accelerated gradient methods, achieving sufficient descent typically requires a global upper bound on the problem's curvature. However, this requirement is not satisfied when $\Lcal(\xbf, \ybf)$ varies with the decision variables. To overcome this issue, we apply a variant of the accelerated gradient method that employs an adaptive line search, as described in \Cref{alg:AGM-LS}. 
Additionally, we incorporate an early stopping criterion, enabling the algorithm to serve effectively as a subroutine within the proximal point framework for nonconvex optimization problems.

\begin{algorithm}[h!]
\KwIn{$\xbf^{0}$, $\mu\ge 0$, $L_{0}\in(0,+\infty)$,
$\tau_\mathrm{d}\in(0,1)$, $\tau_\mathrm{u}\in(1,+\infty)$;}

set $\ybf^{0}=\zbf^{0}=\xbf^{0}$, $\hat{L}_{0}=L_{0}$\;
\For{$t=1,2,\ldots,T$}{

\textbf{Line search:} Let $\bar{L}_{t}=\tau_\mathrm{d}\hat{L}_{t-1}$ and $k_{t}$ be the smallest
number such that $\xbf^{t}$, $\ybf^{t}$, $\zbf^{t}$ and $\hat{L}_{t}=\bar{L}_{t}\tau_\mathrm{u}^{k_{t}}$
satisfy: 
\begin{align}
\xbf^{t} & =(1-\alpha_{t})\ybf^{t-1}+\alpha_{t}\brbra{(1-\beta_{t})\ybf^{t-1}+\beta_{t}\zbf^{t-1}}, \label{eq:step-xbf} \\
\zbf^{t} & =\argmin_{\xbf} \big\{ \langle \nabla g(\xbf^{t}), \xbf \rangle+\pi(\xbf)+\frac{\gamma_{t}}{2}\norm{\xbf-\zbf^{t-1}}^{2} \big\}, \label{eq:step-prox} \\
\ybf^{t} & =(1-\alpha_{t})\ybf^{t-1}+\alpha_{t}\zbf^{t}, \nonumber 
\end{align}
and 
\begin{equation}
g(\ybf^{t})\le g(\xbf^{t})+\inner{\nabla g(\xbf^{t})}{\ybf^{t}-\xbf^{t}}+\frac{\hat{L}_{t}}{2}\norm{\ybf^{t}-\xbf^{t}}^{2}, \label{eq:ls-descent}
\end{equation}
where $\alpha_{t}$, $\beta_{t}$ and $\gamma_{t}$ satisfy \eqref{eq:step-1},
\eqref{eq:step-2}, and \eqref{eq:step-3}:
\begin{align}
\hat{L}_{t}(1-\beta_{t})\alpha_{t}^ {} & \le(1-\alpha_{t})\mu, & t \ge 1,\label{eq:step-1}\\
\hat{L}_{t}\beta_{t}\alpha_{t} & \le\gamma_{t}, &  t\ge 1,\label{eq:step-2}\\
\frac{\alpha_{t}}{1-\alpha_{t}}\gamma_{t} &  = \alpha_{t-1}(\gamma_{t-1}+\mu), \quad &  t\ge 2.\label{eq:step-3}
\end{align}
\textbf{Early stop (optional):}   \lIf{$\prod_{i=1}^t\brbra{1-\sqrt{\frac{\mu}{\hat{L}_i+\mu}}}\le \frac{1}{4}$}{\Return $\ybf^t$}
}
\KwOut{$\ybf^T$.}
\caption{Accelerated gradient method with line search (\agls)}\label{alg:AGM-LS}
\end{algorithm}

Our convergence analysis proceeds in three main steps. We first establish a general
convergence property under the assumption that the line search procedure terminates and that all iterates are well defined. We then use an induction argument to prove the boundedness of the iterates, which in turn guarantees the validity of the line search at
every iteration. Finally, we derive explicit convergence rates for both convex and strongly convex cases.

\begin{prop}\label{prop:convergence-bound}
Suppose $g$ is convex, generalized smooth, and that $\pi$ is $\mu$-strongly convex. If the line search in \Cref{alg:AGM-LS} succeeds for all $t=1,...,T$, then the following inequality holds:
\begin{equation}
\Gamma_{T}\Delta_{T}+\frac{\Gamma_{T}\alpha_{T}(\gamma_{T}+\mu)}{2}\norm{\zbf^{T} - {\xbf}^{\star}}^{2}\le(1-\alpha_{1})\Delta_{0}+\frac{\alpha_{1}\gamma_{1}}{2}\norm{\zbf^{0} - {\xbf}^{\star}}^{2},\label{eq:rec-bound-1}
\end{equation}
where $\ensuremath{\Gamma_{t}=\begin{cases}
(1-\alpha_{t})^{-1}\Gamma_{t-1} & t>1\\
1 & t=1
\end{cases}}$, $\Delta_t\coloneqq \psi(\ybf^t)-\psi(\xbf^\star)$ and $\xbf^\star$ is the optimal solution.
\end{prop}

\Cref{prop:convergence-bound} hinges on the successful termination of the line search at each step. The following proposition ensures this condition by demonstrating the boundedness of all the iterates, which in turn guarantees the success of the line search and allows us to bound the total line search complexity.
\begin{prop}
\label{prop:bound-all-iters}
Under the same conditions as \Cref{prop:convergence-bound}, all the iterates $\{\xbf^{t},\ybf^{t},\zbf^{t}\}_{0\le t\le T}$ generated by \Cref{alg:AGM-LS} fall in $\Bcal_{D^\star}({\xbf}^{\star})\coloneqq \{\xbf: \norm{\xbf-\xbf^\star}\le D^\star\}$, where $D^\star=\sqrt{\frac{2}{\alpha_{1}\gamma_{1}}(1-\alpha_{1})\Delta_{0}+\norm{{\xbf}^{\star}-\xbf^{0}}^{2}}$.
Let ${L}^{\star}= 2\, \sup\left\{ \Lcal(\xbf,\ybf):\xbf,\ybf\in\Bcal_{D^\star}({\xbf}^{\star})\right\} $, then
\eqref{eq:rec-bound-1} holds for all $t=1,\ldots,T$ and the total number of line search steps after $T$ iterations of \Cref{alg:AGM-LS}
is at most $\big\lceil\big(1+\frac{\log\tau_\mathrm{d}^{-1}}{\log\tau_\mathrm{u}}\big)T+\log\frac{\tau_{u}{L}^{\star}}{L_{0}} \big\rceil.$
\end{prop}

In order to prove \Cref{prop:bound-all-iters}, we first show a technical lemma which establishes the well-definedness of the line search conditioned on the success of the previous steps. Then using the induction principle, we shall prove solution boundedness and the success of the line search in all the iterations.
\begin{lem}
\label{lem:line-search-valid}Suppose that \Cref{alg:AGM-LS} generates $\xbf^{s},\ybf^{s},\zbf^{s}\in\Bcal_{D}(\xbf^{\star})$, where
$s=1,2,\ldots,t-1$, for some $D>0$. Then at the $t$-th iteration,
the line search is well-defined and terminates in a finite number
of steps. 
\end{lem}

\begin{proof}[Proof of \Cref{prop:bound-all-iters}]
Applying \Cref{lem:line-search-valid}, we have that $\ybf^{1},\zbf^{1}$
are well-defined and $\gamma_{1}$ has a finite value.  Let us prove the result by strong induction. First, it is clear that $\xbf^{0},\ybf^{0},\zbf^{0}\in\Bcal_{D^\star}(\xbf^{\star})$.
Now suppose we have $\xbf^{t},\ybf^{t},\zbf^{t}\in\Bcal_{D^\star}(\xbf^{\star})$,
for $t\le T-1$, then applying \Cref{lem:line-search-valid},
we know that the line search to find $\ybf^{T}$ is successful. Using \Cref{prop:convergence-bound} and non-negativity of $\Delta_{T}$,
we have 
\[
\frac{\Gamma_{T}\alpha_{T}(\gamma_{T}+\mu)}{2}\norm{\zbf^{T} - \xbf^{\star}}^{2}\le(1-\alpha_{1})\Delta_{0}+\frac{\alpha_{1}\gamma_{1}}{2}\norm{\zbf^{0} - \xbf^{\star}}^{2}.
\]
Note that by \eqref{eq:step-3}, we have $\Gamma_{t}\alpha_{t}\gamma_{t}\le\Gamma_{t+1}\alpha_{t+1}\gamma_{t+1}$,
which implies $\Gamma_{T}\alpha_{T}(\gamma_{T}+\mu)\ge\Gamma_{1}\alpha_{1}\gamma_{1}=\alpha_{1}\gamma_{1}$.
It follows that 
\[
\norm{\zbf^{T}-\xbf^{\star}}^{2}\le\frac{2(1-\alpha_{1})}{\alpha_{1}\gamma_{1}}\Delta_{0}+\norm{\zbf^{0} - \xbf^{\star}}^{2}\le(D^\star)^{2}.
\]
Since $\ybf^{T}$ is a convex combination of $\ybf^{T-1}$ and $\zbf^{T}$,
we obtain $\ybf^{T}\in\Bcal_{D^{\star}}(\xbf^{\star})$. 

In view of the line search procedure, we have $k_{t}\leq\frac{1}{\log\tau_\mathrm{u}}\left(\log\frac{\hat{L}_{t}}{\hat{L}_{t-1}}+\log\tau_\mathrm{d}^{-1}\right)$,
where $\hat{L}_{t}\le\tau_\mathrm{u} {L}^\star$. The total number of line searches
after $T$ iterations is 
\[
N_{T}=\sum_{t=1}^{T}(k_{t}+1)\le\sum_{t=1}^{T}\left(1+\frac{\log\tau_\mathrm{d}^{-1}}{\log\tau_\mathrm{u}}\right)+\log\frac{\hat{L}_{T}}{L_{0}}\le\left(1+\frac{\log\tau_\mathrm{d}^{-1}}{\log\tau_\mathrm{u}}\right)T+\log\frac{\tau_{u}{L^\star}}{L_{0}}.
\]
\end{proof}

With \Cref{prop:bound-all-iters}, we establish specific convergence rates of {\agls} under different convexity assumptions. It is easy to check that the following rules enforce the conditions \eqref{eq:step-1}-\eqref{eq:step-3}:
\begin{equation}\label{eq:choose-gamma-beta}
\gamma_{t}=(\hat{L}_{t}+\mu)\alpha_{t}-\mu,\ \beta_{t}=\frac{\gamma_{t}}{\hat{L}_{t}\alpha_{t}},
\end{equation}
where $\alpha_1\in[0,1]$ and $\alpha_{t}$ ($t\ge 2$) is the solution of 
\begin{equation}
\brbra{\hat{L}_{t}+\mu}\alpha_{t}^{2}+(b_{t-1}-\mu)\alpha_{t}-b_{t-1}=0,  \text{ where } b_{t}=\alpha_{t}^{2}(\hat{L}_{t}+\mu). \label{eq:root-alphat}
\end{equation}
\Cref{thm:rate-nag-convex} establishes the convergence rate of \Cref{alg:AGM-LS} when $\mu=0$.
\begin{thm}\label{thm:rate-nag-convex}
Under the same conditions as \Cref{prop:convergence-bound}, assume $\mu=0$, set $\alpha_{1}=1$, and choose the rest of the parameters according to \eqref{eq:choose-gamma-beta}. 
Then all the iterates of \Cref{alg:AGM-LS} remain in $\Bcal_{D_1^\star}(\xbf^*)$, where $D_1^\star\assign\norm{\xbf^0-\xbf^\star}$.
Moreover, the convergence rate is given by
\begin{equation}
\psi(\ybf^{T})-\psi({\xbf}^{\star})\le\frac{2\tau_\mathrm{u}L_{1}^{\star}}{(T+1)^{2}}\norm{{\xbf}^{\star}-\xbf^{0}}^{2},\label{eq:rate-convex}
\end{equation}
where $L_{1}^{\star}=\sup\left\{ 2\Lcal(\xbf,\ybf):\xbf,\ybf\in \Bcal_{D_1^\star}(\xbf^\star)\right\} $. 
\end{thm}

\begin{proof}
First, the boundedness of $\{\xbf^{t},\ybf^{t},\zbf^{t}\}$ follows
from \Cref{prop:bound-all-iters}. In view of the parameter
selection, we have the relation $\hat{L}_{t}\alpha_{t}^{2}=\hat{L}_{t-1}\alpha_{t-1}^{2}(1-\alpha_{t})$,
and hence 
$\alpha_{t}=\frac{2}{1+\sqrt{1+\frac{4\hat{L}_{t}}{\hat{L}_{t-1}\alpha_{t-1}^{2}}}}$.

Define $\tilde{L}_{t}=\max\{\hat{L}_{1},\hat{L}_{2},\ldots,\hat{L}_{t}\}$.
We use induction to show $\hat{L}_{t}\alpha_{t}^{2}\le\frac{4\tilde{L}_{t}}{(1+t)^{2}}$.
The $t=1$ is clear from the initialization. Next, we assume this result holds for
the $(t-1)$-th iteration, namely, $\hat{L}_{t-1}\alpha_{t-1}^{2}\le\frac{4\tilde{L}_{t-1}}{t^{2}}.$
We have 
\[
\sqrt{\hat{L}_{t}}\alpha_{t}=\frac{2\sqrt{\hat{L}_{t}}}{1+\sqrt{1+\frac{4\hat{L}_{t}}{\hat{L}_{t-1}\alpha_{t-1}^{2}}}}\le\frac{2\sqrt{\hat{L}_{t}}}{1+\sqrt{1+\frac{4\hat{L}_{t}}{4\tilde{L}_{t-1}}t^{2}}}\le\frac{2\sqrt{\hat{L}_{t}}}{1+\sqrt{\frac{\hat{L}_{t}}{\tilde{L}_{t-1}}}\cdot t}\le\frac{2\sqrt{\hat{L}_{t}}}{1+\sqrt{\frac{\hat{L}_{t}}{\tilde{L}_{t}}}\cdot t}\le\frac{2\sqrt{\tilde{L}_{t}}}{1+t}.
\]
As a result,
\[
\Gamma_{t}=\frac{\Gamma_{1}\alpha_{1}\gamma_{1}}{\hat{L}_{t}\alpha_{t}^{2}}\ge\frac{\gamma_{1}}{4\tilde{L}_{t}}(t+1)^{2}.
\]
Applying \Cref{prop:convergence-bound}, we have the first
inequality in \eqref{eq:rate-convex}. Finally, since all the iterates
are in $\Bcal_{D_1^\star}(\xbf^\star)$, we have $\tau_\mathrm{d}\hat{L}_{t-1}\le\hat{L}_{t}\le\tau_\mathrm{u}L_{1}^{\star}$,
which implies $\tilde{L}_{t}\le\tau_\mathrm{u}L_{1}^{\star}$, for any $t>0$. Taking $t = T$ completes he proof.
\end{proof}

We now present the convergence rate for the case when the objective function is strongly convex ($\mu > 0$).
\begin{thm}\label{thm:strongly-convex}
Under the same conditions as \Cref{prop:convergence-bound}, assume $\mu > 0$, set $\alpha_{1}=\sqrt{\frac{\mu}{\hat{L}_{1}+\mu}}$, and choose the rest of the parameters according to \eqref{eq:choose-gamma-beta}. 
Then all the iterates of \Cref{alg:AGM-LS} remain in $\Bcal_{D_2^\star}(\xbf^\star)$, where $D_2^\star\assign\sqrt{\frac{2}{\mu}\left[\psi(\xbf^{0})-\psi({\xbf}^{\star})\right]+\norm{\xbf^{0}-{\xbf}^{\star}}^{2}}$. Moreover, the convergence rate is given by
\begin{equation}
\begin{aligned}
\psi(\ybf^{T})-\psi({\xbf}^{\star})+\frac{\mu}{2}\norm{\zbf^{T} - {\xbf}^{\star}}^{2}
& \le \prod_{1\le t\le T}\left(1-\sqrt{\frac{\mu}{\hat{L}_{t}+\mu}}\right)\left[\psi(\xbf^{0})-\psi({\xbf}^{\star})+\frac{\mu}{2}\norm{\xbf^{0} - {\xbf}^{\star}}^{2}\right]\\
& \le 2 \exp \left(-\sqrt{\frac{\mu}{\tau_u{L}_2^\star+\mu}} T\right)\left[\psi(\xbf^{0})-\psi({\xbf}^{\star})\right],
\end{aligned} \label{eq:rate-strongly-convex}
\end{equation}
where $L_{2}^{\star}=\sup\left\{ 2\Lcal(\xbf, \ybf ):\xbf\in \Bcal_{D_2^\star}(\xbf^\star)\right\} $.
\end{thm}

\begin{proof}
First, the boundedness of $\{\xbf^{t},\ybf^{t},\zbf^{t}\}$ follows
from \Cref{prop:bound-all-iters}. Next, we show $\alpha_{t}\ge\sqrt{\frac{\mu}{\hat{L}_{t}+\mu}}$
by induction. The $t=1$ is clear from the initialization. Suppose the condition holds for $s=1,2,\ldots,t-1$, which implies, in \eqref{eq:root-alphat}, that $b_{s}\ge\mu$ for $1\le s\le t-1$. Then we have
\[
\frac{(\hat{L}_{t}+\mu)\alpha_{t}^{2}-\mu\alpha_{t}}{(1-\alpha_{t})}=b_{t-1}\ge\mu.
\]
As it is clear from \eqref{eq:root-alphat} that $\alpha_{t}<1$,
we immediately have $\alpha_{t}\ge\sqrt{\frac{\mu}{\hat{L}_{t}+\mu}}$.
Consequently, we have
\[
\Gamma_{T}=\prod_{2\le t\le T}\left(1-\sqrt{\frac{\mu}{\hat{L}_{t}+\mu}}\right)^{-1},\ \text{and}\ \Gamma_{T}\alpha_{T}(\gamma_{T}+\mu)\ge\mu\prod_{2\le t\le T}\left(1-\sqrt{\frac{\mu}{\hat{L}_{t}+\mu}}\right)^{-1}.
\]
Applying \Cref{prop:convergence-bound} with the lower
bound on $\Gamma_{T}$ and noticing $\alpha_{1}\gamma_{1}=\mu(1-\alpha_{1})$,we
have the desired convergence rate \eqref{eq:rate-strongly-convex}.

Lastly, since all the iterates fall in $\Bcal_{D_2^\star}(\xbf^\star)$, the line
search guarantees $\hat{L}_{t}\le\tau_\mathrm{u}L_{2}^{\star}$, using the property $(1-x)^T\le \exp(-T x)$ for $x \in(0,1)$ and $T>0$, and strong convexity $\psi(\xbf^0)-\psi(\xbf^\star)\ge \frac{\mu}{2} \norm{\xbf^0-\xbf^\star}^2$, we have the second inequality.
\end{proof}

\subsection{Smooth approximation algorithms for weakly convex problems}

We are ready to incorporate the accelerated gradient method with line search as a subroutine for solving the subproblems arising within the {\sipp} framework. We refer to the resulting approach as the {\aglssipp} method. To facilitate the subsequent analysis, we introduce the following key assumption.
\begin{assumption}\label{assum:level-phi-eta}
Both $r$ and $f_\eta$ (for any $\eta > 0$) are bounded below. Specifically, let $l \in \mathbb{R}$ denote a lower bound of $\phi_\eta$. Furthermore, we assume that the smoothness parameter $\Lcal_\eta(\xbf, \ybf)$ can be expressed as $\Lcal_\eta(\xbf, \ybf) = B(\xbf, \ybf) + \frac{L(\xbf, \ybf)}{\eta}$, where $B\colon \mathbb{R}^d \times \mathbb{R}^d \to [0, \infty)$ and $L\colon \mathbb{R}^d \times \mathbb{R}^d \to [0, \infty)$ is a symmetric negative continuous map.
\end{assumption}

\begin{rem}
The property of lower boundedness of \(\phi_\eta\) typically follows from that of the original objective \(\phi\). Assume both $f$ and $r$ are lower-bounded.
For example, in generalized Nesterov smoothing, we have 
\(\lvert \phi_\eta(\mathbf{x}) - \phi(\mathbf{x}) \rvert = \mathcal{O}(\eta)\), which yields
\(\inf_{\mathbf{x}} \phi_\eta(\mathbf{x}) \ge \inf_{\mathbf{x}} \phi(\mathbf{x}) - \mathcal{O}(\eta) > -\infty\).
For Moreau-envelope smoothing, we have
\[
\inf_\xbf \phi_\eta(\xbf) = \inf_\xbf \inf_\ybf \big[ f(\ybf) + \frac{\rho+\max\{\eta^{-1},\rho\}}{2}\norm{\ybf-\xbf}^2 \big] + r(\xbf) \ge \inf_\ybf f(\ybf) + \inf_\xbf r(\xbf) > -\infty.
\]
\end{rem}

Next, we establish the complexity of the overall algorithm.
\begin{thm}\label{thm:complexity-agm-ls-ipp}
    In {\aglssipp}, suppose \Cref{assum:level-phi-eta} holds, and {\agls} (\Cref{alg:AGM-LS}) employs the early stop strategy. Then the iterates produced by the inexact proximal point scheme satisfy
    \[
    \mathbf{x}^k,\ \hat{\mathbf{x}}^{\,k}\in
    S_0\,\assign\,\bigl\{\mathbf{x}:\phi_\eta(\mathbf{x})\le \phi_\eta(\mathbf{x}^0)\bigr\},
    \qquad k=0,1,\dots,K.
    \]
Moreover, all the intermediate iterates produced by {\agls} lie in the set
    $
    \Bcal_{D^\dagger}(S_0) \coloneqq \{\xbf: \dist(\xbf,S_0)\le D^\dagger\}, 
    $
    where $D^\dagger=\sqrt{\frac{8}{\hat\rho-\bar\rho}\bsbra{{\phi}(\xbf^{0})-l}}$.
Besides, after $K$ iterations, the total number of iterations of {\agls} is bounded by 
       $ \big\lceil \log  4   \sqrt{\frac{\tau_u \Lcal^\dagger+\hat\rho-\bar\rho}{\hat\rho-\bar\rho}} K \big\rceil$, 
    where $\Lcal^\dagger = \sup \{2\Lcal_\eta(\xbf,\ybf) + 2\bar\rho : \xbf,\ybf\in \Bcal_{{D^\dagger}}(S_0)\}$.
\end{thm}
\begin{proof}
At the $k$-th outer iteration, the subproblem considered is
\[
\phi_\eta(\xbf)=f_\eta(\xbf)+\frac{\bar\rho}{2}\norm{\xbf-\xbf^{k-1}}^2 + \frac{\hat\rho-\bar\rho}{2}\norm{\xbf-\xbf^{k-1}}^2+ r(\xbf).
\]
By invoking \Cref{thm:strongly-convex}, the accelerated gradient method achieves the following rate of convergence:
\begin{equation}
\begin{aligned}
    \hat{\phi}_k(\xbf^k) - \hat{\phi}_k(\hatx^k) 
    &\leq 2\exp\left(-\sqrt{\frac{\hat\rho-\bar\rho}{\tau_u{L}^\star(\hatx^k)+\hat\rho-\bar\rho}}T_k\right)\left[ \hat{\phi}_k(\xbf^{k-1}) - \hat{\phi}_k(\hatx^k)\right],
\end{aligned}
\end{equation}
where $T_k$ denotes the number of iterations of \Cref{alg:AGM-LS}, $L^\star(\hatx^k)=\sup\{2\Lcal_\eta(\xbf,\ybf)+2\bar\rho: \xbf,\ybf\in\Bcal_{D_k}(\hatx^k)\}$, and $D_k^2=\frac{2}{\hat\rho-\bar\rho}\bsbra{\hat{\phi}_k(\xbf^{k-1})-\hat{\phi}_k(\hatx^k)}+\norm{\xbf^{k-1}-\hatx^k}^2$.

Consequently, it takes at most
$
T_k^\star \leq \log  4 \sqrt{\frac{\tau_u L^\star(\hatx^k)+\hat\rho-\bar\rho}{\hat\rho-\bar\rho}}
$
iterations of \Cref{alg:AGM-LS} to ensure
\begin{equation}\label{eq:relative-gap}
\hat{\phi}_k(\xbf^k) - \hat{\phi}_k(\hatx^k) \le 
\frac{1}{2}\bsbra{\hat{\phi}_k(\xbf^{k-1}) - \hat{\phi}_k(\hatx^k)}.
\end{equation}
Therefore, the total number of {\agls} iterations across the $K$ outer steps is bounded by
\begin{equation}\label{eq:total-number-agls}
    \sum_{k=1}^K T_k^\star = \sum_{k=1}^K \log  4 \sqrt{\frac{\tau_u L^\star(\hatx^k)+\hat\rho-\bar\rho}{\hat\rho-\bar\rho}}.
\end{equation}

Note that
\[
\phi_\eta(\xbf^k)\le \hat{\phi}_k(\xbf^k) \le \hat{\phi}_k(\xbf^{k-1}) = \phi_\eta(\xbf^{k-1}) \le \dots \le \phi_\eta(\xbf^0),
\]
implying that the sequence $\{\phi_\eta(\xbf^k)\}_k$ is non-increasing. Thus, the entire sequence $\{\xbf^k\}_{k}$ is contained in $S_0$. Similarly, since $\phi_\eta(\hat{\xbf}^k)\le \hat{\phi}_k(\xbf^{k-1})= \phi_\eta(\xbf^{k-1})$, it follows that $\{\hat{\xbf}^k\}_k\subset S_0$ as well.

Using the definition of $D_k$ and exploiting the strong convexity of $\hat{\phi}_k$, we observe that
\begin{equation}
    D_k^2 
    \leq \frac{4}{\hat\rho-\bar\rho}\bsbra{\hat{\phi}_k(\xbf^{k-1})-\hat{\phi}_k(\hatx^k)} 
    \leq \frac{8}{\hat\rho-\bar\rho}\bsbra{\hat{\phi}_k(\xbf^{k-1})-\hat{\phi}_k(\xbf^k)} 
    \leq \frac{8}{\hat\rho-\bar\rho}\bsbra{\phi_\eta(\xbf^{k-1})-\phi_\eta(\xbf^k)},
\end{equation}
where the second equality rearranges from \eqref{eq:relative-gap} and the last one follows from the definition of $\hat\phi_k$ and nonnegativity of the quadratic term. 
Summing over $k$ yields
\[
\sum_{k=1}^K D_k^2 \le \frac{8}{\hat\rho-\bar\rho}\bsbra{\phi(\xbf^{0})-\phi_\eta(\xbf^K)} \le \frac{8}{\hat\rho-\bar\rho}\bsbra{\phi(\xbf^{0})-l} = (D^\dagger)^2.
\]
Therefore, all iterates generated by the accelerated gradient method reside within the set
\[
\bigcup_{1\leq k\leq K} \Bcal_{D_k}(\hatx^k) \subset \left\{ \xbf: \dist(\xbf, S_0) \leq D^\dagger \right\} = \Bcal_{D^\dagger}(S_0).
\]
Consequently, we have
\[
L^\star(\hatx^k) \leq \sup \left\{ 2\Lcal_\eta(\xbf, \ybf) + 2\bar\rho : \xbf, \ybf \in \Bcal_{D^\dagger}(S_0) \right\} 
= \Lcal^\dagger.
\]
Substituting this uniform bound into \eqref{eq:total-number-agls} yields
$
\sum_{k=1}^K T_k^\star \leq \log  4 \sqrt{\frac{\tau_u \Lcal^\dagger+\hat\rho-\bar\rho}{\hat\rho-\bar\rho}} K.
$
\end{proof}

We observe that $\Lcal^\dagger$ does not depend on the iteration index $k$. We primarily consider the following two scenarios where $\Lcal^\dagger$ is globally bounded.
\begin{assumption}\label{assum:level-bounded-phi-eta}
    For every \(v \in \mathbb{R}\), the sublevel set \(\level(\phi_\eta, v) = \{\xbf : \phi_\eta(\xbf) \le v\}\) is bounded.
\end{assumption}
    Imposing level-boundedness is a natural and reasonable requirement in our setting. For instance, if the regularizer \(r\) is level-bounded, then $\phi_\eta(\xbf)$ automatically possesses this property. In particular, for any $\xbf$ such that $\phi_\eta(\xbf) \le v$, we have
    \(\inf_\ybf f_\eta(\ybf) + r(\xbf) \le v\), which implies $\xbf \in \level(r, v - \inf_\ybf f_\eta(\ybf))$.

Next, we present a more specific assumption tailored to generalized Nesterov smoothing:
\begin{assumption}\label{assum:level-bounded-h-eta}
    Suppose generalized Nesterov smoothing is employed, i.e., $\phi_\eta(\xbf) = f_|eta(\xbf)= h_\eta(A(\xbf))$, where $A:\Rbb^d\!\raw\!\Rbb^m$ and each component $A_i, i \in [m]$, is a smooth function. We assume that $h_\eta$ is level-bounded and that each $A_i$ is lower-bounded.
\end{assumption}
The level-boundedness of $h_\eta$ is a natural assumption, as $h_\eta$ frequently serves as a loss function in machine learning. Notably, widely used examples of $h_\eta$, such as the Huber loss, possess this property.

We now establish that $\Lcal^\dagger$ is bounded under the proposed assumptions.
\begin{prop}
    Suppose either \Cref{assum:level-bounded-phi-eta} holds, or that generalized Nesterov smoothing is employed and \Cref{assum:level-bounded-h-eta} holds. Then, $\Lcal^\dagger$ is bounded.
\end{prop}
\begin{proof}
First, consider the case where \Cref{assum:level-bounded-phi-eta} holds. It is straightforward to see that $\Bcal_{D^\dagger}(S_0)$ is a bounded set. 
Since $\Lcal_\eta(\xbf,\ybf)$ is continuous and $\Bcal_{D^\dagger}(S_0)$ is compact, it follows that $\Lcal^\dagger$ is finite. 

Next, consider the case where \Cref{assum:level-bounded-h-eta} holds. 
Recall that $\Lcal_\eta(\xbf,\ybf) = \frac{1}{\sigma \eta} \sup_{0\leq\theta\leq1} \norm{\nabla A(\theta \xbf + (1-\theta)\ybf)}_\mathrm{op}^2 + B L_A$. Therefore, it suffices to establish that $\{\norm{\nabla A(\xbf)}_\mathrm{op}: \xbf \in \Bcal_{D^\dagger}(S_0)\}$ is bounded above.
 We first analyze the situation for $\xbf \in S_0$. By the definition of $S_0$, we have $\phi_\eta(A(\xbf)) \le \phi_\eta(A(\xbf^0))$. Given our assumption that $r(\xbf)$ is lower-bounded, i.e., $r(\xbf) \ge a$ for some $a \in \Rbb$, it follows that $h_\eta(A(\xbf)) \le \phi_\eta(\xbf^0) - a < +\infty$. Since $h_\eta$ is level-bounded, the set $\{A(\xbf) : \xbf \in S_0\}$ is also bounded. We denote $\kappa = \sup \{\norm{A(\xbf)}_\infty: \xbf \in S_0\}$.

Boundedness of $\{\nabla A(\xbf): \xbf\in S_0\}$ then follows from the self-bounding property of lower-bounded smooth functions. Specifically, suppose each $A_i, i \in [m]$ is $L_i$-smooth. Then, the relation
\[
A_i(\ybf) \le A_i(\xbf) + \langle \nabla A_i(\xbf),  \ybf-\xbf \rangle + \frac{L_i}{2}\norm{\ybf-\xbf}^2
\]
holds for any $\xbf \in S_0$ and any $\ybf\in\Rbb^d$. Minimizing both sides with respect to $\ybf$ gives
\[
A_i^\star = \min_\ybf A_i(\ybf) \le A_i(\xbf) - \frac{1}{2L_i}\norm{\nabla A_i(\xbf)}^2,
\]
which implies $\norm{\nabla A_i(\xbf)}^2 \le 2L_i \left[A_i(\xbf)-A_i^\star\right] \le 2\max_{i \in [m]} L_i (\kappa - \min_{i \in [m]} A_i^\star)<\infty$. This ensures the boundedness of $\{\norm{\nabla A(\xbf)}_\mathrm{op}: \xbf \in S_0\}$.

Next, for $\xbf \in \Bcal_{D^\dagger}(S_0) \setminus S_0$, we obtain by the triangle inequality that
\begin{equation*}
    \begin{aligned}
\norm{\nabla A(\xbf)}_\mathrm{op} & \leq \norm{\nabla A(\xbf) - \nabla A(\hatx)}_\mathrm{op} + \norm{\nabla A(\hatx)}_\mathrm{op} 
\leq L_A \norm{\xbf - \hatx} + \norm{\nabla A(\hatx)}_\mathrm{op} 
 \leq L_A D^\dagger + \norm{\nabla A(\hatx)}_\mathrm{op},
\end{aligned}
\end{equation*}
where $\hatx \in S_0$ denotes the projection of $\xbf$ onto $S_0$. This completes the proof.
\end{proof}

Under the preceding assumptions, we now summarize the overall complexity of the inexact proximal point method.
\begin{coro}
    Suppose that either \Cref{assum:level-bounded-phi-eta} or \Cref{assum:level-bounded-h-eta} holds, then $\Lcal^\dagger$ is bounded. Define $B^\dagger = \sup\{B(\xbf,\ybf) : \xbf, \ybf \in \Bcal_{D^\dagger}(S_0)\}$ and $L^\dagger = \sup\{L(\xbf,\ybf) : \xbf, \ybf \in \Bcal_{D^\dagger}(S_0)\}$. 
    Set the parameters $\hat\rho = 2\bar\rho$ and $\eta = \varepsilon^{2}$. 
    Then, to compute an \((\varepsilon, \varepsilon)\)-stationary point of problem~\eqref{pb:main}, 
    it requires at most
    $
        \mathcal{O}\left(\frac{1}{\vep^2}\left(\sqrt{B^\dagger} + \frac{\sqrt{L^\dagger}}{\vep} + 1\right)\right)$
    iterations of \Cref{alg:AGM-LS}.
\end{coro}

\begin{proof}
Note that $\Lcal^\dagger$ can be bounded as $\Lcal^\dagger \le 2B^\dagger + 2\frac{L^\dagger}{\eta} + 2\rho$. According to \Cref{thm:strongly-convex}, this yields
$
T_k^\star \le \log  4 \sqrt{\frac{2\tau_{\mathrm{u}}\left(B^\dagger + L^\dagger/\eta\right) + \hat\rho + (2\tau_{\mathrm{u}} - 1)\bar\rho}{\hat\rho - \bar\rho}}$.
By invoking \Cref{thm:sipp-rate}, we conclude that the inexact proximal point method requires $K = \mathcal{O}(1/\vep^2)$ iterations to obtain an $\mathcal{O}(\vep)$-approximate stationary point. Consequently, the total number of accelerated gradient method iterations is bounded by
$
\sum_{k=1}^K T_k^\star = \mathcal{O}\left(\frac{1}{\vep^2}\left(\sqrt{B^\dagger} + \frac{\sqrt{L^\dagger}}{\vep} + 1\right)\right).
$
\end{proof}
\begin{rem}
    It is important to emphasize that \Cref{assum:level-bounded-phi-eta,assum:level-bounded-h-eta} are imposed primarily for the purposes of theoretical analysis. In practice, these assumptions may not always hold. In such cases, the complexity bound instead depends on the local smoothness constant $L^\star(\hat{\xbf}^k)$ as stated in \eqref{eq:total-number-agls}. Nevertheless, this does not significantly impact our algorithm, as the line search in \Cref{alg:AGM-LS} is designed to automatically adapt to local smoothness property. Consequently, the algorithm remains practical and efficient without requiring a priori knowledge of global Lipschitz constants.
\end{rem}

\section{Numerical experiments}

In this section, we conduct numerical experiments to demonstrate the
efficiency of the smoothing approach developed in this paper. In particular, we consider the following robust nonlinear regression
\begin{equation*}
  \min_{\x \in \mathcal{X} \subseteq \mathbb{R}^d}  f ( \x )
  \assign \textstyle \frac{1}{m} \sum_{i = 1}^m f ( \x, \xi_i ), 
\end{equation*}
where $\xi = ( \tma, b )$ and $f ( \x, \xi ) \assign
| h ( \langle \tma, \x \rangle ) - b |$. The
nonlinear function $h (z) \in \{ z^2, z^5 + z^3 + 1, e^z + 10 \}$.
\vspace{-5pt}
\paragraph{Smoothing functions}
Denote by $f_{\eta} ( \x, \xi )$ the smooth approximation of $f
( \x, \xi )$. We consider:
\begin{itemize}[leftmargin=15pt]
  \item \textit{Nesterov smoothing} smoothes the outer absolute value function by
  taking 
  \[ | z | \approx \alpha_{\eta} (z) \assign \bigg\{ \begin{array}{cc}
       \frac{z^2}{2 \eta}, & | z | \leq \eta\\
       | z | - \frac{\eta}{2}, & \text{else}
     \end{array}  \]
  and $f_{\eta} ( \x, \xi ) = \alpha_{\eta} ( h (
  \langle \tma, \x \rangle ) - b )$.
  
  \item \textit{Moreau envelop smoothing} takes $f_{\eta} ( \x, \xi ) = \argmin_{\y} \{ f ( \y, \xi )
     + \frac{2 \rho + \eta^{- 1}}{2} \| \y - \x \|^2\}$. Experiments using Moreau envelope use $h(z) = z^2$ to ensure a closed-form solution \cite{davis2019stochastic}.
\end{itemize}

\subsection{Experiment setups}

\paragraph{Dataset generation~\citep{gao2024stochastic}} 
Let $A \in \Rbb^{m\times d}$ have $\{\tma_i\}$ as its rows.
Given a condition number parameter $\kappa \geq 1$, we generate $A = Q D \in \mathbb{R}^{m \times d}$, where each element of $Q \in \mathbb{R}^{m \times d}$ is sampled from standard normal distribution and $D \in \mathbb{R}^{d \times d}$ is a diagonal matrix whose diagonal elements are evenly distributed between 1 and $\kappa$. We generate $\xbf^{\star} \sim \mathcal{N} (0, I_d)$ and let $b_i = h
(\langle \abf_i, \xbf^{\star} \rangle) + \theta_i \varepsilon_i$, where
$\varepsilon_i$ simulates corruption by random noise, $\theta_i \sim
\tmop{Bernoulli} (p)$ and $\varepsilon_i \sim \mathcal{N} (0, 25)$. Here $p \in [0, 1]$ represents the fraction of corrupted data on expectation.
\begin{enumerate}[leftmargin=15pt,label=\textbf{\arabic*)},itemsep=1pt]
  \item \tmtextbf{Dataset}. We use $m = 300, d = 100$ to test deterministic
  algorithms and $m = 1000, d = 20$ to test stochastic algorithms. In particular, we use $m = 30, d = 10$ when Moreau envelope smoothing is used.
  
  \item \tmtextbf{Initial point}. We set the initial point of all the
  algorithms to be $\xbf^0 \sim \frac{\hat{\xbf}}{\| \hat{\xbf} \|}$, where $\hat{\xbf}
  \sim \mathcal{N} (0, I_d)$.
  
  \item \tmtextbf{Stopping criterion}. The stopping criterion is set to $f
  (\xbf^k) \leq 1.5 f (\xbf^{\star})$.
  
  \item \tmtextbf{Oracle access}. We allow at most $400 m$ gradient oracle
  accesses for all the algorithms.
  
  \item \tmtextbf{Bounded feasible region}. We take $\Xcal = \{\x: \|\x\| \leq M\} $ for $M = 10^5$.
\end{enumerate}

\paragraph{Parameter configuration} 
For each algorithm, we tune its parameters as follows.  First, for the {stepsize} $\alpha$ (and $\gamma$), we set $\alpha = \gamma^{-1} = \frac{\alpha_0}{\sqrt{K}}$, where $\alpha_0$ is selected as the best value from the range $\{10^{-2}, 10^{-1}, 1, 10\}$.  Second, the smoothing parameter is set to $\eta = \varepsilon^2 = 2 f(\xbf^\star)^2 \approx 0.8$.  Finally, in the {\proxacc} method, the {proximal point subproblem} is solved to a gradient-norm tolerance of $0.75$ with a maximum of eight iterations. If more than six iterations are required, we update the regularization parameter via $\gamma_{k+1} = \max\{0.5 \gamma_k, 10\}$ and adjust the smoothing parameter by $\eta \leftarrow \varepsilon^2 / k$.

\subsection{Experiments on deterministic problems}

We compare the following deterministic algorithms:
\begin{itemize}[leftmargin=15pt]
\item \textit{Deterministic subgradient method} (\gm). $\xbf^{k + 1} = \xbf^k - \alpha_k f' (\xbf^k)$.
\item \textit{Deterministic gradient descent on smoothed function} ({\sspg} with no randomness, \Cref{alg:Smoothed-spg}).
\item \textit{Inexact proximal point with deterministic Nesterov acceleration} ({\proxaccvr} with no randomess, \Cref{alg:Smoothed-ipp}).
\end{itemize}
\begin{figure}[h]
\centering
\includegraphics[scale=0.22]{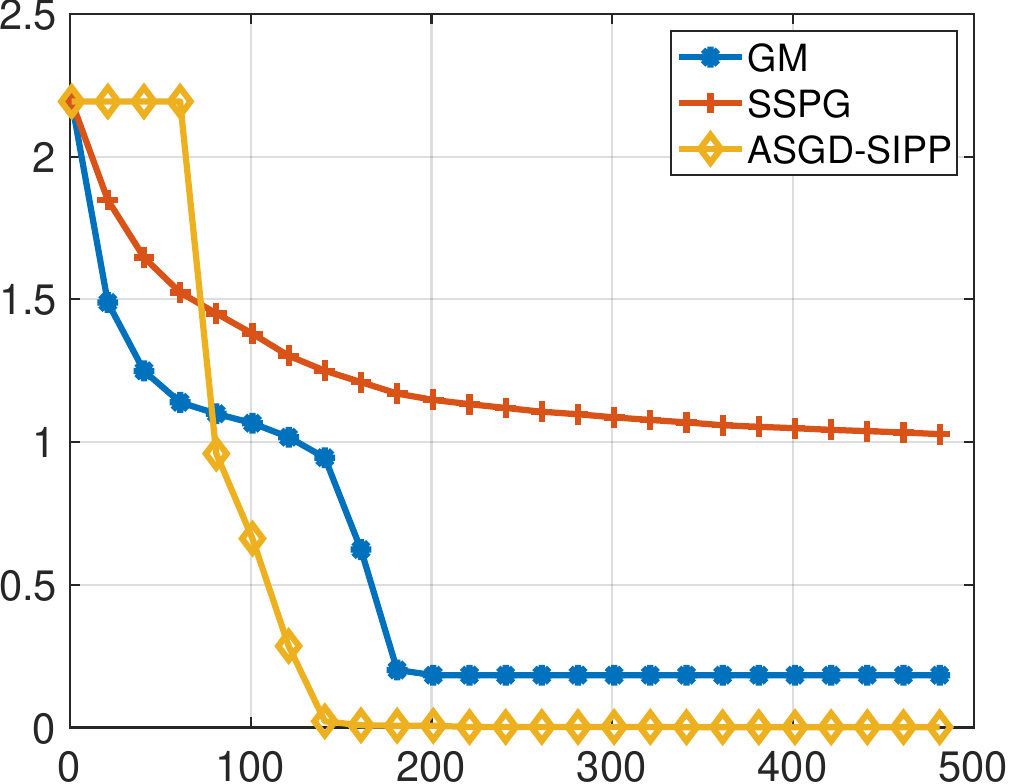}
\includegraphics[scale=0.22]{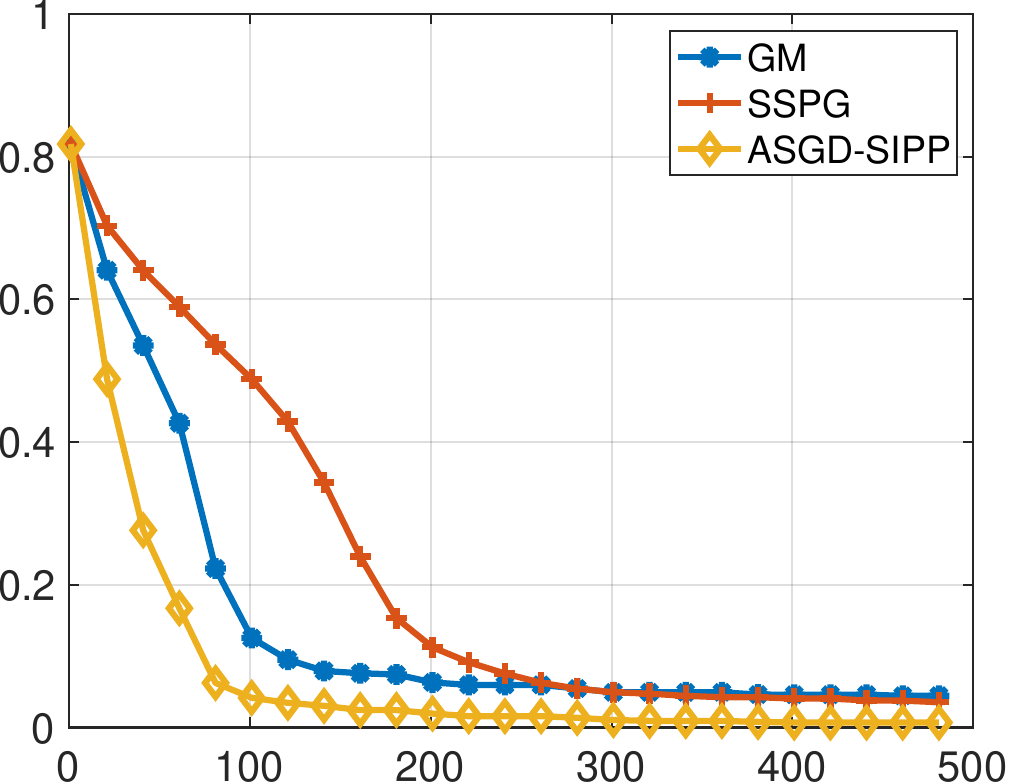}
\includegraphics[scale=0.22]{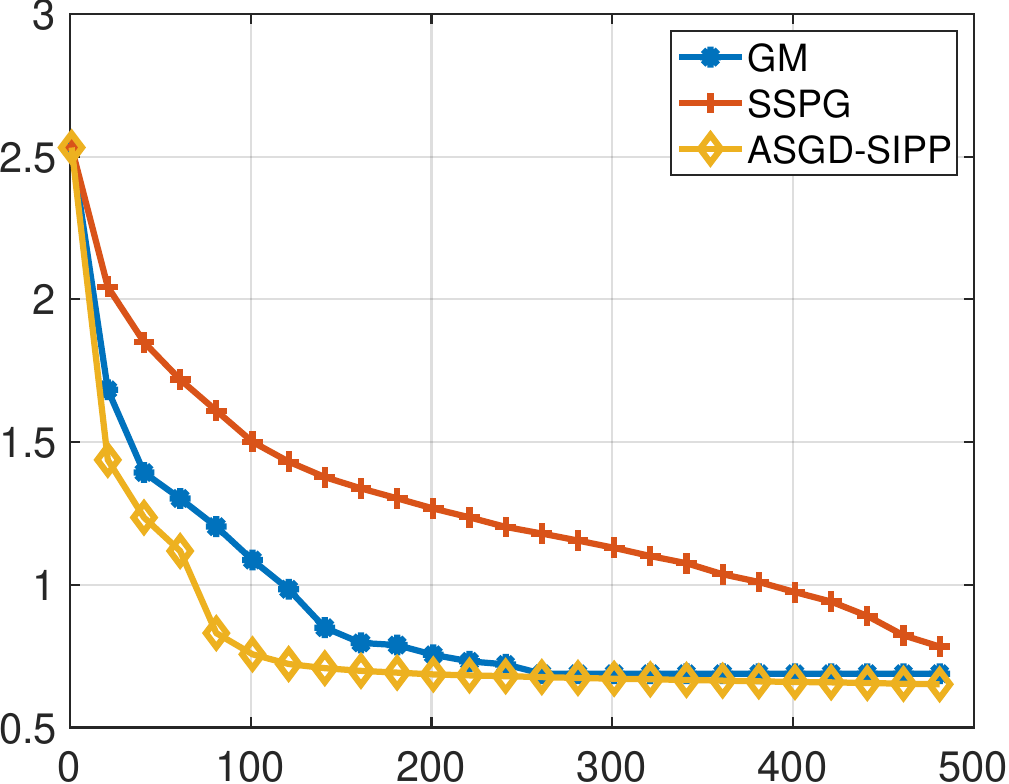}
\includegraphics[scale=0.22]{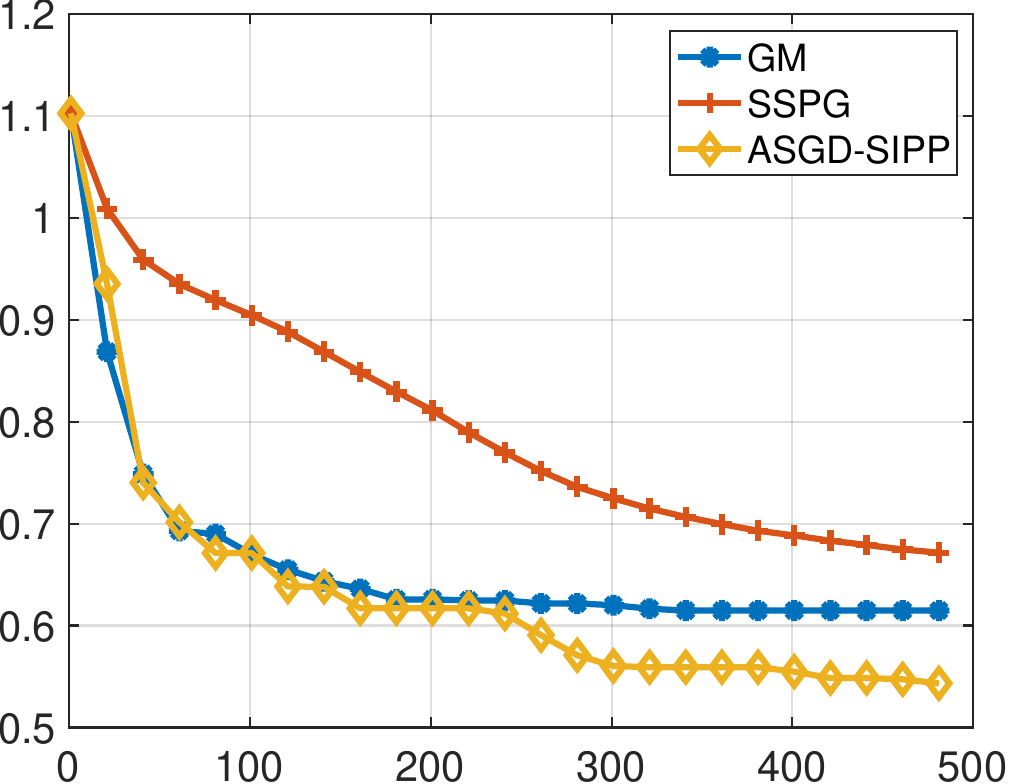}
\includegraphics[scale=0.22]{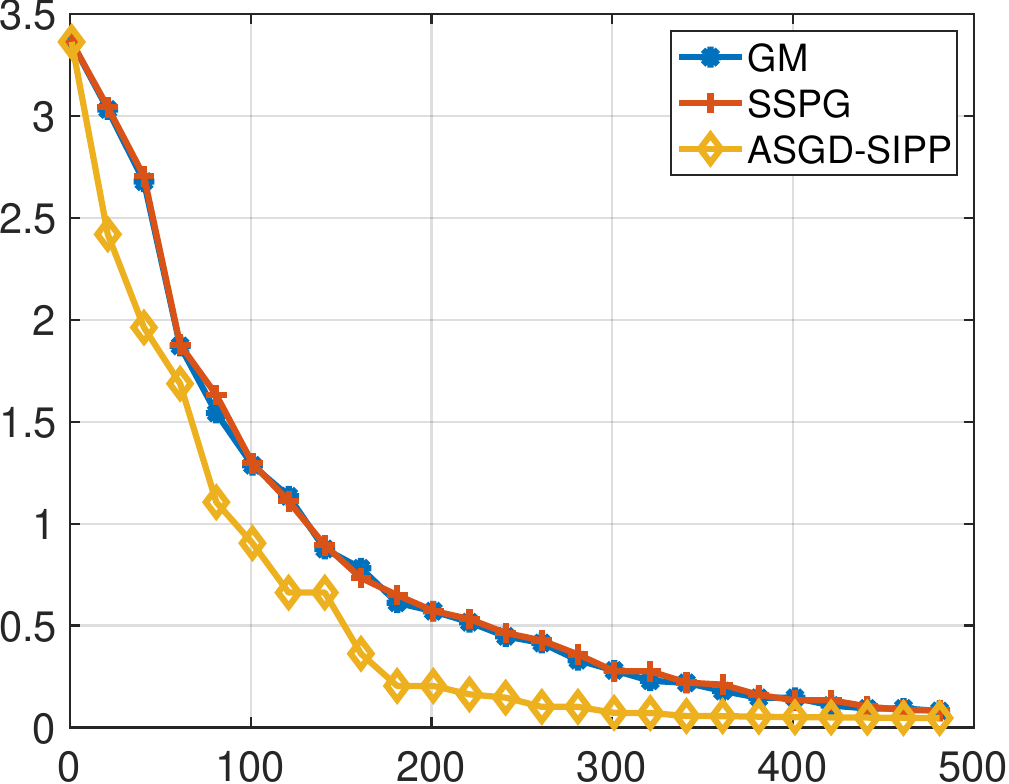}
\includegraphics[scale=0.22]{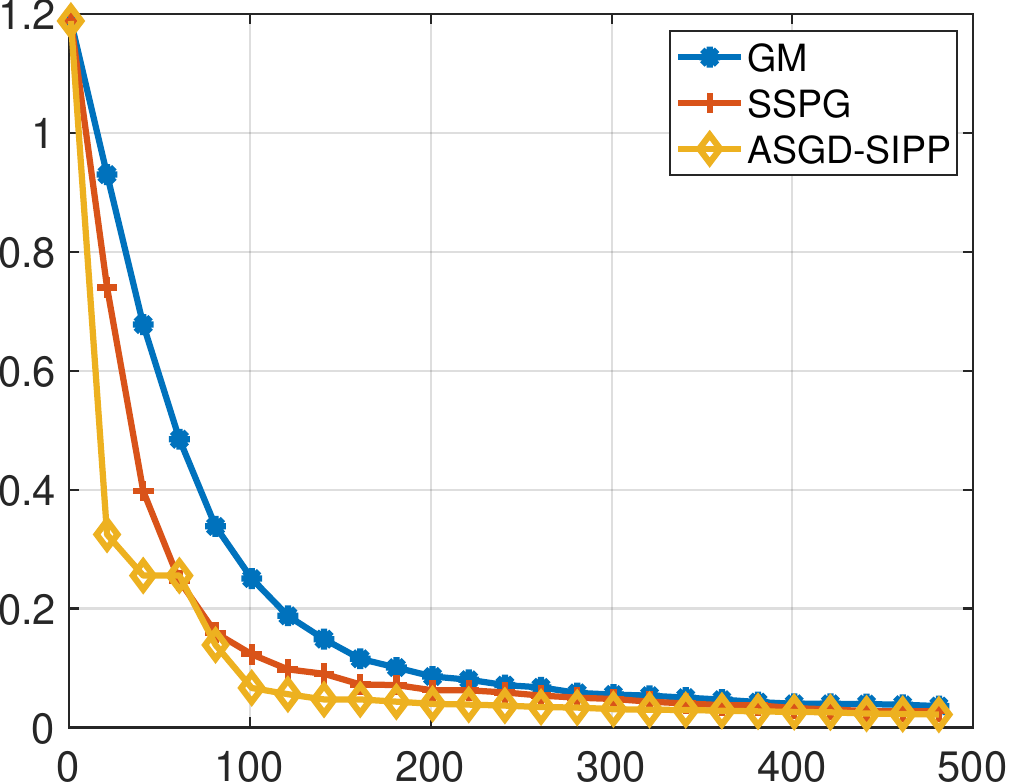}
\includegraphics[scale=0.22]{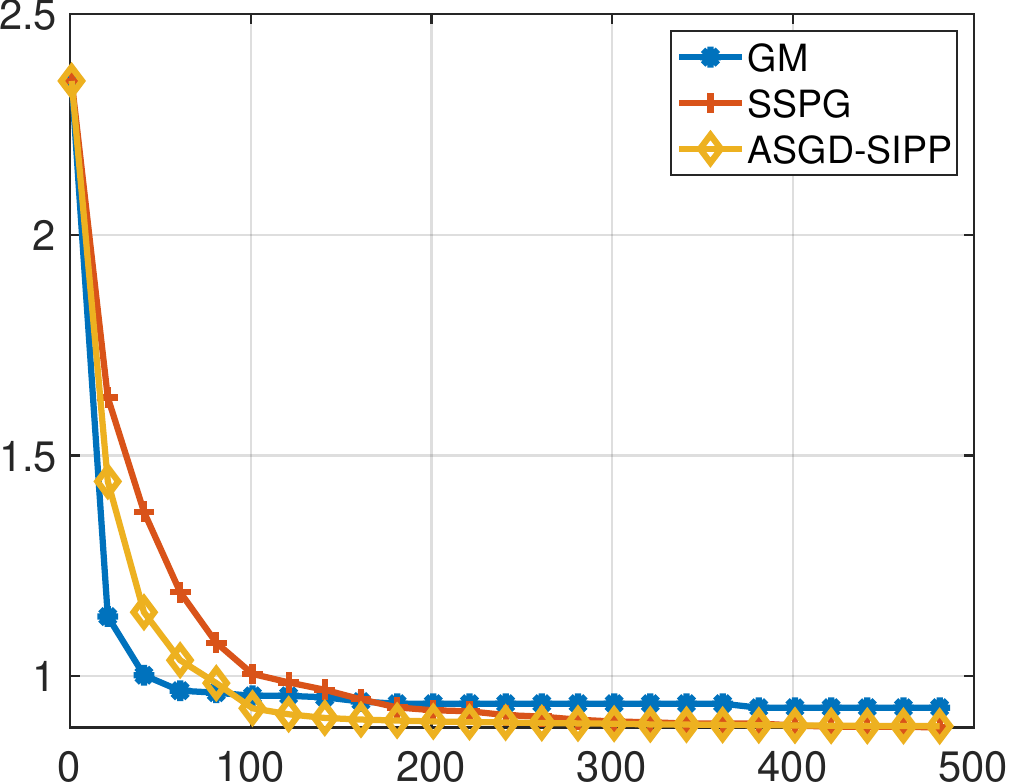}
\includegraphics[scale=0.22]{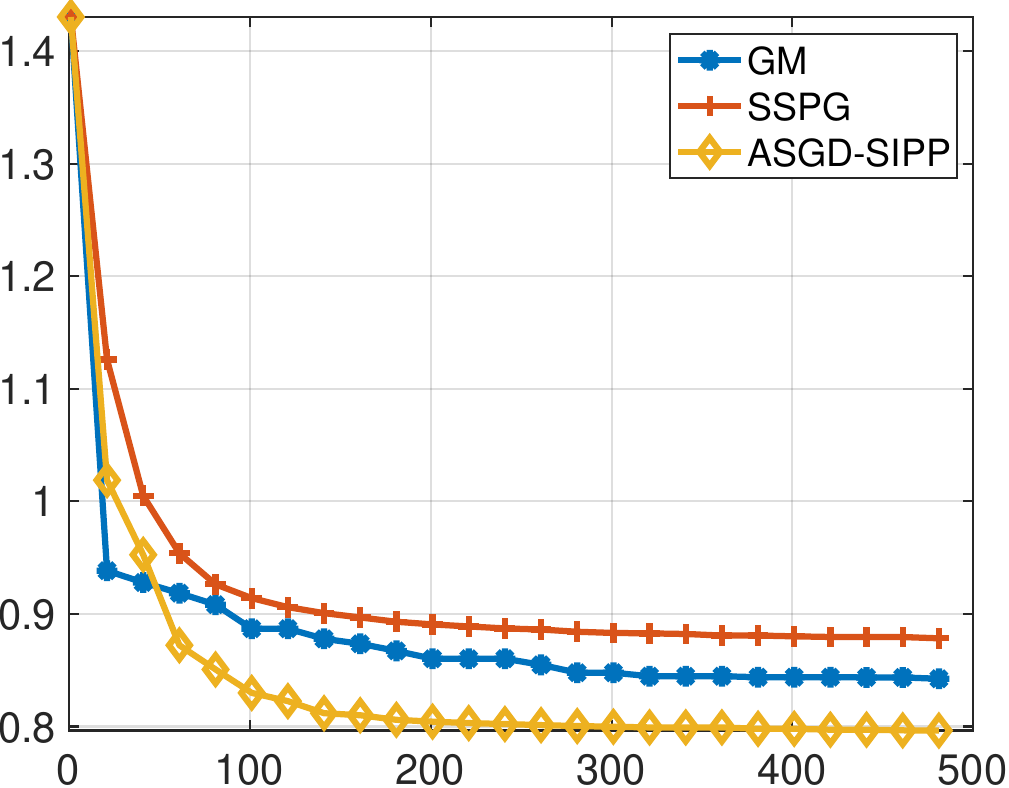}	
\caption{Deterministic problems. First row: $h(z) = z^2$; Second row: $h(z) = e^x + 10$; Within each row from left to right: $(\kappa, p) \in \{(1, 0), (10, 0), (1, 0.2), (10, 0.2)\}$. x-axis: iteration number. y-axis: $f(\xbf^k)$~.\label{fig-deterministic}}
\end{figure}
\Cref{fig-deterministic} illustrates the performance of different algorithms on the tested problems. As our theory suggests, we often observe that {\proxaccvr} outperforms {\gm} when the smoothing parameters are appropriately configured. In addition, even if {\sspg} does not yield improved complexity, its practical performance in terms of function value decrease is often more stable.
\vspace{-5pt}
\paragraph{Range of Optimal Stepsizes for Different Smoothing Approaches.}
Our theory predicts that Moreau envelope smoothing admits a different range of
admissible stepsizes than Nesterov smoothing. To examine this behavior empirically,
our second experiment studies the sensitivity of each method to the choice of
stepsize. \Cref{fig-det-robust} reports the number of iterations required to achieve
convergence under various stepsize configurations. As the figure illustrates, the admissible stepsizes for Moreau envelope smoothing are larger than for Nesterov smoothing, which is  consistent with our theoretical findings.

\begin{figure}[h]
\centering
\includegraphics[scale=0.22]{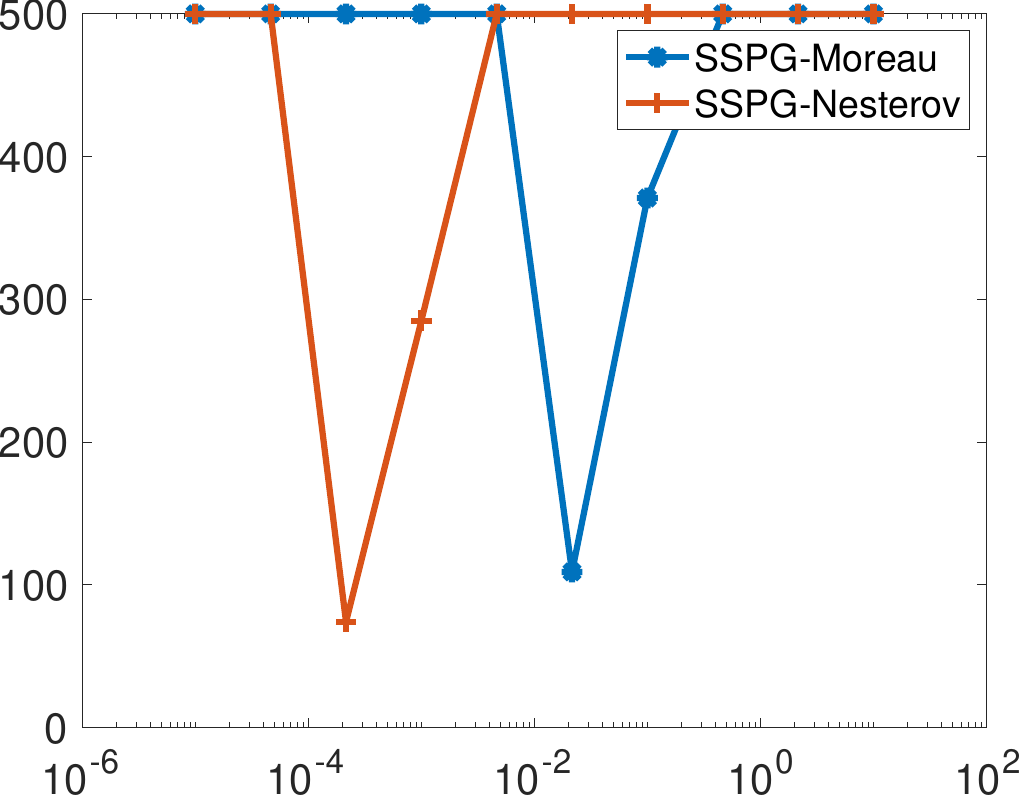}
\includegraphics[scale=0.22]{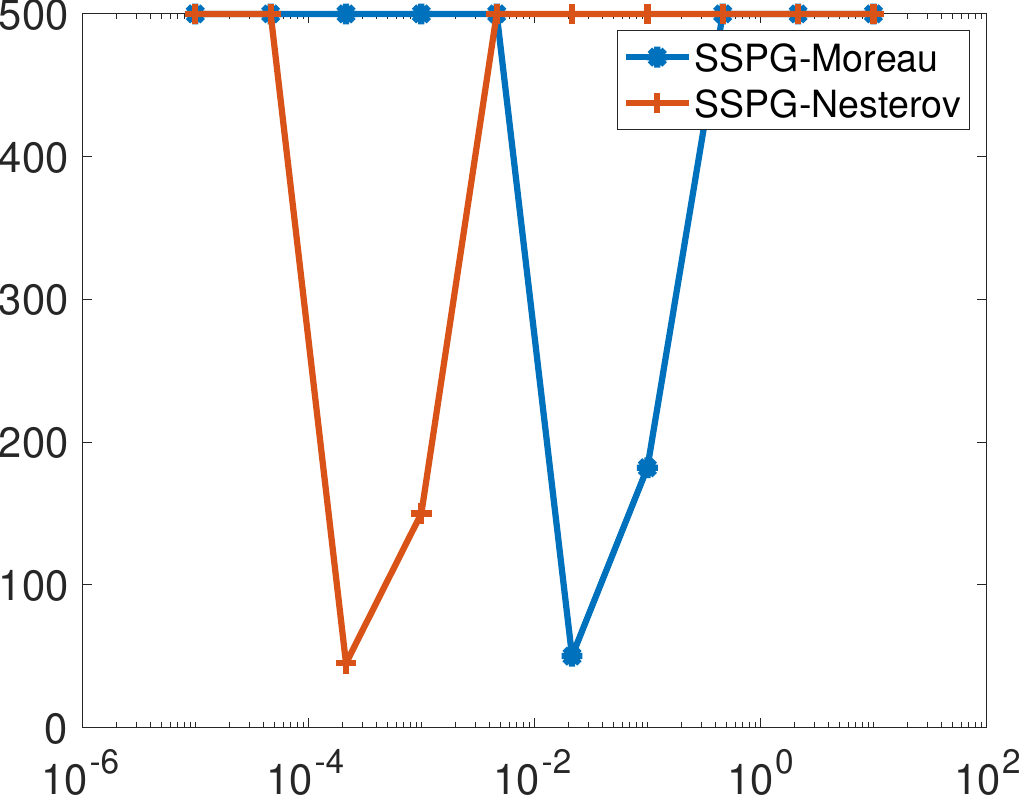}
\includegraphics[scale=0.22]{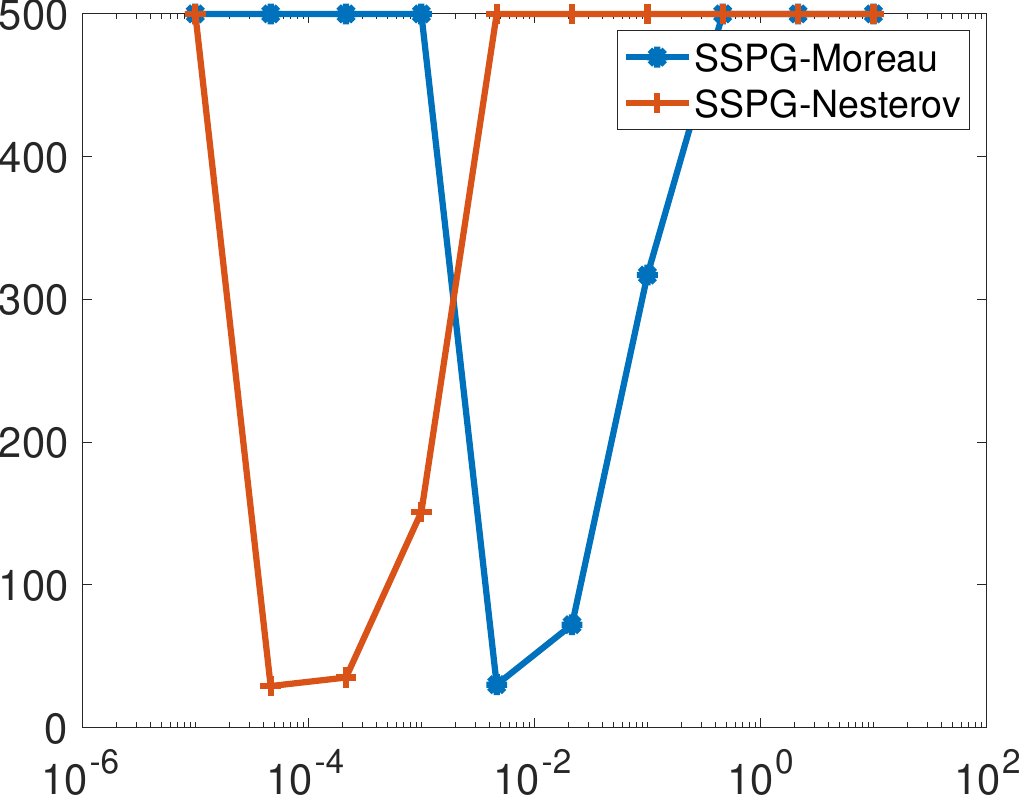}
\includegraphics[scale=0.22]{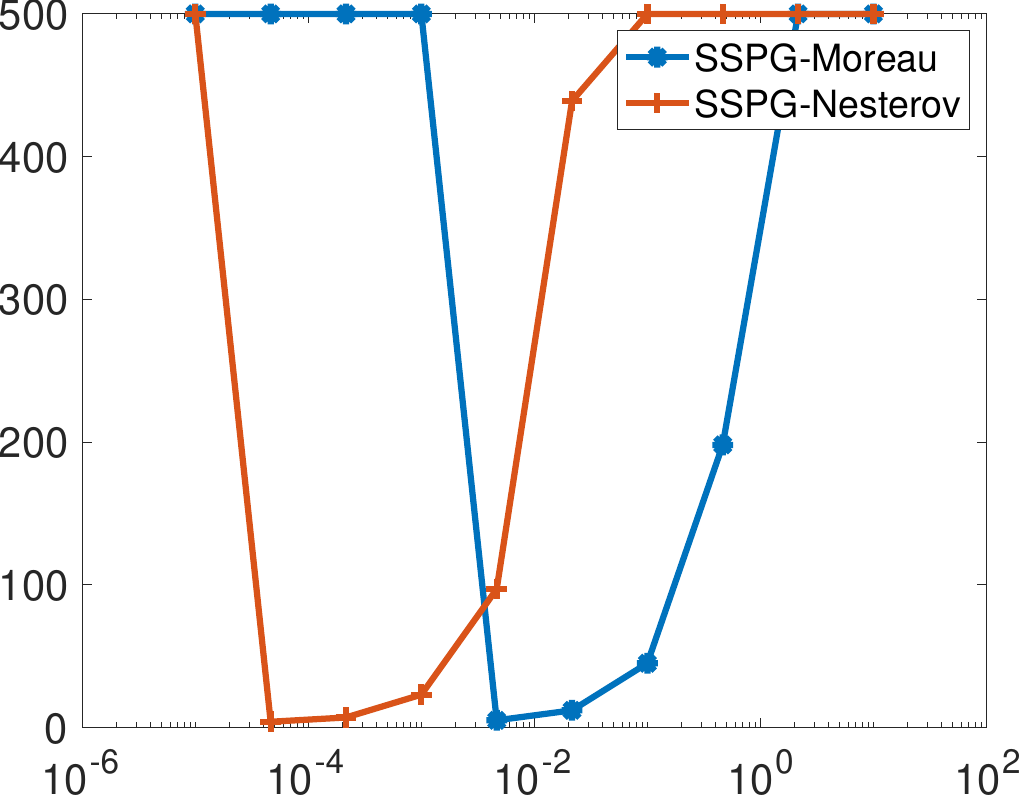}
\caption{Experiments comparing the range of optimal stepsize for different smoothing approaches\label{fig-det-robust}. x-axis: $\alpha_0$. y-axis: number of iterations to reach the stopping criterion.}
\end{figure}

\subsection{Experiments on stochastic problems}
 Let $\xi_k \sim \text{Uniform}(\{ \xi_1,
   \ldots, \xi_m \}) $ be a sample drawn uniformly at random.
We evaluate the following stochastic algorithms:
\begin{itemize}[leftmargin=15pt,itemsep=2pt,topsep=5pt]
\item \textit{Stochastic subgradient method} (\sgm). $\xbf^{k + 1} = \xbf^k - \alpha_k f' (\xbf^k, \xi^k)$
\item \textit{Stochastic gradient descent on smoothed function} ({\sspg}, \Cref{alg:Smoothed-spg}).
\item \textit{Inexact proximal point with stochastic Nesterov acceleration} ({\proxaccvr}, \Cref{alg:Smoothed-ipp}).
\end{itemize}
\begin{figure}[h!]
\centering
\includegraphics[scale=0.22]{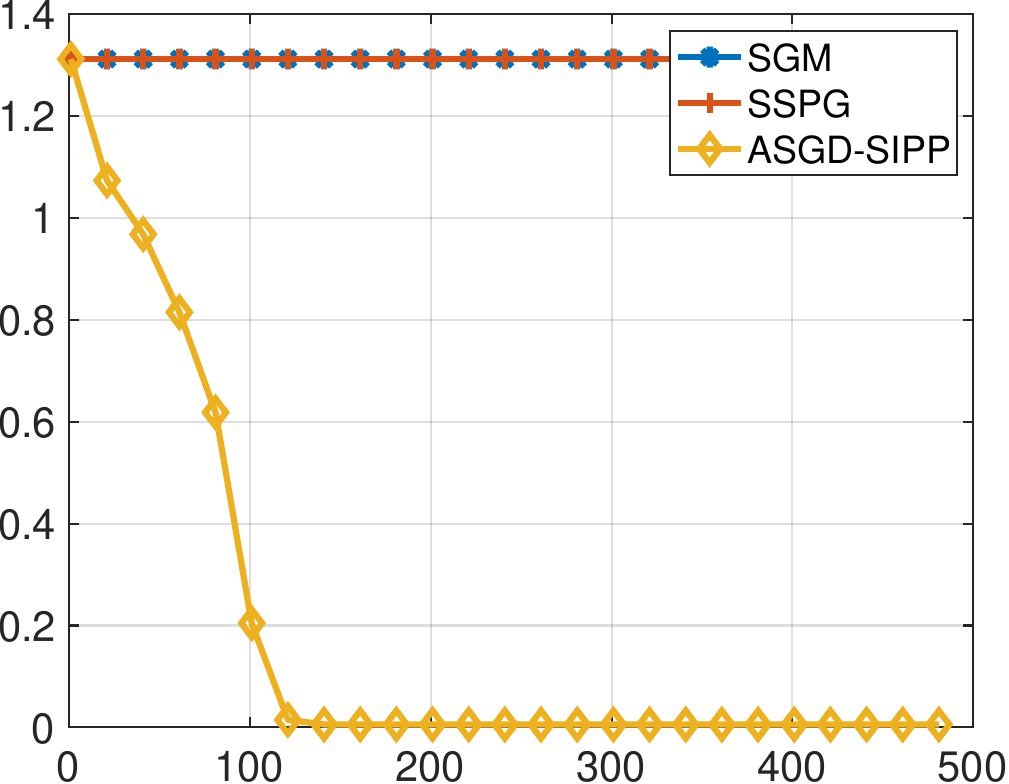}
\includegraphics[scale=0.22]{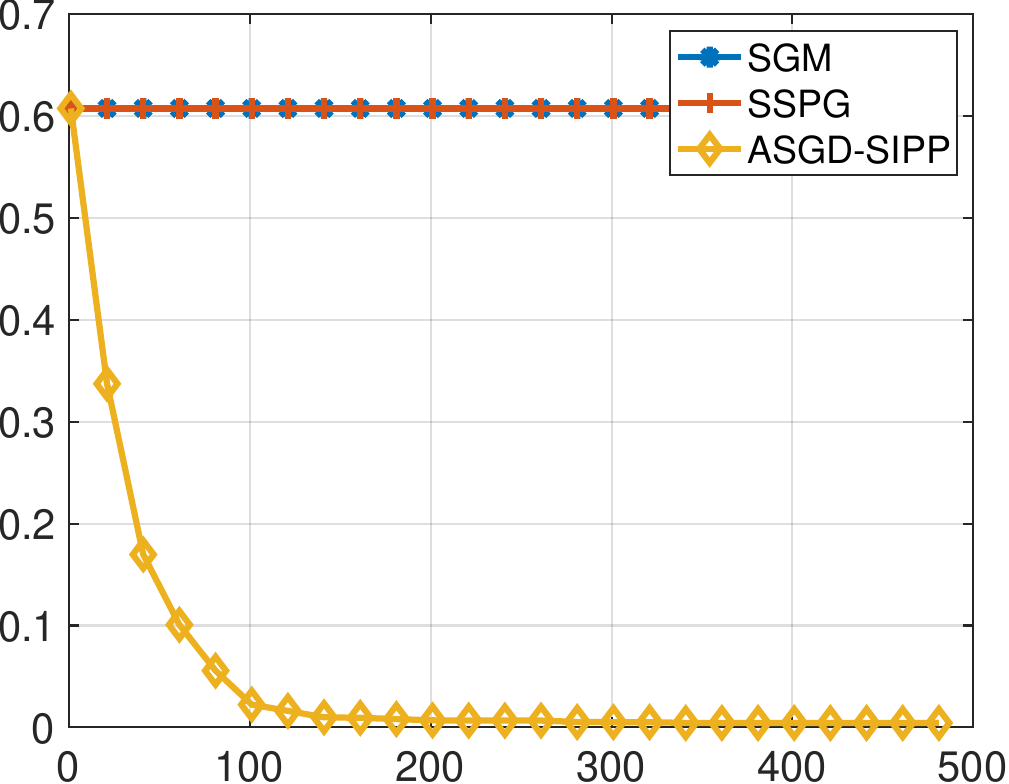}
\includegraphics[scale=0.22]{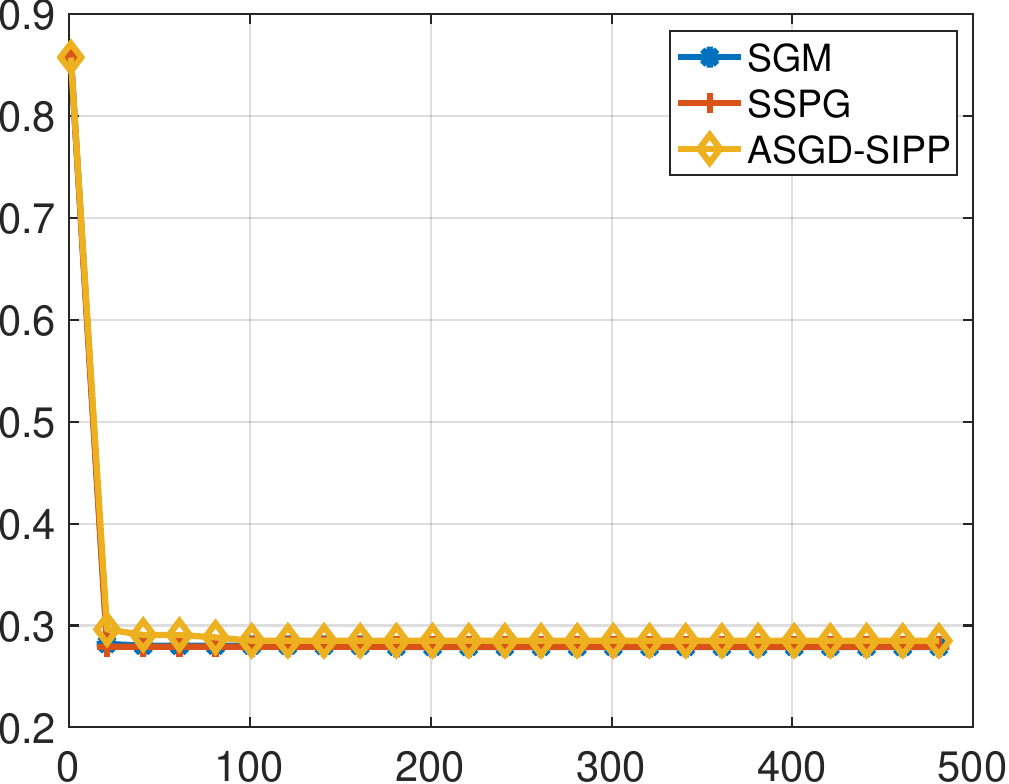}
\includegraphics[scale=0.22]{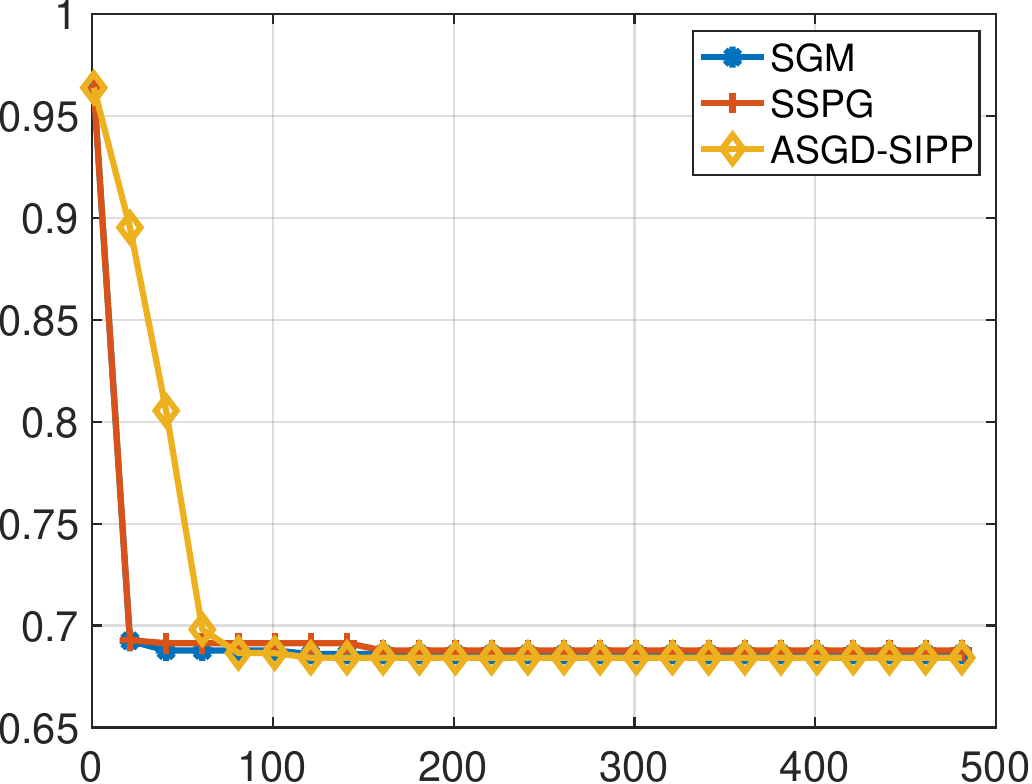}
\includegraphics[scale=0.22]{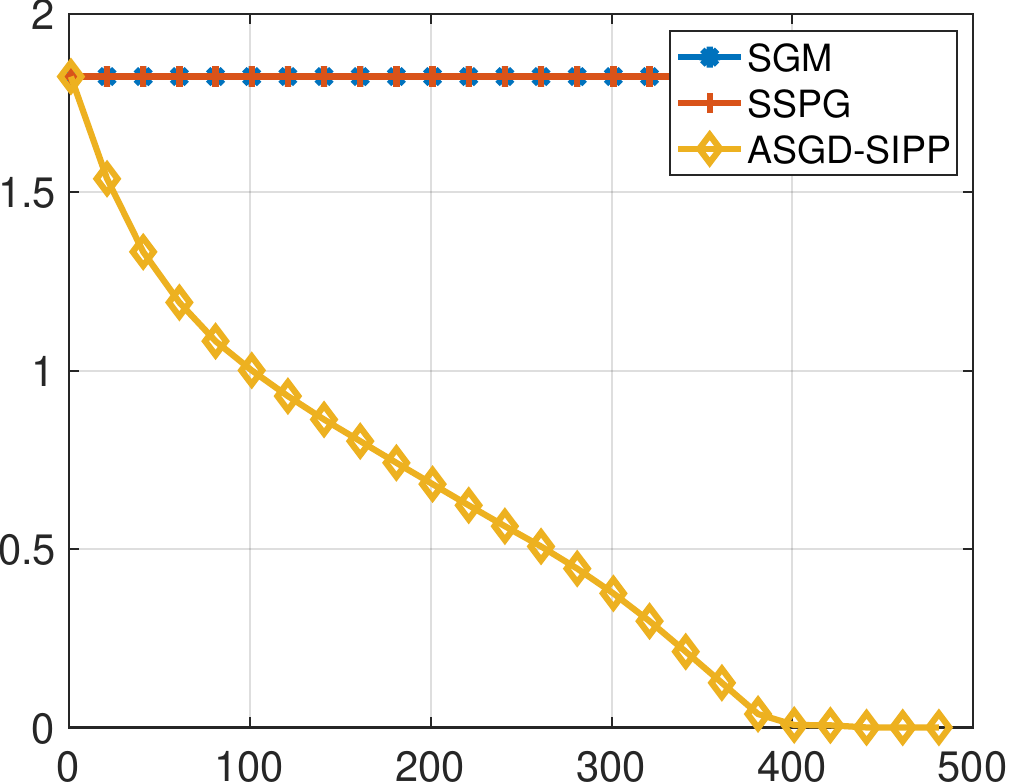}
\includegraphics[scale=0.22]{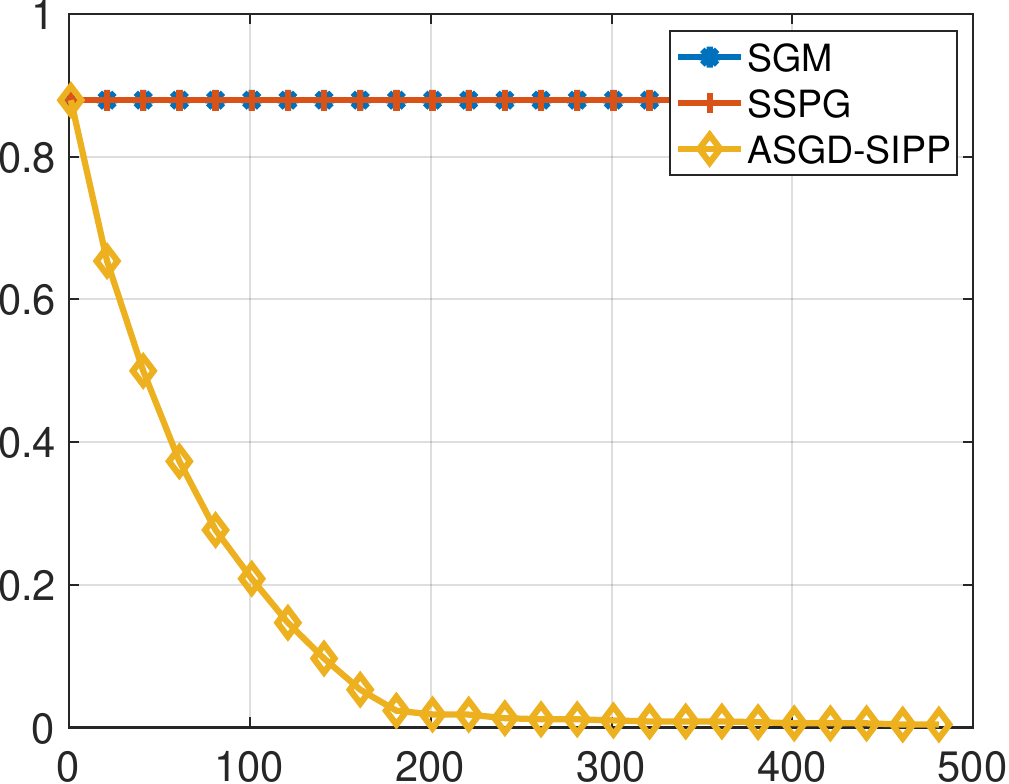}
\includegraphics[scale=0.22]{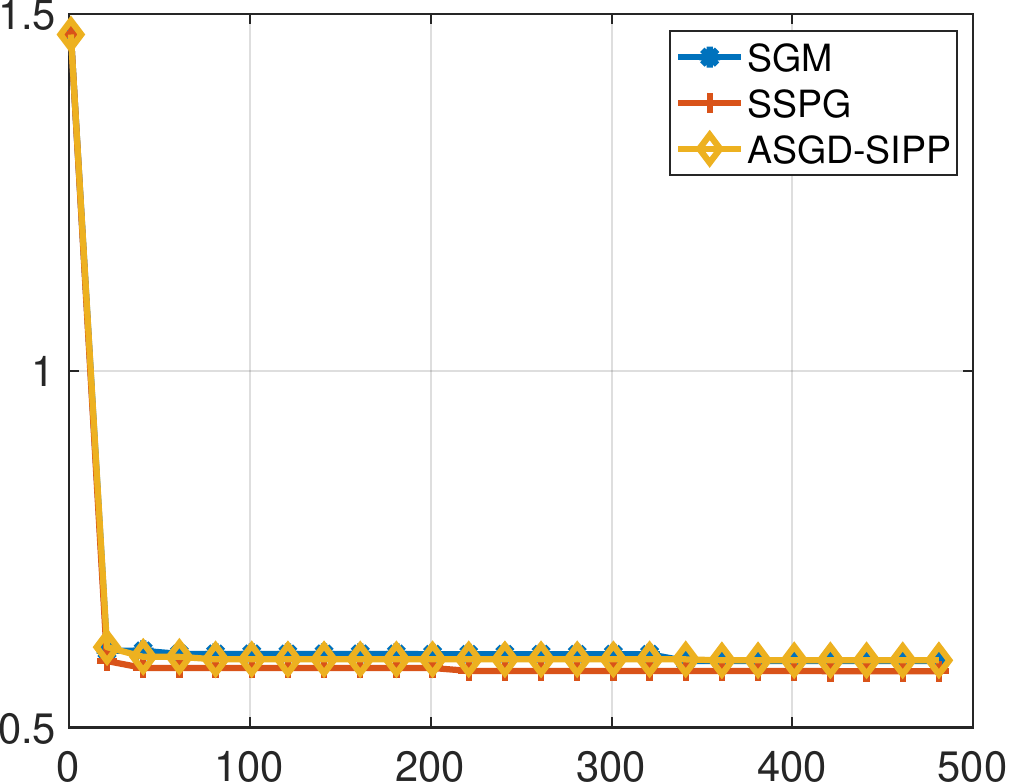}
\includegraphics[scale=0.22]{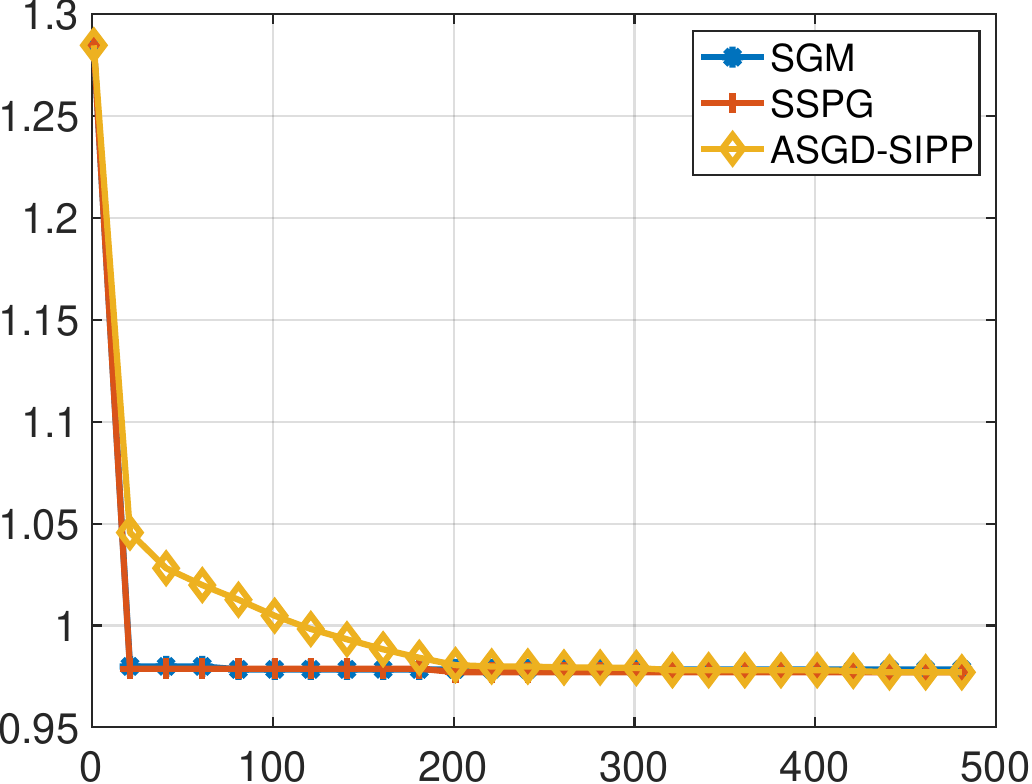}
\caption{Stochastic problems. First row: $h(z) = z^2$; Second row: $h(z) = z^5 + z^3 + 1$. Within each row from left to right: $(\kappa, p) \in \{(1, 0), (10, 0), (1, 0.2), (10, 0.2)\}$. x-axis: iteration number. y-axis: $f(\xbf^k)$.\label{fig-stochastic}}
\end{figure}
\Cref{fig-stochastic} summarizes the performance of the stochastic algorithms on the
tested problems. With appropriate choices of the smoothing parameter and the
minibatch size for the proximal-point subproblems, we observe that smoothing often
achieves significantly better performance than the vanilla {\sgm}. Although smoothing
does not universally accelerate convergence, it typically yields more robust behavior
across different instances. Notably, on certain problems, convergence is observed only
when {\aglssipp} is employed.
\vspace{-5pt}
\paragraph{Robustness of Moreau Envelope Smoothing.}
Moreau envelope smoothing exploits the full stochastic objective via a proximal
update, rather than relying on a first-order approximation as in
gradient-based methods. As a result, it is closely related to the stochastic proximal
point method, and we therefore expect it to inherit the stability properties of
proximal algorithms \citep{deng2021minibatch, asi2019stochastic}. \Cref{fig-robust}
reports the number of iterations required to meet the stopping criterion for different
values of $\alpha_0$. The results confirm that Moreau envelope smoothing is indeed more
robust than both {\sgm} and Nesterov smoothing, consistent with the existing theoretical findings.

\begin{figure}[h!]
\centering
\includegraphics[scale=0.22]{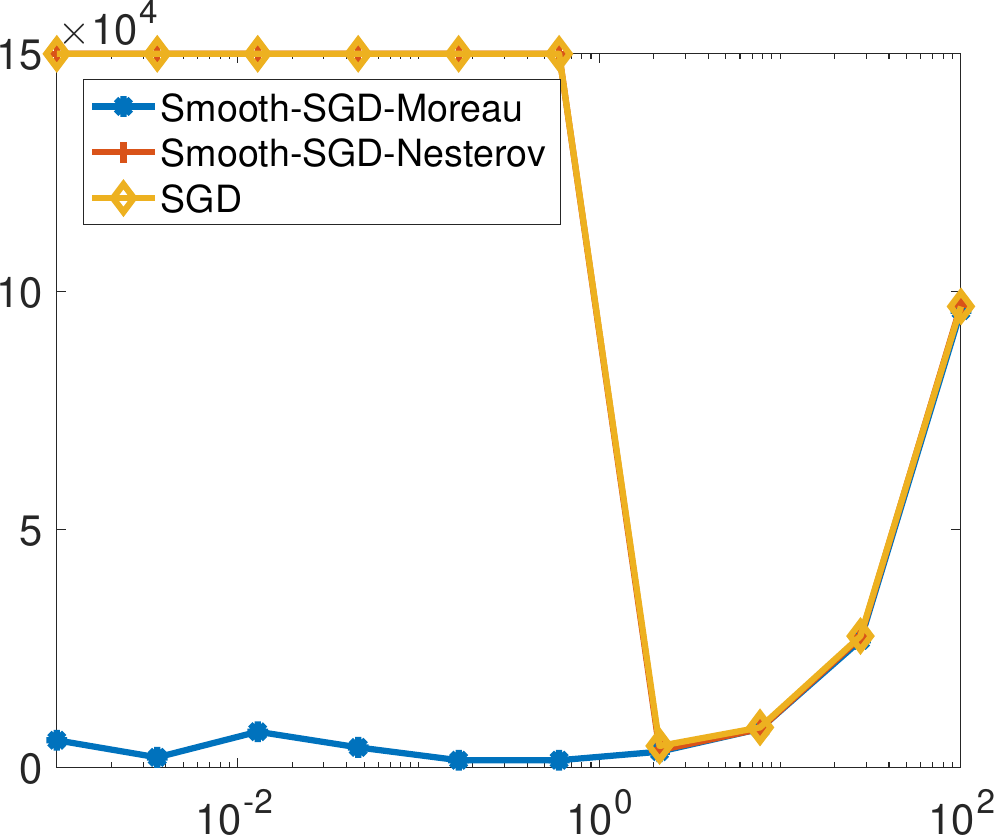}
\includegraphics[scale=0.22]{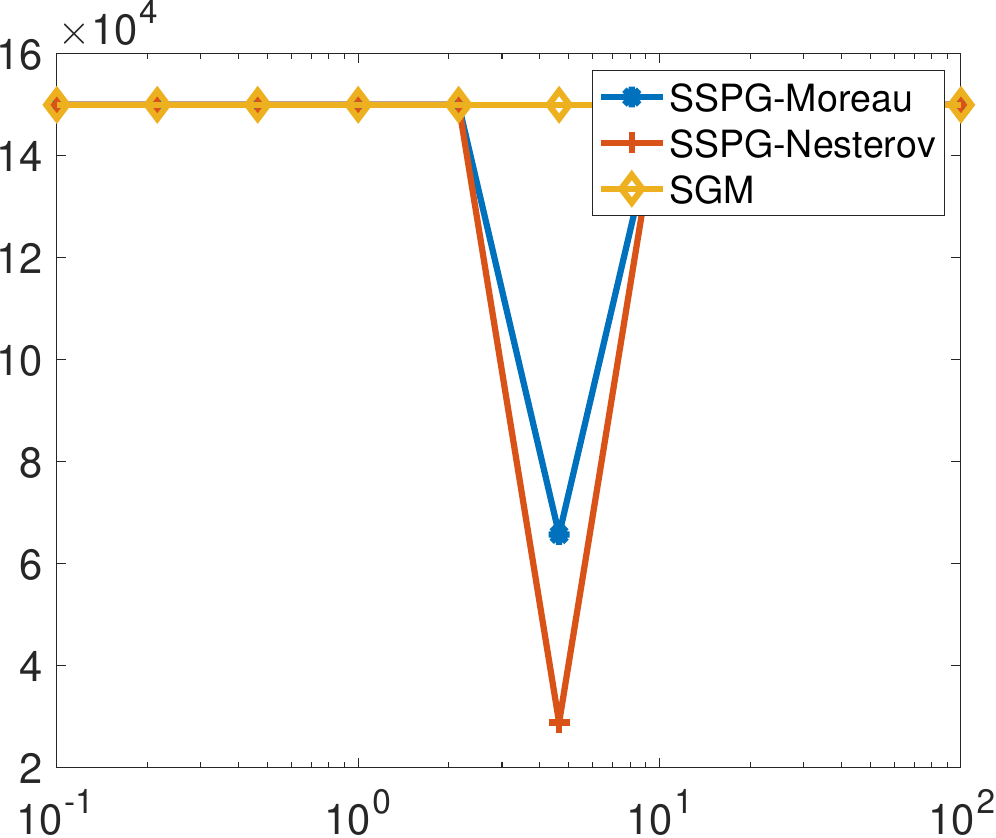}
\includegraphics[scale=0.22]{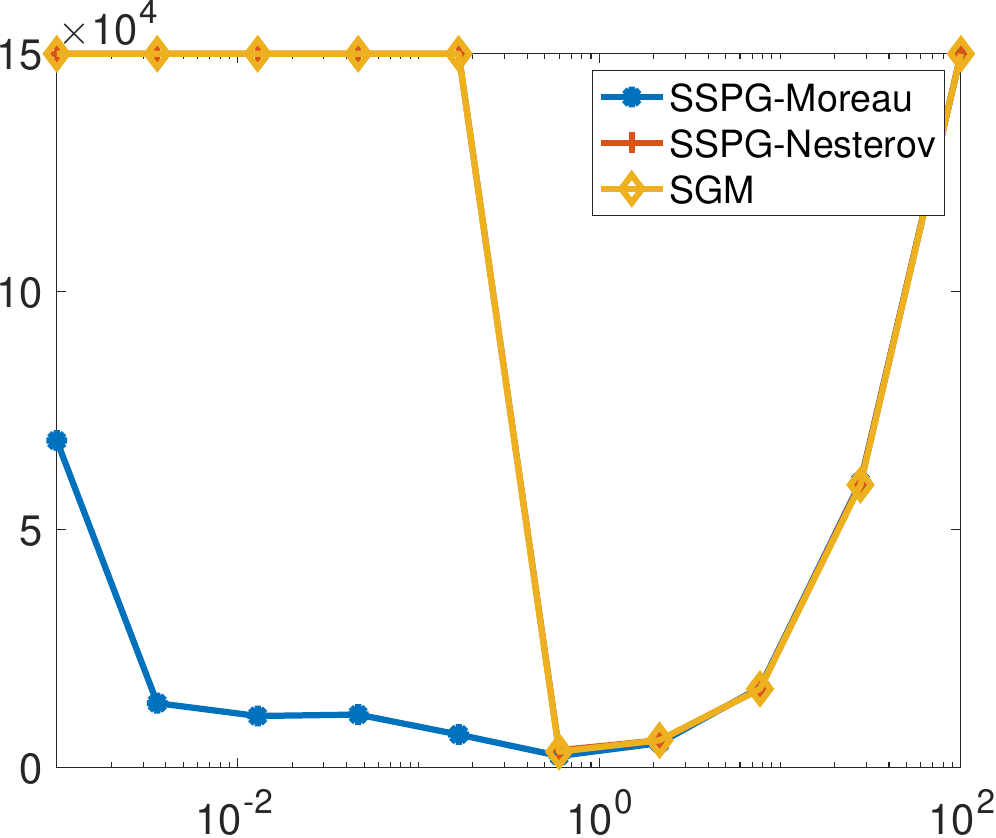}
\includegraphics[scale=0.22]{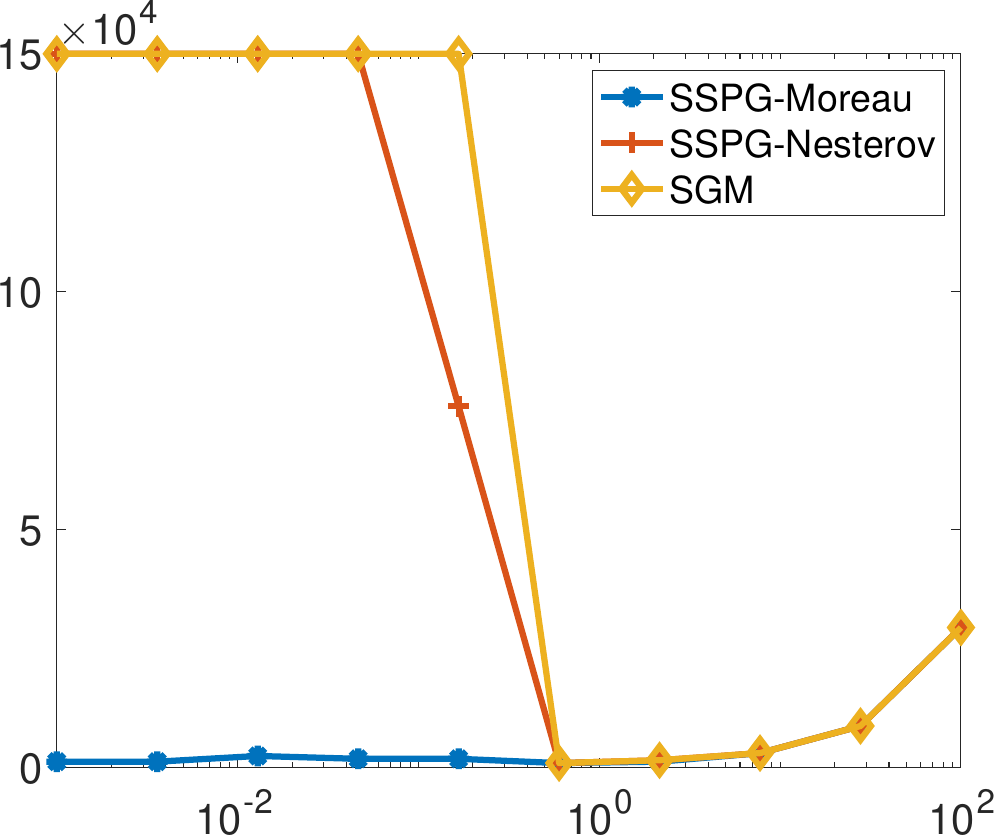}
\caption{Experiments comparing robustness of different smoothing approaches\label{fig-robust}. $x$-axis: $\alpha_0$. $y$-axis: number of iterations to reach the stopping criterion.}
\end{figure}

\subsection{Experiments on generalized Lipschitz problems}

This section reports additional experiments evaluating our smoothing algorithm on non-Lipschitz problems. In particular, we consider the piecewise quadratic problem
\begin{align*}
  \min_{\xbf \in \mathbb{R}^d}  \max_{1 \leq j \leq m} \{ \tfrac{1}{2}
  \langle \xbf, A_j \xbf \rangle - \langle \bbf_j, \xbf \rangle \}, & 
\end{align*}
where $\{A_i\}$ are symmetric matrices. 
\paragraph{Dataset generation} We generate $A_i = C C^{\top} - \sigma I$ with $C_{i j}
\sim \mathcal{N} (0, 1), \bbf \sim \mathcal{N} (0, I_d)$ and $\sigma \in \{0, 1\}$. When $\sigma = 0$, we generate $\{A_i \}$ such that at least one of them is positive definite, so that the objective function is coercive. We test $m \in \{ 5, 10 \}$ and $d \in \{ 20, 100 \}$.
As a result, the objective is not globally Lipschitz, but it satisfies \Cref{assum:level-bounded-phi-eta}.

\paragraph{Benchmark algorithms}
When $\sigma = 0$, we compare {\agls} applied to the softmax-smoothed objective with the normalized (sub)gradient method (\texttt{NSGM}) that is designed for generalized Lipschitz problems \cite{grimmer2019convergence}.
The \texttt{NSGM} update takes the form $\xbf^{k + 1} = \xbf^k - \eta \frac{f' (\xbf^k)}{\| f' (\xbf^k) \|}$.
When $\sigma = 1$, we use {\aglssipp} as the benchmark algorithm. As \Cref{fig-generalized-smooth} illustrates, on non-Lipschitz optimization tasks, our method exhibits competitive performance in practice compared to the subgradient-based method.
\begin{figure}[h]
\centering
\includegraphics[scale=0.21]{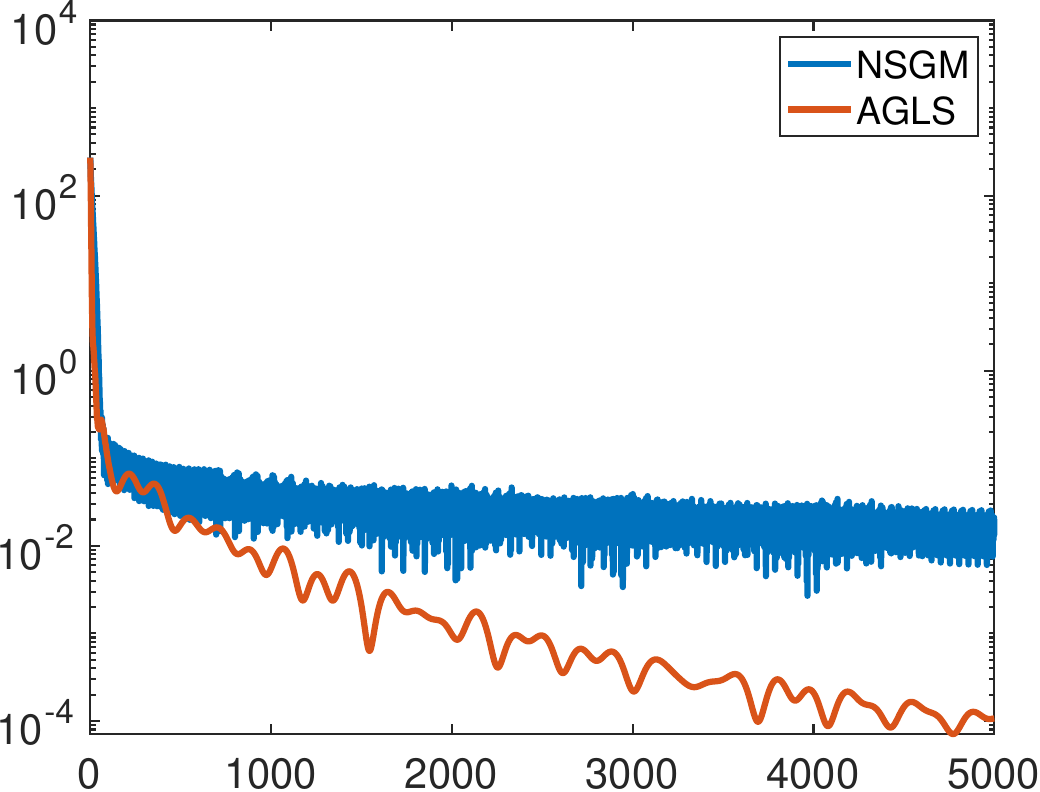}
\includegraphics[scale=0.21]{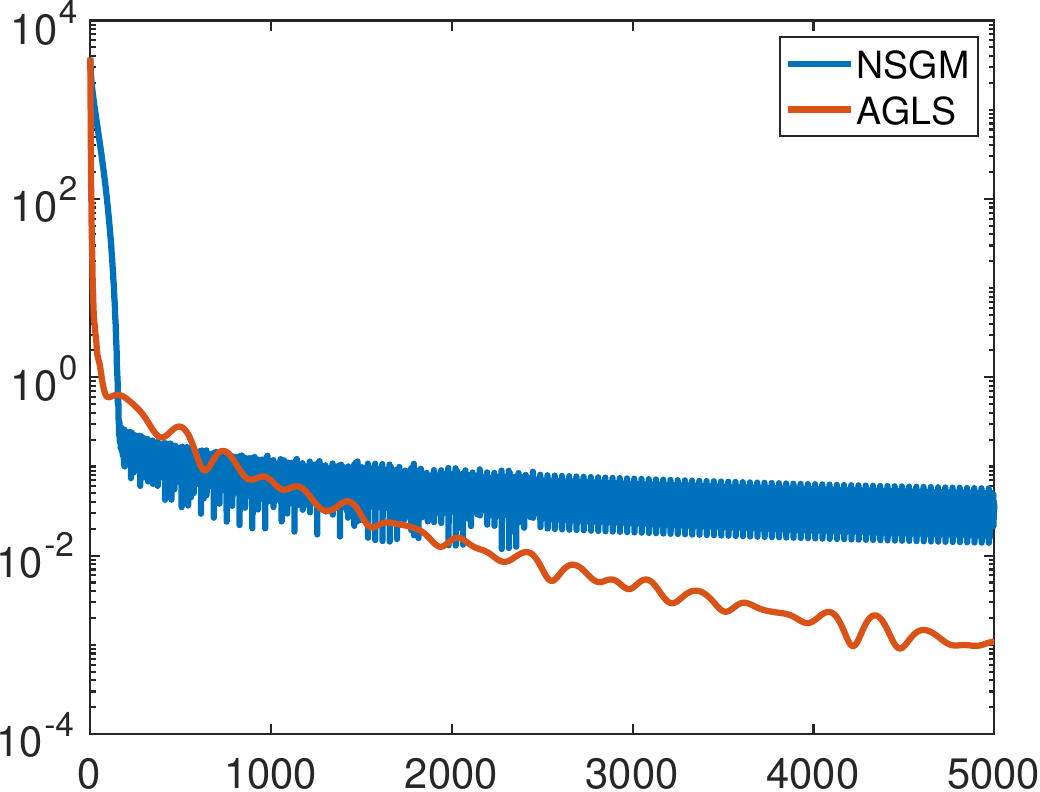}
\includegraphics[scale=0.21]{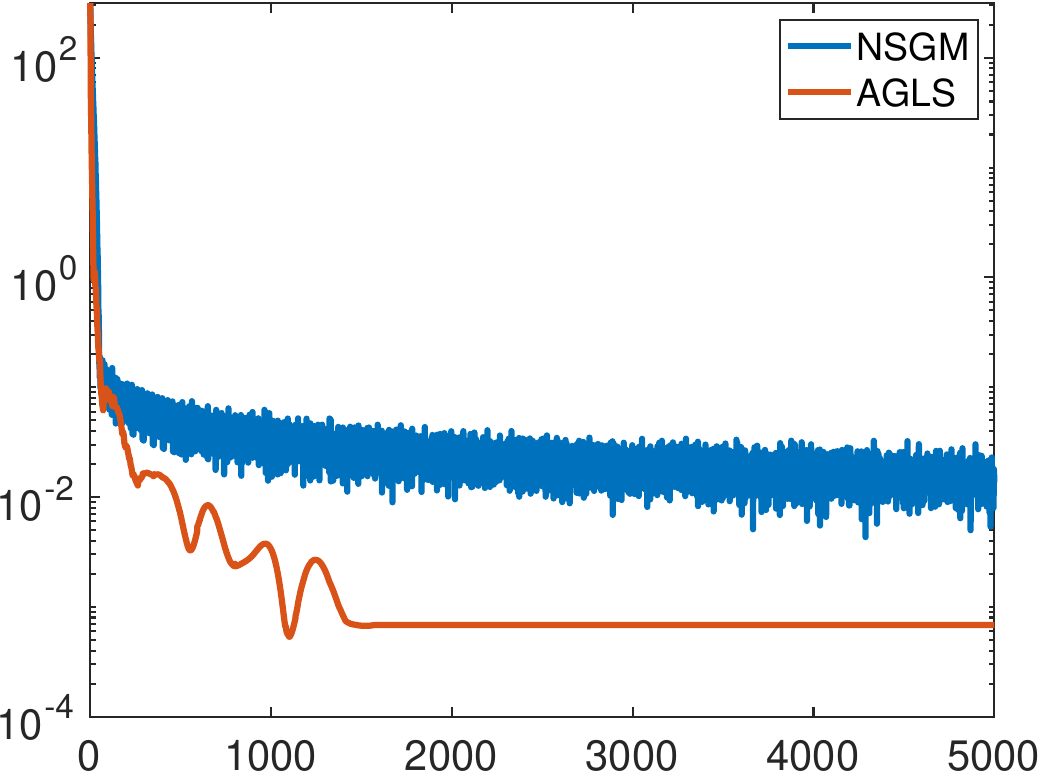}
\includegraphics[scale=0.21]{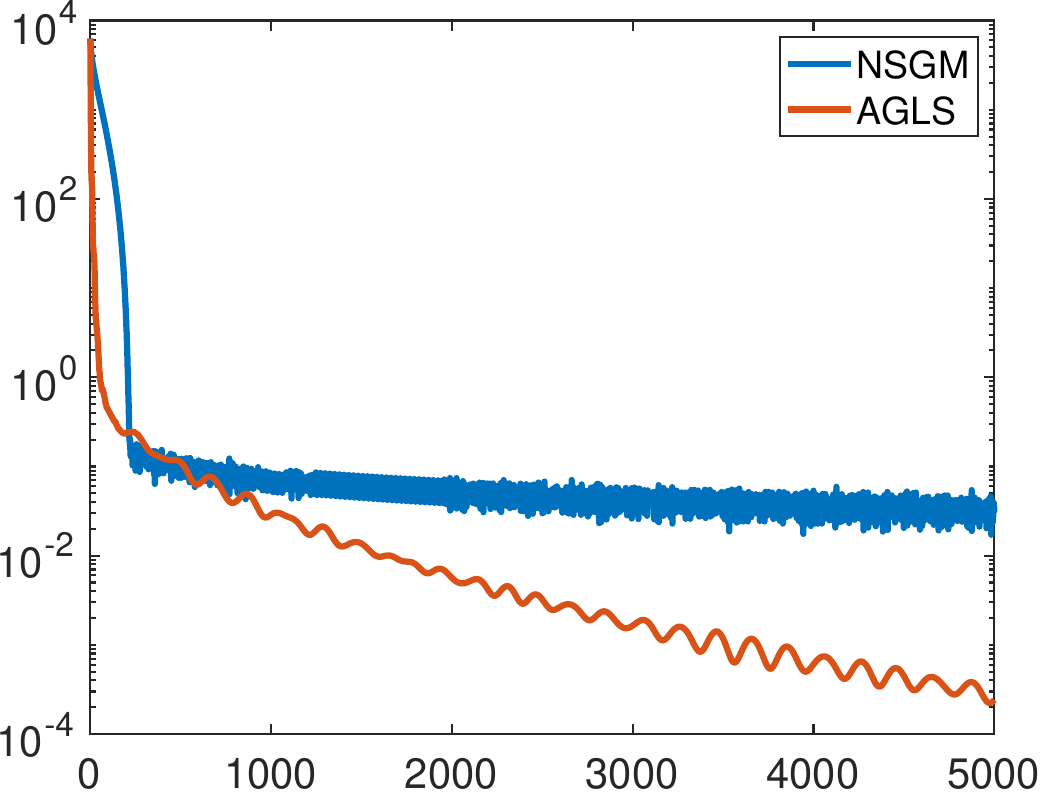}
\includegraphics[scale=0.22]{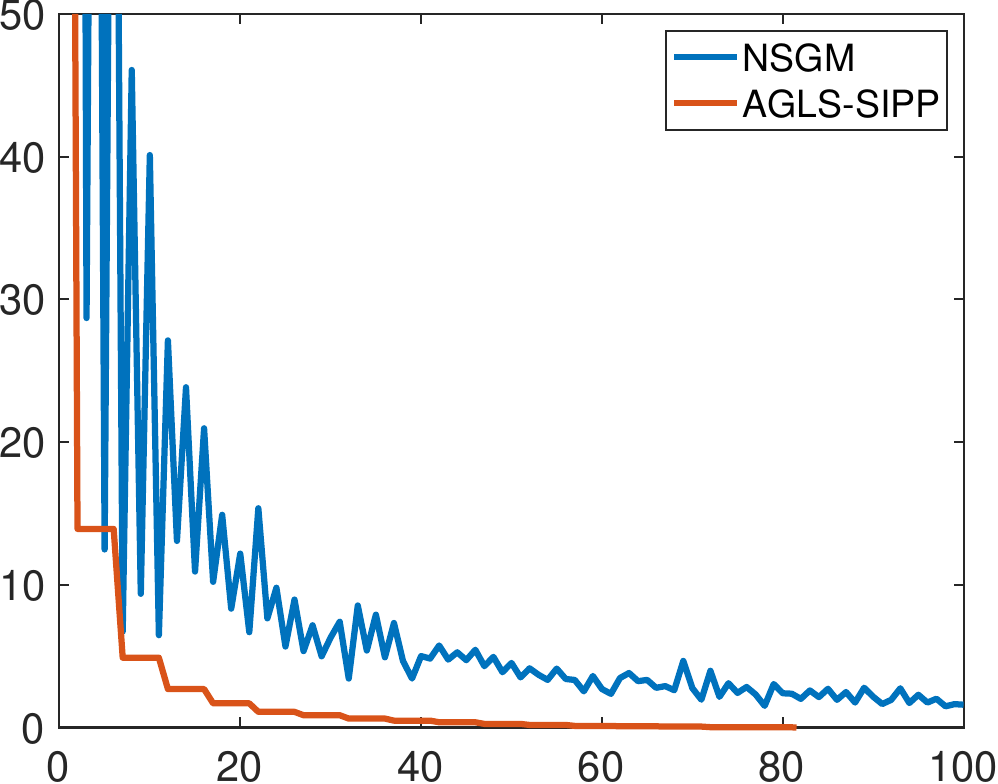}
\includegraphics[scale=0.22]{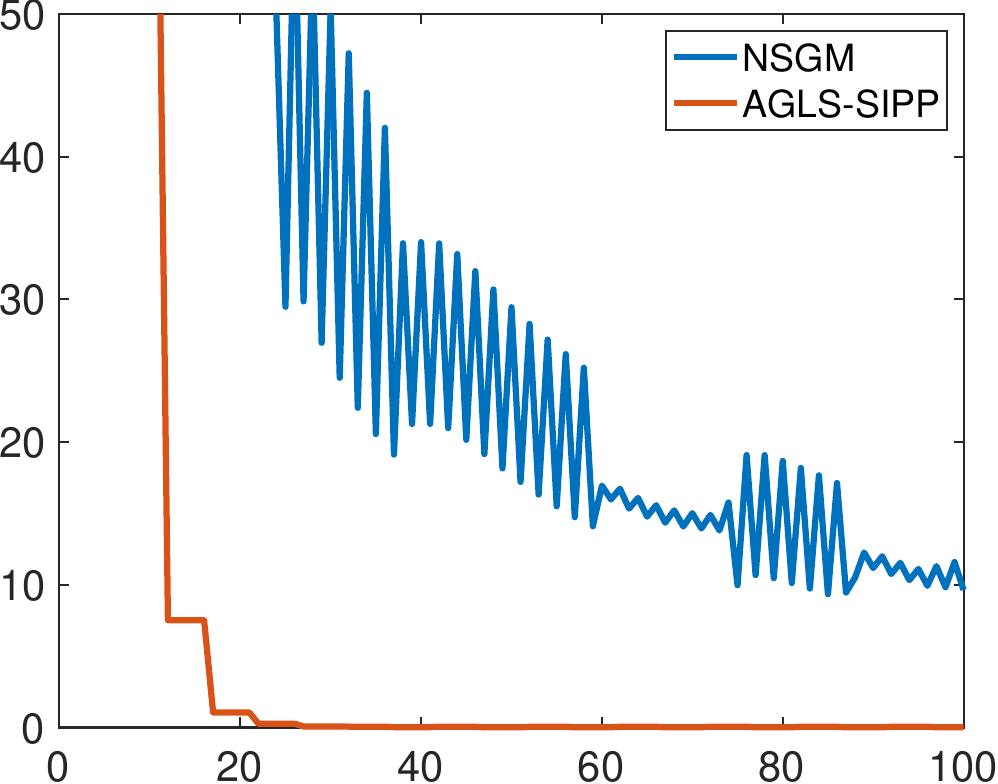}
\includegraphics[scale=0.22]{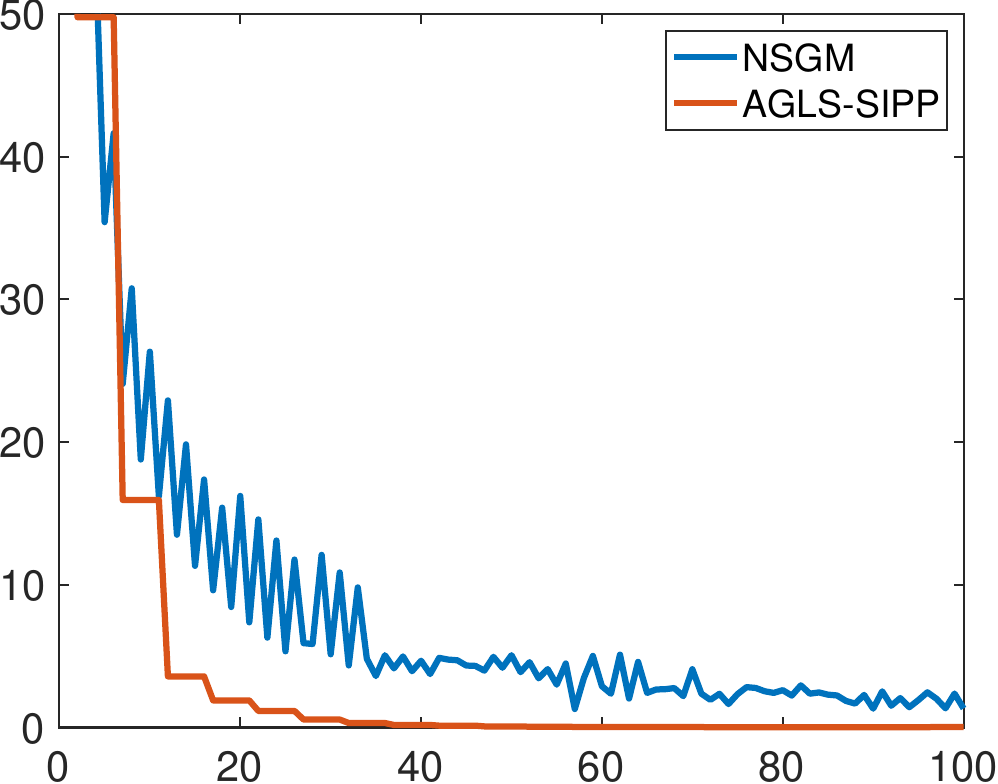}
\includegraphics[scale=0.22]{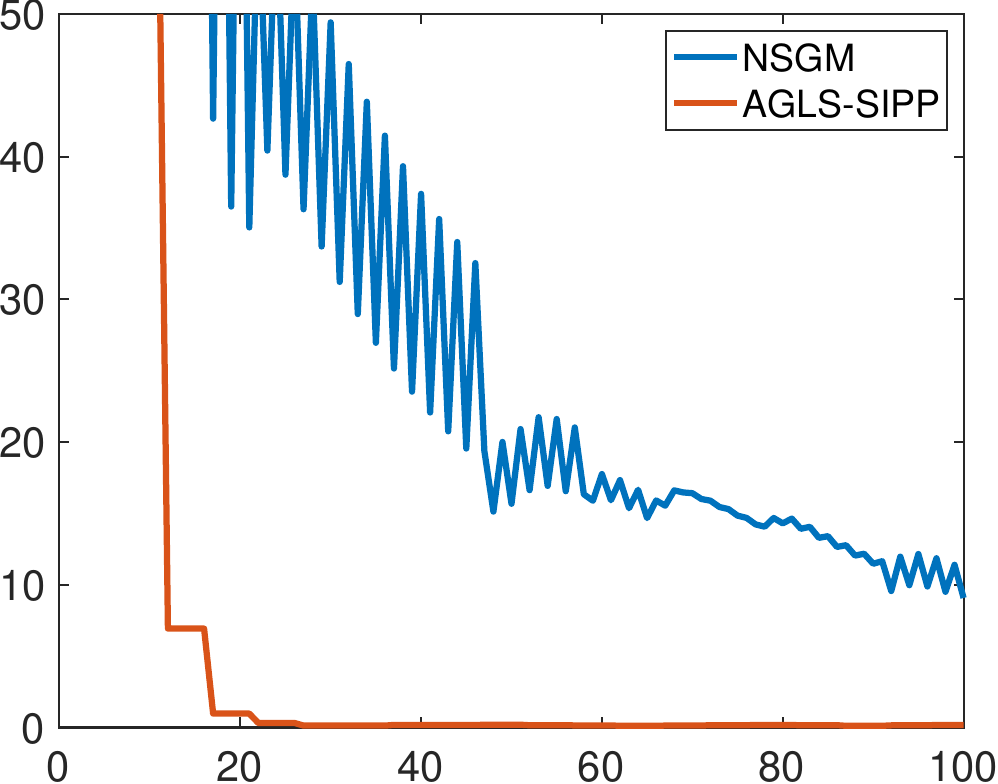}
\caption{Experiments on the comparison between {\agls} (\aglssipp) and \texttt{NSGM} for generalized smooth problems. The first row corresponds to $\sigma = 0$ and the second row corresponds to $\sigma = 1$. From left to right: $(m, d) \in \{(5, 20), (5, 100), (10, 20), (10, 100)\}$. $x$-axis: iteration number. $y$-axis: $f(\xbf^k)$  (or $f(\xbf^k) - f^\star$ for convex problems) .\label{fig-generalized-smooth}}
\end{figure}

\section{Conclusions}
We introduced a general smoothing framework for weakly convex optimization that unifies and extends approaches such as Nesterov-type smoothing and Moreau-envelope smoothing. Our analysis provides a unified complexity theory for both deterministic and stochastic settings. By applying an inexact proximal point scheme to the smooth approximations, we improve the deterministic complexity for achieving an $\varepsilon$-approximate stationary point from $\Ocal(1/\varepsilon^{4})$ to $\Ocal(1/\varepsilon^{3})$. We also establish a complexity of $\Ocal(\max\{1/\varepsilon^{3}, 1/(m\varepsilon^{4})\})$ in the stochastic setting. 
Furthermore, the proposed line search accelerated method enables an $\Ocal(1/\varepsilon^{3})$ complexity without requiring global smoothness. 
Several promising directions remain for future research. One avenue is to explore additional smoothing techniques, such as randomized smoothing via Gaussian convolution. From a practical standpoint, developing adaptive strategies for choosing the smoothing parameter and the proximal regularization may yield further speedups.

\renewcommand \thepart{}
\renewcommand \partname{}

\bibliographystyle{abbrvnat}
\bibliography{ref}

\doparttoc
\faketableofcontents
\part{}

\newpage
\appendix

\addcontentsline{toc}{section}{Appendix}

\part{Appendix} 
\parttoc

\section{Auxiliary results}
\paragraph{Local Lipschitzness and subgradient bound}
We establish some connection between local Lipschitz continuity and the subgradient norm bound. 
\begin{lem}
\label{lem:lip-bdd-subgrad}Let $g: \Ebb\raw \Rbb\cup\{+\infty\}$ be a proper weakly
convex function and $\Scal\subseteq \dom g$ be a convex open set.  Then the following two claims are equivalent. 
\begin{enumerate}
\item[a).] There exists $L>0$ such that $\vert g(\xbf)-g(\ybf)\vert\le L\norm{\xbf-\ybf}$
for any $\xbf,\ybf\in\Scal$. 
\item[b).] There exists $L>0$ such that $\norm{\vbf}\le L$ for any
$\vbf\in\partial g(\xbf)$ and $\xbf\in\Scal$.
\end{enumerate}
\end{lem}

\begin{proof}
First, we show ``$b \Raw a$''. Since $g$ is weakly convex on $\Scal$, it is locally Lipschitz continuous (see \citet[Lemma 4.4.1]{cui2021modern}). 
  By the mean value theorem~\citep[Theorem 7.44]{mordukhovich2022convex}, we have 
\[
g(\xbf)-g(\ybf)  = \inner{\zeta}{\xbf-\ybf}, 
\]
for some $\zeta\in\partial g(\theta\xbf+(1-\theta)\ybf)$, $\theta\in(0,1)$. Applying Cauchy-Schwarz's inequality, we immediately obtain part~a.

Next, we show ``$a \Raw b$''. The idea follows closely the proof of a similar result for a convex function~\citep[Theorem 3.61]{beck2017first}. Since $\Scal$ is an open set, for any $\xbf\in\Scal$,
$\vbf\in\partial g(\xbf)$, there exists an $\bar{\vep}\in(0,\infty)$
such that for any $\vep\in(0,\bar{\vep})$, the neighborhood $N_{\vep}(\xbf)=\{\ybf:\norm{\ybf-\xbf}\le\vep\}\subseteq\Scal$.
Moreover, let $\vbf^\dag\in\Ebb$ be a vector such that $\norm{\vbf^\dag}=1$, and $\inner{\vbf}{\vbf^\dag}=\norm{\vbf}_*$.
Therefore, $\xbf+\vep\vbf^\dag\in\Scal$. Using Lipschitz
continuity and weak convexity (of modulus $\mu$, $\mu>0$), we have 
\[
L\vep=L\norm{\vep\vbf^\dag}\ge g\rbra{\xbf+\vep\vbf^\dag}-g(\xbf)\ge\inner{\vbf}{\vep\vbf^\dag}-\frac{\mu}{2}\norm{\vep\vbf^\dag}^{2}=\norm{\vbf}\vep-\frac{\mu}{2}\vep^{2}.
\]
Dividing both sides by $\vep$ gives $\norm{\vbf}-\frac{\mu}{2}\vep\le L.$
Since $\vep$ can be arbitrarily small, we conclude that $\norm{\vbf}\le L$,
thereby completing the proof.
\end{proof}

\paragraph{Three-point inequality}
The following three-point property is known, and it will be used several times in our analysis.
\begin{lem}\label{lem:three-point}
Let $f$ be a $\rho$-weakly convex function and define $\ybf_{\xbf} = \prox_{f/\beta}(\xbf)$ for some $\beta\ge\rho$. Then, for any $\zbf \in \reals^{d}$,
\begin{equation}
\begin{aligned}
f(\zbf) - f(\ybf_{\xbf}) \geq \frac{\beta}{2}\norm{\xbf-\ybf_{\xbf}}^{2} + \frac{\beta-\rho}{2}\norm{\zbf-\ybf_{\xbf}}^{2} - \frac{\beta}{2}\norm{\zbf-\xbf}^{2}.
\end{aligned}
\label{eq:opt-01}
\end{equation}
If $f$ is $\mu$-strongly convex ($\mu>0$), then the inequality~\eqref{eq:opt-01} holds with $\rho=-\mu$.
\end{lem}

\begin{proof}
Observe that the function $f(\zbf) + \frac{\beta}{2}\norm{\zbf-\xbf}^2$ is $(\beta-\rho)$-strongly convex with respect to $\zbf$. Consequently, for any subgradient $\vbf \in \partial f(\ybf_\xbf) + \beta(\ybf_\xbf - \xbf)$, the following inequality holds:
\[
f(\zbf) + \frac{\beta}{2}\norm{\zbf-\xbf}^2 \geq f(\ybf_\xbf) + \frac{\beta}{2}\norm{\ybf_\xbf-\xbf}^2 + \inner{\vbf}{\zbf-\ybf_\xbf} + \frac{\beta-\rho}{2}\norm{\zbf-\ybf_\xbf}^2.
\]
The optimality condition yields $\zerobf \in \partial f(\ybf_\xbf) + \beta(\ybf_\xbf - \xbf)$. Substituting $\vbf = \zerobf$ into the above inequality establishes the desired result.
\end{proof}

\section{Missing Proofs}

\subsection{Proof of Lemma \ref{lem:moreau}}

\textsf{Proof of relation \eqref{eqn:moreau-gradient}.} For any $\xbf,\hat{\xbf}\in\reals^{d}$, denote $\ybf_{\xbf}=\prox_{f/\beta}(\xbf)$
and $\ybf_{\hat{\xbf}}=\prox_{f/\beta}(\hat{\xbf})$. In view of \Cref{lem:three-point}, for any $\ybf\in\reals^{d}$, 
we have 
\begin{equation}
\begin{aligned}f(\ybf) & \ge f(\ybf_{\xbf})+\frac{\beta}{2}\norm{\ybf_{\xbf}-\xbf}^{2}-\frac{\beta}{2}\norm{\ybf-\xbf}^{2}+\frac{\beta-\rho}{2}\norm{\ybf-\ybf_{\xbf}}^{2}\\
 & =f(\ybf_{\xbf})+\beta\inner{\xbf-\ybf_{\xbf}}{\ybf-\ybf_{\xbf}}-\frac{\rho}{2}\norm{\ybf-\ybf_{\xbf}}^{2}.
\end{aligned}
\label{eq:opt-04}
\end{equation}
Plugging $\ybf=\ybf_{\hat{\xbf}}$ into \eqref{eq:opt-04} and switching the roles of $\xbf$ and $\hat\xbf$, we obtain the following two relations:
\vspace{-20pt}
\begin{align}
f(\ybf_{\hat{\xbf}}) & \ge f(\ybf_{\xbf})+\beta\inner{\xbf-\ybf_{\xbf}}{\ybf_{\hat{\xbf}}-\ybf_{\xbf}}-\frac{\rho}{2}\norm{\ybf_{\xbf}-\ybf_{\hat{\xbf}}}^{2},\label{eq:opt-05} \\
f(\ybf_{\xbf}) & \ge f(\ybf_{\hat{\xbf}})+\beta\inner{\hat{\xbf}-\ybf_{\hat{\xbf}}}{\ybf_{\xbf}-\ybf_{\hat{\xbf}}}-\frac{\rho}{2}\norm{\ybf_{\xbf}-\ybf_{\hat{\xbf}}}^{2}.
\label{eq:opt-06}
\end{align}
Using \eqref{eq:opt-05} again, we have
\begin{equation}
\begin{aligned}f^{\beta}(\hat{\xbf})-f^{\beta}(\xbf) & =f(\ybf_{\hat{\xbf}})+\frac{\beta}{2}\norm{\ybf_{\hat{\xbf}}-\hat{\xbf}}^{2}-f(\ybf_{\xbf})-\frac{\beta}{2}\norm{\ybf_{\xbf}-\xbf}^{2}\\
 & \ge\beta\inner{\xbf-\ybf_{\xbf}}{\ybf_{\hat{\xbf}}-\ybf_{\xbf}}+\frac{\beta}{2}\norm{\ybf_{\hat{\xbf}}-\hat{\xbf}}^{2}-\frac{\beta}{2}\norm{\ybf_{\xbf}-\xbf}^{2}-\frac{\rho}{2}\norm{\ybf_{\xbf}-\ybf_{\hat{\xbf}}}^{2}\\
 & =\beta\inner{\xbf-\ybf_{\xbf}}{\hat{\xbf}-\xbf}+\beta\inner{\xbf-\ybf_{\xbf}}{\ybf_{\hat{\xbf}}-\hat{\xbf}-(\ybf_{\xbf}-\xbf)}\\
 & \quad+\frac{\beta}{2}\norm{\ybf_{\hat{\xbf}}-\hat{\xbf}}^{2}-\frac{\beta}{2}\norm{\ybf_{\xbf}-\xbf}^{2}-\frac{\rho}{2}\norm{\ybf_{\xbf}-\ybf_{\hat{\xbf}}}^{2}\\
 & =\beta\inner{\xbf-\ybf_{\xbf}}{\hat{\xbf}-\xbf}+\frac{\beta}{2}\norm{\ybf_{\hat{\xbf}}-\hat{\xbf}-(\ybf_{\xbf}-\xbf)}^{2}-\frac{\rho}{2}\norm{\ybf_{\xbf}-\ybf_{\hat{\xbf}}}^{2}.
\end{aligned}
\label{eq:opt-07}
\end{equation}
It follows that
\[
\begin{aligned}
    &\frac{\beta}{2}\norm{\ybf_{\hat{\xbf}}-\hat{\xbf}-(\ybf_{\xbf}-\xbf)}^{2}  -\frac{\rho}{2}\norm{\ybf_{\xbf}-\ybf_{\hat{\xbf}}}^{2} 
    \\
    & = \frac{\beta-\rho}{2}\norm{\ybf_{\xbf}-\ybf_{\hat{\xbf}}}^2 -\beta\inner{\ybf_{\xbf}-\ybf_{\hat{\xbf}}}{\xbf-\hat\xbf} +\frac{\beta}{2} \norm{\xbf-\hat\xbf}^2\\
    & = \frac{\beta-\rho}{2}\bnorm{\ybf_\xbf-\ybf_{\hat\xbf} - \frac{\beta}{\beta-\rho}(\xbf-\hat\xbf)}^2  - \frac{\rho}{2(1-\rho/\beta)} \norm{\xbf-\hat\xbf}^2 \\
    &\ge -\frac{\rho}{2(1-\rho/\beta)} \norm{\xbf-\hat\xbf}^2.
\end{aligned}
\]
Combining the above inequalities gives
\begin{equation}
f^{\beta}(\hat{\xbf})-f^{\beta}(\xbf)\ge\beta\inner{\xbf-\ybf_{\xbf}}{\hat{\xbf}-\xbf}-\frac{\rho}{2(1-\rho/\beta)} \norm{\xbf-\hat\xbf}^2.\label{eq:opt-08}
\end{equation}
On the other hand, following from \eqref{eq:opt-06} and using $\rho < \beta$, we have 
\begin{equation}
\begin{aligned} & f^{\beta}(\hat{\xbf})-f^{\beta}(\xbf)-\beta\inner{\xbf-\ybf_{\xbf}}{\hat{\xbf}-\xbf}\\
 & =f(\ybf_{\hat{\xbf}})+\frac{\beta}{2}\norm{\ybf_{\hat{\xbf}}-\hat{\xbf}}^{2}-f(\ybf_{\xbf})-\frac{\beta}{2}\norm{\ybf_{\xbf}-\xbf}^{2}-\beta\inner{\xbf-\ybf_{\xbf}}{\hat{\xbf}-\xbf}\\
 & \le\frac{\beta}{2}\norm{\ybf_{\xbf}-\hat{\xbf}}^{2}-\frac{\beta}{2}\norm{\ybf_{\xbf}-\xbf}^{2}-\beta\inner{\xbf-\ybf_{\xbf}}{\hat{\xbf}-\xbf}=\frac{\beta}{2}\norm{\xbf-\hat{\xbf}}^{2}.
\end{aligned}
\label{eq:opt-09}
\end{equation}
The relations \eqref{eq:opt-07} and \eqref{eq:opt-08} together imply that
\[
\lim_{\hat{\xbf}\rightarrow\xbf}\frac{\left|f^{\beta}(\hat{\xbf})-f^{\beta}(\xbf)-\beta\inner{\xbf-\ybf_{\xbf}}{\hat{\xbf}-\xbf}\right|}{\norm{\xbf-\hat{\xbf}}}
\le
\max\left\{\frac{\beta}{2}, \frac{\rho}{2(1-\rho/\beta)} \right\}\lim_{\hat{\xbf}\rightarrow\xbf} \norm{\xbf-\hat\xbf} =0.
\]
Hence $f^{\beta}$ is Fr\'echet differentiable with $\nabla f^{\beta}(\xbf)=\beta(\xbf-\ybf_{\xbf})$. Finally, using the optimality condition yields $\zerobf \in \beta(\ybf_{\xbf} - \xbf) + \partial f(\prox_{f/\beta}(\xbf))$, which gives the desired inclusion.\\

\subsection{Proof of Lemma \ref{lem:smoothness}}
 \textsf{Part 1)}. The proof for the  $L$-Lipschitz smooth case is standard; we include a proof for completeness.
        Fix $\xbf,\ybf\in\Xcal$ and define \( h(t) = g(\ybf + t(\xbf - \ybf)) \). By the fundamental theorem of calculus, we have  $g(\xbf)-g(\ybf)=\int_0^1 h'(t) \mathrm{d}t=\int_0^1 \langle \nabla g(\ybf+t(\xbf-\ybf)), \xbf-\ybf \rangle \mathrm{d}t$. 
        Subtracting $\langle \nabla g(\ybf), \xbf-\ybf \rangle $, applying Cauchy-Schwarz, and invoking $L$ smoothness yields
        \begin{equation*}
        \begin{aligned}
        g(\xbf)-g(\ybf)-\langle \nabla g(\ybf), \xbf-\ybf \rangle  & = \int_0^1 \langle \nabla g(\ybf+t(\xbf-\ybf))-\nabla g(\ybf), \xbf-\ybf \rangle \mathrm{d} t \\
        & \le \int_0^1 \norm{\nabla g(\ybf+t(\xbf-\ybf))-\nabla g(\ybf)}\norm{\xbf-\ybf} \mathrm{d}t \\
        &  \le  \int_0^1 Lt \norm{\xbf-\ybf} ^2  \mathrm{d}t  = \frac{L}{2} \norm{\xbf-\ybf}^2.
        \end{aligned}
        \end{equation*}
For \textsf{Part 2)}, we first use $\rho$-weak convexity:
    $    g(\xbf)-g(\ybf)
        \;\le\;
        \langle\nabla g(\xbf),\xbf-\ybf\rangle
        +\frac{\rho}{2}\|\xbf-\ybf\|^{2}.$
    Adding and subtracting $\langle \nabla g(\ybf), \xbf - \ybf \rangle$ and applying Cauchy-Schwarz inequality together with~\eqref{eq:gen-smooth} gives
    \begin{equation}
        \begin{aligned}
            g(\xbf)-g(\ybf) & \leq \inner{\nabla g(\ybf)}{\xbf-\ybf}  + \inner{\nabla g(\xbf)-\nabla g(\ybf)}{\xbf-\ybf} + \frac{\rho}{2}\norm{\xbf-\ybf}^2 \\
            & \le \inner{\nabla g(\ybf)}{\xbf-\ybf}  + \norm{\nabla g(\xbf)-\nabla g(\ybf)}\cdot \norm {\xbf-\ybf} + \frac{\rho}{2}\norm{\xbf-\ybf}^2 \\
            & \le \inner{\nabla g(\ybf)}{\xbf-\ybf}  + \frac{\rho + 2\Lcal(\ybf,\xbf)}{2} \norm{\xbf-\ybf}^2.
        \end{aligned}
    \end{equation}

\subsection{Proof of Proposition \ref{prop:tightness}}

Let us consider the function 
$g(x)\coloneqq\frac{1}{p} x^p-\frac{\rho}{2} x^2,$
where $p>0$ is an even number and $x^p$ is convex. Therefore, $g(x)$ is $\rho$-weakly convex by definition.  The gradient  of $g$ is $g'(x)=x^{p-1}-\rho x$. For any $x, y\in\reals$, 
we have
    \begin{equation}
\begin{aligned}
    |g'(x)-g'(y)| & = |x^{p-1} -y^{p-1} - \rho(x-y)| =\big|\sum_{k=0}^{p-2} x^{p-2-k} y^k-\rho\big|  |x-y|
\end{aligned}
\end{equation}
Thus, $g$ is $\Lcal_g$-generalized smooth with $\Lcal_g(x,y)=\big|\sum_{k=0}^{p-2} x^{p-2-k} y^k-\rho\big|$.

 Let $p$ be any even integer larger than $2/\vep$, set $x=0$, $y\in (\rho^{1/(p-2)},+\infty)$. We denote $0^0=1$ and $\Lcal_g(0,y)=y^{p-2}-\rho
$.
Note that 
\begin{equation}
    g(x) - g(y)-g'(y)(x-y) = \underbrace{\left(\frac{2}{p}  \sum_{k=0}^{p-2} (k+1) x^{p-2-k} y^k-\rho\right)}_{\tilde{L}_g(x,y)} \frac{(x-y)^2}{2}.
\end{equation}
We have $\tilde{L}_g(0, y)=\frac{2(p-1)}{p}y^{p-2} - \rho > 0$. 
It follows that
\[
\begin{aligned}
         \tilde{L}_g(0,y)- (1-\vep)\sbra{\rho+2\Lcal_g(0,y)} 
        & = \tfrac{2(p-1)}{p}y^{p-2} - \rho - (1-\vep)\sbra{2y^{p-2}-\rho}\nonumber \\
        & =2(\vep-\tfrac{1}{p})  y^{p-2} - \vep \rho  \nonumber \\
        & \ge \vep (y^{p-2}-\rho) \ge 0, 
    \end{aligned}
\]
where the last inequality uses the assumption $p \geq \frac{2}{\varepsilon}$. The relation rearranges to the desired inequality.

\subsection{Proof of Proposition \ref{prop:moreau-smoothing}}

\textsf{Part 1).} 
 In view of \eqref{eq:opt-08} and \eqref{eq:opt-09}, we conclude the quadratic bound~\eqref{eq:quad-bound-2}. The Lipschitz smoothness follow from a more general result~\citep[Lemma~7]{boob2024level}. Here, we use the argument of Corollary 3.4~\citep{hoheisel2010proximal} to give a self-contained proof specified for the Moreau envelope. 
 We first establish some nonexpansiveness properties. Summing up \eqref{eq:opt-05} and \eqref{eq:opt-06}, we obtain 
\begin{equation}
(\beta-\rho)\norm{\ybf_{\xbf}-\ybf_{\hat{\xbf}}}^{2}\le\beta\inner{\xbf-\hat{\xbf}}{\ybf_{\xbf}-\ybf_{\hat{\xbf}}}.\label{eq:expansive}
\end{equation}
Applying Cauchy-Schwarz inequality, we immediately obtain 
\begin{equation}
(\beta-\rho)\norm{\ybf_{\xbf}-\ybf_{\hat{\xbf}}}\le\beta\norm{\xbf-\hat{\xbf}}.\label{eq:expansive-2}
\end{equation}
 Note that
\begin{equation}
\begin{aligned}\norm{\nabla f^{\beta}(\hat{\xbf})-\nabla f^{\beta}(\xbf)}^{2} & =\beta^{2}\norm{\xbf-\hat{\xbf}-\ybf_{\xbf}+\ybf_{\hat{\xbf}}}^{2}\\
 & =\beta^{2}\brbra{\norm{\xbf-\hat{\xbf}}^{2}+\norm{\ybf_{\xbf}-\ybf_{\hat{\xbf}}}^{2}-2\inner{\xbf-\hat{\xbf}}{\ybf_{\xbf}-\ybf_{\hat{\xbf}}}}.
\end{aligned}
\label{eq:opt-11}
\end{equation}
When $\beta\ge2\rho$, combining \eqref{eq:expansive} and \eqref{eq:opt-11}, gives
\begin{equation}
\begin{aligned}\norm{\nabla f^{\beta}(\hat{\xbf})-\nabla f^{\beta}(\xbf)}^{2} & \le\beta^{2}\big({\norm{\xbf-\hat{\xbf}}^{2}+\rbra{2\rho/\beta-1}\norm{\ybf_{\xbf}-\ybf_{\hat{\xbf}}}^{2}\big)}%
\end{aligned}
\end{equation}
and we have $\norm{\nabla f^{\beta}(\hat{\xbf})-\nabla f^{\beta}(\xbf)}^{2}\le\beta^{2}\norm{\xbf-\hat{\xbf}}^{2}$.
Hence $\nabla f^{\beta}(\xbf)$ is $\beta$-Lipschitz continuous.
When $\beta \in (\rho, 2\rho)$, \eqref{eq:expansive-2} implies
\begin{equation}
\begin{aligned}\norm{\nabla f^{\beta}(\hat{\xbf})-\nabla f^{\beta}(\xbf)}^{2} & \le\beta^{2}\Brbra{\norm{\xbf-\hat{\xbf}}^{2}+\brbra{\frac{1}{1-\rho/\beta}-2}\inner{\xbf-\hat{\xbf}}{\ybf_{\xbf}-\ybf_{\hat{\xbf}}}}\\
 & \le\beta^{2}\brbra{\norm{\xbf-\hat{\xbf}}^{2}+\brbra{\frac{1}{1-\rho/\beta}-2}\norm{\xbf-\hat{\xbf}}\norm{\ybf_{\xbf}-\ybf_{\hat{\xbf}}}}\\
 & \le\beta^{2}\brbra{\norm{\xbf-\hat{\xbf}}^{2}+\brbra{\frac{1}{1-\rho/\beta}-2}\frac{1}{1-\rho/\beta}\norm{\xbf-\hat{\xbf}}^{2}}\\
 & =\brbra{\frac{\rho}{1-\rho/\beta}}^{2}\norm{\xbf-\hat{\xbf}}^{2},
\end{aligned}
\label{eq:opt-12}
\end{equation}
and this completes the proof.

\textsf{For Part 2)}, placing $\zbf=\xbf$ in \eqref{eq:opt-01} gives 
$f(\xbf)-f^{\beta}(\xbf)\ge\frac{\beta-\rho}{2}\norm{\xbf-\ybf_{\xbf}}^{2}=\frac{(1-\rho/\beta)}{2\beta}\norm{\nabla f^{\beta}(\xbf)}^{2}.$
Let $\vbf\in\partial f(\xbf)$ denote a minimum-norm subgradient. 
Using the definition of $f^{\beta}(\xbf)$,  we have 
\begin{equation}
\begin{aligned}f^{\beta}(\xbf) & =\min_{\ybf}\bcbra{f(\ybf)+\frac{\beta}{2}\norm{\xbf-\ybf}^{2}}\\
 & \ge\min_{\ybf}\bcbra{f(\xbf)+\inner{\vbf}{\ybf-\xbf}+\frac{\beta-\rho}{2}\norm{\xbf-\ybf}^{2}}\\
 & \ge f(\xbf)+\min_{\ybf}\bcbra{-\norm{\vbf}\cdot\norm{\ybf-\xbf}+\frac{\beta-\rho}{2}\norm{\xbf-\ybf}^{2}}\\
 & =f(\xbf)-\frac{\norm{\vbf}^{2}}{2(\beta-\rho)} = f(\xbf) - \frac{\|\partial f(\xbf)\|^2}{2(\beta - \rho)}.
\end{aligned}
\end{equation}

\subsection{Proof of Theorem \ref{thm:rate-smooth-sgd}}
For convenience, define $\deltabf^{k}\coloneqq\gbf^{k}-\nabla f_{\eta}(\xbf^{k})$ and $G^{k}\coloneqq\frac{1}{\gamma}(\xbf^{k-1}-\xbf^{k})$. Let $\hat{\xbf}^{k}=\prox_{\gamma r}(\xbf^{k-1}-\gamma\nabla f_{\eta}(\xbf^{k-1}))$
and $\hat{G}^{k}\coloneqq\Gcal_{\gamma}(\xbf^{k-1})$. Using the optimality condition on $\hat{\xbf}$ and $\xbf^k$ in their proximal updates, respectively, we have 
\begin{equation}
\inner{\nabla f_{\eta}(\xbf^{k-1})}{\hat{\xbf}^{k}-\xbf}+r(\hat{\xbf}^{k})-r(\xbf)\le\frac{1}{2\gamma}(\norm{\xbf-\xbf^{k-1}}^{2}-\norm{\hat{\xbf}^{k}-\xbf^{k-1}}^{2}-\norm{\hat{\xbf}^{k}-\xbf}^{2}),\label{eq:Feta-opt}
\end{equation}
and
\begin{equation}
\inner{\gbf^{k-1}}{\xbf^{k}-\xbf}+r(\xbf^{k})-r(\xbf)\le\frac{1}{2\gamma}(\norm{\xbf-\xbf^{k-1}}^{2}-\norm{\xbf^{k}-\xbf^{k-1}}^{2}-\norm{\xbf^{k}-\xbf}^{2}).\label{eq:gk-opt}
\end{equation}
Placing $\xbf=\xbf^{k-1}$ in \eqref{eq:gk-opt}, we have
\[
\inner{\gbf^{k-1}}{\xbf^{k}-\xbf^{k-1}}+r(\xbf^{k})-r(\xbf^{k-1})\le-\frac{1}{\gamma}\norm{\xbf^{k}-\xbf^{k-1}}^{2}.
\]
Applying the $L_\eta$-Lipschitz smoothness of $f_\eta$, we deduce that
\[
\begin{aligned}
f_\eta(\xbf^{k}) & \le f_\eta(\xbf^{k-1})+\inner{\nabla f_{\eta}(\xbf^{k-1})}{\xbf^{k}-\xbf^{k-1}}+\frac{L_{\eta}}{2}\norm{\xbf^{k}-\xbf^{k-1}}^{2}\\
 & =f_\eta(\xbf^{k-1})+\inner{\gbf^{k-1}-\deltabf^{k-1}}{\xbf^{k}-\xbf^{k-1}}+\frac{L_{\eta}}{2}\norm{\xbf^{k}-\xbf^{k-1}}^{2}\\
 & \le f_\eta(\xbf^{k-1})+\inner{\gbf^{k-1}}{\xbf^{k}-\xbf^{k-1}}+\frac{1}{2L_{\eta}}\norm{\deltabf^{k-1}}^{2}+L_{\eta}\norm{\xbf^{k}-\xbf^{k-1}}^{2},
\end{aligned}
\]
where the last inequality uses $\langle \xbf, \ybf\rangle \leq \frac{\alpha}{2}\|\xbf\|^2 + \frac{1}{2\alpha} \|\ybf\|^2$. Combining the above two inequalities and using the definition of $\deltabf^{k-1}$, we obtain 
\[\begin{aligned}
\phi_{\eta}(\xbf^{k}) & \le\phi_{\eta}(\xbf^{k-1})+\frac{1}{2L_{\eta}}\norm{\deltabf^{k-1}}^{2}-\frac{1-L_{\eta}\gamma}{\gamma}\norm{\xbf^{k}-\xbf^{k-1}}^{2}\\
 & \le\phi_{\eta}(\xbf^{k-1})+\frac{1}{2L_{\eta}}\norm{\deltabf^{k-1}}^{2}-\frac{1-L_{\eta}\gamma}{\gamma}\norm{\xbf^{k}-\xbf^{k-1}}^{2}.
\end{aligned}
\]
Using the definition of $G^{k}$, we have
\[
\norm{G^{k}}^{2}\le\frac{1}{\gamma-L_{\eta}\gamma^{2}}[\phi_{\eta}(\xbf^{k-1})-\phi_{\eta}(\xbf^{k})]+\frac{1}{(2\gamma-2L_{\eta}\gamma^{2})L_{\eta}}\norm{\deltabf^{k-1}}^{2}.\label{eq:descent-Gk}
\]

Placing $\xbf=\xbf^{k}$ and $\xbf=\hat{\xbf}^{k}$ in \eqref{eq:Feta-opt}
and \eqref{eq:gk-opt}, respectively, and sum up the resulting inequalities,
we have
\[
\inner{\deltabf^{k-1}}{\xbf^{k}-\hat{\xbf}^{k}}\le-\frac{1}{\gamma}\norm{\hat{\xbf}^{k}-\xbf^{k}}^{2}.
\]
Applying Cauchy Schwartz inequality $-\inner{\deltabf^{k-1}}{\xbf^{k}-\hat{\xbf}^{k}}\ge-\norm{\deltabf^{k-1}}\cdot\norm{\hat{\xbf}^{k}-\xbf^{k}}$.
Combining the above results, we have 
\[
\norm{\hat{\xbf}^{k}-\xbf^{k}}\le\gamma\norm{\deltabf^{k-1}}.
\]
Thus we have 
\begin{equation}
\frac{1}{2}\norm{\hat{G}^{k}}^{2}\le\norm{G^{k}}^{2}+\norm{\hat{G}^{k}-G^{k}}^{2}\le\norm{G^{k}}^{2}+\frac{1}{\gamma^{2}}\norm{\hat{\xbf}^{k}-\xbf^{k}}^{2}\le\norm{G^{k}}^{2}+\norm{\deltabf^{k-1}}^{2}.\label{eq:Gksquare-bound}
\end{equation}
Combining \eqref{eq:descent-Gk} and \eqref{eq:Gksquare-bound}, we
obtain 
\[
\norm{\hat{G}^{k}}^{2}\le\frac{2}{\gamma-L_{\eta}\gamma^{2}}[\phi_{\eta}(\xbf^{k-1})-\phi_{\eta}(\xbf^{k})]+\frac{1}{(\gamma-L_{\eta}\gamma^{2})L_{\eta}}\norm{\deltabf^{k-1}}^{2}+2\norm{\deltabf^{k-1}}^{2}.
\]
Setting $\gamma=1/(2L_{\eta})$, the relation simplifies to 
\[
\norm{\hat{G}^{k}}^{2}\le8L_{\eta}[\phi_{\eta}(\xbf^{k-1})-\phi_{\eta}(\xbf^{k})]+6\norm{\deltabf^{k-1}}^{2}.
\]
Summing over $k=1,2,\ldots, K$, taking expectation with respect to all sources of randomness, and noting that $\Expe[\norm{\deltabf^{k}}^{2}] \le \frac{\sigma^{2}}{m}$, we obtain
\[
\begin{aligned}
\frac{1}{K}\sum_{k=1}^{K} \Expe[\norm{\hat{G}^{k}}^{2}] &\le \frac{8L_{\eta} \Expe[\phi_{\eta}(\xbf^{0}) - \phi_{\eta}(\xbf^{K})]}{K} + \frac{6\sigma^{2}}{m} \\
&\le \frac{8L_{\eta} \Expe[\phi(\xbf^{0}) - \phi(\xbf^{K}) + R\eta]}{K} + \frac{6\sigma^{2}}{m} \\
&\le \frac{8L_{\eta} (\Delta + R\eta)}{K} + \frac{6\sigma^{2}}{m},
\end{aligned}
\]
where the second inequality uses the fact that $\phi_\eta(\xbf) \le \phi(\xbf)$, as established in the definition of the smooth approximation, and the final inequality follows from the bound $\Delta \ge \phi(\xbf^{0}) - \min_{\xbf} \phi(\xbf)$.

\subsection{Proof of Theorem \ref{thm:moreau-analysis}}
The key idea follows from the Moreau envelope-based analysis in \citet{davis2019stochastic}.
We set $\hat{\xbf}^{k}=\prox_{\phi_{\eta}/\hat{\rho}}(\xbf^{k-1})$.
By the optimality condition of the proximal mapping, we have
\begin{equation}
\phi_{\eta}(\hat{\xbf}^{k})\le\phi_{\eta}(\xbf)+\frac{\hat{\rho}}{2}\norm{\xbf-\xbf^{k-1}}^{2}-\frac{\hat{\rho}}{2}\norm{\hat{\xbf}^{k}-\xbf^{k-1}}^{2}-\frac{\hat{\rho}-\bar{\rho}}{2}\norm{\xbf-\hat{\xbf}^{k}}^{2}.\label{eq:phieta-opt}
\end{equation}
Placing $\xbf=\xbf^{k}$ in \eqref{eq:phieta-opt}, we obtain that
\[
\phi_{\eta}(\hat{\xbf}^{k})\le\phi_{\eta}(\xbf^{k})+\frac{\hat{\rho}}{2}\norm{\xbf^{k}-\xbf^{k-1}}^{2}-\frac{\hat{\rho}}{2}\norm{\hat{\xbf}^{k}-\xbf^{k-1}}^{2}-\frac{\hat{\rho}-\bar{\rho}}{2}\norm{\xbf^{k}-\hat{\xbf}^{k}}^{2}.
\]
Placing $\xbf=\hat{\xbf}^{k}$ in \eqref{eq:gk-opt} we have
\[
\inner{\gbf^{k-1}}{\xbf^{k}-\hat{\xbf}^{k}}+r(\xbf^{k})-r(\hat{\xbf}^{k})\le\frac{1}{2\gamma}(\norm{\hat{\xbf}^{k}-\xbf^{k-1}}^{2}-\norm{\xbf^{k}-\xbf^{k-1}}^{2}-\norm{\xbf^{k}-\hat{\xbf}^{k}}^{2}).
\]
Combining these two results gives
\begin{equation}
    \begin{aligned}
	& f_{\eta}(\hat{\xbf}^{k})-f_{\eta}(\xbf^{k})+\inner{\gbf^{k-1}}{\xbf^{k}-\hat{\xbf}^{k}}\nonumber 
	\\& \le{}-\frac{\gamma^{-1}-\hat{\rho}}{2}\norm{\xbf^{k}-\xbf^{k-1}}^{2}-\frac{\hat{\rho}-\bar{\rho}+\gamma^{-1}}{2}\norm{\xbf^{k}-\hat{\xbf}^{k}}^{2}+\frac{\gamma^{-1}-\hat{\rho}}{2}\norm{\hat{\xbf}^{k}-\xbf^{k-1}}^{2}. \label{eqn:thm-moreau-analyai-1}
    \end{aligned}
\end{equation}
We further notice that 
\begin{equation*}
\begin{aligned}
 & f_{\eta}(\hat{\xbf}^{k})-f_{\eta}(\xbf^{k})+\inner{\gbf^{k-1}}{\xbf^{k}-\hat{\xbf}^{k}}\\
 & =f_{\eta}(\hat{\xbf}^{k})-f_{\eta}(\xbf^{k-1})+\inner{\nabla f_{\eta}(\xbf^{k-1})}{\xbf^{k-1}-\hat{\xbf}^{k}}\\
 & \quad+f_{\eta}(\xbf^{k-1})-f_{\eta}(\xbf^{k})+\inner{\nabla f_{\eta}(\xbf^{k-1})}{\xbf^{k}-\xbf^{k-1}}+\inner{\deltabf^{k-1}}{\xbf^{k}-\hat{\xbf}^{k}}\\
 & \ge-\frac{\bar{\rho}}{2}\norm{\xbf^{k-1}-\hat{\xbf}^{k}}^{2}-\frac{L_{\eta}}{2}\norm{\xbf^{k}-\xbf^{k-1}}^{2}+\inner{\deltabf^{k-1}}{\xbf^{k}-\hat{\xbf}^{k}}, 
\end{aligned}
\end{equation*}
where the last inequality is due to the $\bar\rho$-weak convexity and $L_\eta$-smoothness of $f_\eta(\xbf)$. Combining the above two inequalities, we have
\begin{equation} \label{eqn:thm-moreau-analyai-3}
    \begin{aligned}
	0 & \le-\inner{\deltabf^{k-1}}{\xbf^{k}-\hat{\xbf}^{k}}-\frac{\gamma^{-1}-\hat{\rho}-L_{\eta}}{2}\norm{\xbf^{k}-\xbf^{k-1}}^{2} \\ 
    & \quad -\frac{\hat{\rho}-\bar\rho+\gamma^{-1}}{2}\norm{\xbf^{k}-\hat{\xbf}^{k}}^{2}+\frac{\gamma^{-1}+\bar{\rho}-\hat{\rho}}{2}\norm{\hat{\xbf}^{k}-\xbf^{k-1}}^{2}.
    \end{aligned}
\end{equation}
Notice that 
\[\begin{aligned}
 & -\inner{\deltabf^{k-1}}{\xbf^{k}-\hat{\xbf}^{k}}-\frac{\gamma^{-1}-\hat{\rho}-L_{\eta}}{2}\norm{\xbf^{k}-\xbf^{k-1}}^{2}\\
 & =-\inner{\deltabf^{k-1}}{\xbf^{k}-\xbf^{k-1}}  -\frac{\gamma^{-1}-\hat{\rho}-L_{\eta}}{2}\norm{\xbf^{k}-\xbf^{k-1}}^{2}- \inner{\deltabf^{k-1}}{\xbf^{k-1}-\hat{\xbf}^{k}}\\
 & \le\norm{\deltabf^{k-1}}\norm{\xbf^{k}-\xbf^{k-1}}-\frac{\gamma^{-1}-\hat{\rho}-L_{\eta}}{2}\norm{\xbf^{k}-\xbf^{k-1}}^{2}-\inner{\deltabf^{k-1}}{\xbf^{k-1}-\hat{\xbf}^{k}}\\
 & \le\frac{\norm{\deltabf^{k-1}}^{2}}{2(\gamma^{-1}-\hat{\rho}-L_{\eta})}-\inner{\deltabf^{k-1}}{\xbf^{k-1}-\hat{\xbf}^{k}}.
\end{aligned}
\]
Plugging the bound back into \eqref{eqn:thm-moreau-analyai-3}, we have
\[\begin{aligned}
0 & \le-\frac{\hat{\rho}-\rho+\gamma^{-1}}{2}\norm{\xbf^{k}-\hat{\xbf}^{k}}^{2}+\frac{\gamma^{-1}+\bar{\rho}-\hat{\rho}}{2}\norm{\xbf^{k-1}-\hat{\xbf}^{k}}^{2}+\frac{\norm{\deltabf^{k-1}}^{2}}{2(\gamma^{-1}-\hat{\rho}-L_{\eta})}-\inner{\deltabf^{k-1}}{\xbf^{k-1}-\hat{\xbf}^{k}} .
\end{aligned}
\]
Rearranging this inequality and dividing both sides by $\frac{\hat{\rho}-\bar{\rho}+\gamma^{-1}}{2}$,
we have
\[
\norm{\xbf^{k}-\hat{\xbf}^{k}}^{2}\le\norm{\xbf^{k-1}-\hat{\xbf}^{k}}^{2}+\frac{1}{\hat{\rho}-\bar{\rho}+\gamma^{-1}}\Big[\frac{\norm{\deltabf^{k-1}}^{2}}{(\gamma^{-1}-\hat{\rho}-L_{\eta})}-2\inner{\deltabf^{k-1}}{\xbf^{k-1}-\hat{\xbf}^{k}}-2(\hat{\rho}-\bar{\rho})\norm{\xbf^{k-1}-\hat{\xbf}^{k}}^{2}\Big].
\]
Following the definition of the Moreau envelope, we have
\[\begin{aligned}
\phi_\eta^{\hat\rho}(\xbf^{k})
 & =\min_{\xbf}\, \big\{\phi_{\eta}(\xbf)+\frac{\hat{\rho}}{2}\norm{\xbf-\xbf^{k}}^{2} \big\}\\
 & \le \phi_{\eta}(\hat{\xbf}^{k})+\frac{\hat{\rho}}{2}\norm{\hat{\xbf}^{k}-\xbf^{k}}^{2}\\
 & \le\phi_{\eta}(\hat{\xbf}^{k})+\frac{\hat{\rho}}{2}\norm{\hat{\xbf}^{k}-\xbf^{k-1}}^{2}+\frac{\hat{\rho}}{2(\hat{\rho}-\bar{\rho}+\gamma^{-1})}\Big[\frac{\norm{\deltabf^{k-1}}^{2}}{(\gamma^{-1}-\hat{\rho}-L_{\eta})}- \\
 & \qquad 2\inner{\deltabf^{k-1}}{\xbf^{k-1}-\hat{\xbf}^{k}}-2(\hat{\rho}-\bar{\rho})\norm{\hat{\xbf}^{k}-\xbf^{k-1}}^{2}\Big]\\
 & =\phi_\eta^{\hat\rho}(\xbf^{k-1})+\frac{\hat{\rho}}{2(\hat{\rho}-\bar{\rho}+\gamma^{-1})}\left[\frac{\norm{\deltabf^{k-1}}^{2}}{(\gamma^{-1}-\hat{\rho}-L_{\eta})}-2\inner{\deltabf^{k-1}}{\xbf^{k-1}-\hat{\xbf}^{k}}-2(\hat{\rho}-\bar{\rho})\norm{\hat{\xbf}^{k}-\xbf^{k-1}}^{2}\right].
\end{aligned}
\]
Plugging the value $\xbf^{k-1}-\hat{\xbf}^{k}=\hat{\rho}^{-1}\nabla\phi_\eta^{\hat\rho}(\xbf^{k-1})$
in the above inequality and rearranging, we have
\[
\frac{\hat{\rho}-\bar{\rho}}{\hat{\rho}(\hat{\rho}-\bar{\rho}+\gamma^{-1})}\norm{\nabla\phi_\eta^{\hat\rho}(\xbf^{k-1})}^{2}\le\phi_\eta^{\hat\rho}(\xbf^{k-1})-\phi_\eta^{\hat\rho}(\xbf^{k})+\frac{\hat{\rho}}{2(\hat{\rho}-\bar{\rho}+\gamma^{-1})}\left[\frac{\norm{\deltabf^{k-1}}^{2}}{(\gamma^{-1}-\hat{\rho}-L_{\eta})}-2\inner{\deltabf^{k-1}}{\xbf^{k-1}-\hat{\xbf}^{k}}\right].
\]
Taking expectation on both sides and noticing $\Ebb[\inner{\deltabf^{k-1}}{\xbf^{k-1}-\hat{\xbf}^{k}}]=0$,
and then rescaling, we have 
\[
\Ebb[\norm{\nabla\phi_\eta^{\hat\rho}(\xbf^{k-1})}^{2}] \le \frac{\hat{\rho}(\hat{\rho}-\bar{\rho}+\gamma^{-1})}{\hat{\rho}-\bar{\rho}}\Ebb[\phi_\eta^{\hat\rho}(\xbf^{k-1})-\phi_\eta^{\hat\rho}(\xbf^{k})]+\frac{\hat{\rho}^{2}}{2(\hat{\rho}-\bar{\rho})}\left[\frac{\sigma^{2}}{(\gamma^{-1}-\hat{\rho}-L_{\eta})m}\right].
\]
 Summing up the above relation over $k=0,1,\ldots,K-1$, we have
\[\begin{aligned}
\frac{1}{K}\sum^{K-1}_{k=0}\Ebb[\norm{\nabla\phi_\eta^{\hat\rho}(\xbf^{k})}^{2}] & \le\frac{\hat{\rho}(\hat{\rho}-\bar{\rho}+\gamma^{-1})}{\hat{\rho}-\bar{\rho}}\frac{\Ebb[\phi_\eta^{\hat\rho}(\xbf^{0})-\phi_\eta^{\hat\rho}(\xbf^{K})]}{K}+\frac{\hat{\rho}^{2}}{2(\hat{\rho}-\bar{\rho})}\frac{\sigma^{2}}{(\gamma^{-1}-\hat{\rho}-L_{\eta})m}\\
 & \le\frac{\hat{\rho}}{\hat{\rho}-\bar{\rho}}\left[ (\hat{\rho}-\bar{\rho}+\gamma^{-1})\frac{\Delta+R\eta}{K}+\frac{\hat{\rho}\sigma^{2}}{2(\gamma^{-1}-\hat{\rho}-L_{\eta})m}\right], 
\end{aligned}
\]
where the last inequality uses
\[\begin{aligned}
\phi_\eta^{\hat\rho}(\xbf^{0})-\phi_\eta^{\hat\rho}(\xbf^{K}) & \le\phi_\eta^{\hat\rho}(\xbf^{0})-\min_{\xbf}\phi_\eta^{\hat\rho}(\xbf)\\
 & \le\phi_{\eta}(\xbf^{0})-\min_{\xbf}\phi_{\eta}(\xbf)\\
 & \le\phi(\xbf^{0})-\min_{\xbf}\left[\phi(\xbf)-R\eta\right]\\
 & =\phi(\xbf^{0})-\min_{\xbf}\phi(\xbf)+R\eta\\
 & \le\Delta+R\eta.
\end{aligned}
\]
Suppose we take $\gamma=(c\sqrt{K}+\hat{\rho}+L_{\eta})^{-1}$ where
$c=\sqrt{\frac{\hat{\rho}}{2(\Delta+R\eta)m}}\sigma$, then we have
\[\begin{aligned}
\frac{1}{K}\sum^{K-1}_{k=0}\Ebb[\norm{\nabla\phi_\eta^{\hat\rho}(\xbf^{k})}^{2}] & \le\frac{\hat{\rho}}{\hat{\rho}-\bar{\rho}}\left\{ \frac{(2\hat{\rho}-\bar{\rho}+L_{\eta})(\Delta+R\eta)}{K}+\frac{c(\Delta+R\eta)}{\sqrt{K}}+\frac{\hat{\rho}\sigma^{2}}{2cm\sqrt{K}}\right\} \\
 & =\frac{\hat{\rho}}{\hat{\rho}-\bar{\rho}}\left\{ \frac{(2\hat{\rho}-\bar{\rho}+L_{\eta})(\Delta+R\eta)}{K}+\sqrt{\frac{2\hat{\rho}(\Delta+R\eta)}{mK}}\sigma\right\},
\end{aligned}
\]
and this completes the proof.

\subsection{Proof of Proposition~\ref{prop:convergence-bound}}
We begin by applying \Cref{lem:three-point}, which yields
\begin{equation}\label{eq:opt-2}
    \begin{aligned}
     & \langle \nabla g(\xbf^{t}), \zbf^{t}-\xbf^\star \rangle + \pi(\zbf^{t}) - \pi(\xbf^\star) \\
     &\le \frac{\gamma_{t}}{2}\norm{\xbf^\star-\zbf^{t-1}}^{2} - \frac{\gamma_{t}+\mu}{2}\norm{\xbf^\star-\zbf^{t}}^{2} - \frac{\gamma_{t}}{2}\norm{\zbf^{t}-\zbf^{t-1}}^{2}.
    \end{aligned}
\end{equation}
For convenience, let us define $\ell_{g}(\ybf, \xbf) \coloneqq g(\ybf) - g(\xbf) - \nabla g(\xbf)^{\top}(\ybf - \xbf)$ and $\hat{\zbf}^{t-1} \coloneqq (1-\beta_{t})\ybf^{t-1} + \beta_{t}\zbf^{t-1}$.
Utilizing the smoothness of $g(\xbf)$ and the strong convexity of $\pi(\xbf)$, we establish
\[\begin{aligned}
\psi(\ybf^{t}) & \leq g(\xbf^{t}) + \nabla g(\xbf^{t})^{\top}(\ybf^{t} - \xbf^{t}) + \pi(\ybf^{t}) + \frac{\hat{L}_{t}}{2}\norm{\ybf^{t} - \xbf^{t}}^{2}\\
 & \leq (1-\alpha_{t})\ell_{g}(\ybf^{t-1}, \xbf^{t}) + (1-\alpha_{t})\pi(\ybf^{t-1}) + \alpha_{t}\ell_{g}(\zbf^{t}, \xbf^{t}) + \alpha_{t}\pi(\zbf^{t})\\
 & \qquad + \frac{\hat{L}_{t}}{2}\norm{\ybf^{t} - \xbf^{t}}^{2} - \frac{\alpha_{t}(1-\alpha_{t})\mu}{2}\norm{\ybf^{t-1} - \zbf^{t}}^{2}\\
 & \leq (1-\alpha_{t})\psi(\ybf^{t-1}) + \alpha_{t}\ell_{g}(\zbf^{t}, \xbf^{t}) + \alpha_{t}\pi(\zbf^{t}) + \frac{\hat{L}_{t}}{2}\norm{\ybf^{t} - \xbf^{t}}^{2} - \frac{\alpha_{t}(1-\alpha_{t})\mu}{2}\norm{\ybf^{t-1} - \zbf^{t}}^{2},
\end{aligned}
\]
where the last inequality exploits the convexity of $\psi$.

Applying Jensen's inequality, we obtain
\[
\norm{\ybf^{t} - \xbf^{t}}^{2} = \alpha_{t}^{2}\norm{\zbf^{t} - \hat{\zbf}^{t-1}}^{2} \leq (1-\beta_{t})\alpha_{t}^{2}\norm{\zbf^{t} - \ybf^{t-1}}^{2} + \beta_{t}\alpha_{t}^{2}\norm{\zbf^{t} - \zbf^{t-1}}^{2}.
\]
Consequently,
\[\begin{aligned}
\psi(\ybf^{t}) & \leq (1-\alpha_{t})\psi(\ybf^{t-1}) + \alpha_{t}\ell_{g}(\zbf^{t}, \xbf^{t}) + \alpha_{t}\pi(\zbf^{t}) + \frac{\hat{L}_{t}\beta_{t}\alpha_{t}^{2}}{2}\norm{\zbf^{t} - \zbf^{t-1}}^{2}\\
 & \qquad + \frac{\hat{L}_{t}(1-\beta_{t})\alpha_{t}^{2} - \alpha_{t}(1-\alpha_{t})\mu}{2}\norm{\ybf^{t-1} - \zbf^{t}}^{2}\\
 & \leq (1-\alpha_{t})\psi(\ybf^{t-1}) + \alpha_{t}\ell_{g}(\xbf^{\star}, \xbf^{t}) + \alpha_{t}\pi(\xbf^{\star}) + \frac{\hat{L}_{t}\beta_{t}\alpha_{t}^{2} - \gamma_{t}\alpha_{t}}{2}\norm{\zbf^{t} - \zbf^{t-1}}^{2}\\
 & \qquad + \frac{\alpha_{t}\gamma_{t}}{2}\norm{\xbf^{\star} - \zbf^{t-1}}^{2} - \frac{\alpha_{t}(\gamma_{t} + \mu)}{2}\norm{\xbf^{\star} - \zbf^{t}}^{2}\\
 & \qquad + \frac{\hat{L}_{t}(1-\beta_{t})\alpha_{t}^{2} - \alpha_{t}(1-\alpha_{t})\mu}{2}\norm{\ybf^{t-1} - \zbf^{t}}^{2} \\
 & \leq (1-\alpha_{t})\psi(\ybf^{t-1}) + \alpha_{t}\psi(\xbf^{\star}) + \frac{\alpha_{t}\gamma_{t}}{2}\norm{\xbf^{\star} - \zbf^{t-1}}^{2} - \frac{\alpha_{t}(\gamma_{t} + \mu)}{2}\norm{\xbf^{\star} - \zbf^{t}}^{2},
\end{aligned}
\]
where the second inequality applies \eqref{eq:opt-2}, and the final step follows from the convexity of $\psi$ together with (\ref{eq:step-1}) and (\ref{eq:step-2}).

Rearranging and multiplying the preceding inequality by $\Gamma_{t}$, we have
\[
\Gamma_t[\psi(\ybf^t)-\psi(\xbf^\star)] \le \Gamma_{t}(1-\alpha_t) [\psi(\ybf^{t-1})-\psi(\xbf^\star)] + \frac{\Gamma_t\alpha_{t}\gamma_{t}}{2}\norm{\xbf^{\star} - \zbf^{t-1}}^{2} - \frac{\Gamma_t\alpha_{t}(\gamma_{t} + \mu)}{2}\norm{\xbf^{\star} - \zbf^{t}}^{2}.
\]
Summing over $t$, and invoking (\ref{eq:step-3}), leads directly to the assertion in~\eqref{eq:rec-bound-1}.

\subsection{Proof of Lemma \ref{lem:line-search-valid}}

    For notational simplicity, we omit the iteration index and denote by $\gamma$, $\hat{L}$, and $\bar{L}$ the respective line search parameters at the $t$-th iteration. By the optimality condition, we have
    \[
    \langle \nabla g(\xbf^{t}), \zbf^{t}-\xbf^{\star} \rangle + \pi(\zbf^{t}) - \pi(\xbf^{\star}) \leq \frac{\gamma_{t}}{2}\norm{\xbf^{\star}-\zbf^{t-1}}^{2} - \frac{\gamma+\mu}{2}\norm{\xbf^{\star}-\zbf^{t}}^{2}.
    \]
    Moreover, since $\xbf^{\star} \in \argmin_{\xbf} \langle \nabla g(\xbf^{\star}), \xbf \rangle + \pi(\xbf)$, it follows that
    \[
    \langle \nabla g(\xbf^{\star}), \xbf^{\star}-\zbf^{t} \rangle + \pi(\xbf^{\star})-\pi(\zbf^{t}) \leq 0.
    \]
    By summing the preceding two inequalities, we obtain
    \begin{equation} \label{eq:recursive-zt-dist}
        \begin{aligned}
    \frac{\gamma+\mu}{2}\norm{\xbf^\star-\zbf^{t}}^{2} & \le\frac{\gamma}{2}\norm{\xbf^\star-\zbf^{t-1}}^{2}-\inner{\nabla g(\xbf^{t})-\nabla g(\xbf^{\star})}{\zbf^{t}-\xbf^{\star}}\nonumber \\
     & \le\frac{\gamma}{2}\norm{\xbf^\star-\zbf^{t-1}}^{2}+\norm{\nabla g(\xbf^{t})-\nabla g(\xbf^{\star})}\norm{\zbf^{t}-\xbf^{\star}}\nonumber \\
     & \le\frac{\gamma}{2}\norm{\xbf^\star-\zbf^{t-1}}^{2}+\bar{L}D\norm{\zbf^{t}-\xbf^{\star}},
    \end{aligned}
    \end{equation}
    where the last inequality follows from the Lipschitz continuity of $\nabla g$ and the assumption $\norm{\xbf^{\star}-\xbf^{t}}\le D$. Define $D_{s}=\norm{\xbf^{\star}-\zbf^{s}}$, and we denote $\bar{L}=\sup\{L(\xbf,\ybf): \xbf,\ybf\in\Bcal_D(\xbf^\star)\}$. Then, applying~\eqref{eq:recursive-zt-dist}, we have
    \[
    D_{t}\le\frac{\bar{L}D}{\gamma+\mu}+\sqrt{\left(\frac{\bar{L}D}{\gamma+\mu}\right)^{2}+\frac{\gamma}{\gamma+\mu}D_{t-1}^{2}}
    \le\frac{2\bar{L}D}{\gamma+\mu}+D_{t-1}\le\left(\frac{2\bar{L}}{\gamma+\mu}+1\right)D.
    \]
    Since $\ybf^{t}$ is constructed as a convex combination of $\ybf^{t-1}$ and $\zbf^{t}$, we have
    \begin{equation} \label{eq:yt-bound}
        \begin{aligned}
    \norm{\ybf^{t}-\xbf^{\star}} & \le (1-\alpha_{t})\norm{\ybf^{t-1}-\xbf^{\star}}+\alpha_{t}\norm{\zbf^{t}-\xbf^{\star}}\nonumber \\
     & \le(1-\alpha_{t})\norm{\ybf^{t-1}-\xbf^{\star}}+\alpha_{t}\left(\frac{2\bar{L}}{\gamma+\mu}+1\right)D\nonumber \\
     & =\left(1+\frac{2\bar{L}\alpha_{t}}{\gamma+\mu}\right)D\nonumber \\
     & \le\left(1+\frac{2\bar{L}}{\hat{L}}\right)D,
    \end{aligned}
    \end{equation}
    where the last inequality applies the relation $\gamma = (\hat{L}+\mu)\alpha_{t}-\mu$.
    
    We now establish that the line search procedure must terminate after a finite number of iterations, proceeding by contradiction. Define
    \[
    \tilde{L} = \sup \left\{ 2\Lcal(\xbf,\ybf) : \xbf, \ybf \in \Bcal_{2D}(\xbf^{\star}) \right\}.
    \]
    We claim that the line search will terminate in at most $k^\star = \max \left\{ \log_{\tau_\mathrm{u}} \frac{2\tilde{L}}{\bar{L}_{t}}, 0 \right\} + 1$ steps. To see this, suppose for the sake of contradiction that the line search does not terminate within $k^\star$ iterations. Then, upon reaching $\hat{L} = \bar{L}_t \tau_\mathrm{u}^{k^\star+1}$, we would have $\hat{L} \ge 2\tilde{L}$. By applying inequality~\eqref{eq:yt-bound}, we observe that
    \[
    \norm{\ybf^t(\hat{L}) - \xbf^{\star}} \le \left(1 + \frac{2\bar{L}}{\hat{L}}\right) D \le \left(1 + \frac{2\bar{L}}{2\tilde{L}}\right) D \le 2D.
    \]
    Therefore, condition~\eqref{eq:ls-descent} must be satisfied, leading to a contradiction. This completes the argument.

\end{document}